\documentclass{amsart}
\usepackage{amsmath}
\usepackage{amscd}
\usepackage{amssymb}
\usepackage{amsthm}
\usepackage{enumerate}
\usepackage{comment}
\usepackage[normalem]{ulem}
\usepackage{mathtools}
\usepackage{hyperref}
\usepackage[usenames,dvipsnames]{xcolor}
\usepackage{stmaryrd}
\usepackage[left=1in,top=1in,right=1in,bottom=1in]{geometry}
\usepackage[T2A]{fontenc}
\usepackage[utf8]{inputenc}
\usepackage{mathrsfs}
\usepackage{mathabx}
\usepackage{appendix}

\numberwithin{equation}{section}
\theoremstyle{plain}
\newtheorem{thm}{Theorem}[section]

\newtheorem{introthm}{Theorem}

\newtheorem{prop}[thm]{Proposition}
\newtheorem{lem}[thm]{Lemma}
\newtheorem{lemma}[thm]{Lemma}
\newtheorem{cor}[thm]{Corollary}

\theoremstyle{definition}
\newtheorem{defn}[thm]{Definition}
\newtheorem{notation}[thm]{Notation}

\theoremstyle{remark}
\newtheorem{remark}[thm]{Remark}

\newcommand{\bC}{{\mathbb C}}

\newcommand{\bF}{{\mathbb F}}

\newcommand{\bH}{{\mathbb H}}

\newcommand{\bN}{{\mathbb N}}

\newcommand{\bR}{{\mathbb R}}

\newcommand{\bZ}{{\mathbb Z}}

\newcommand{\cC}{{\mathcal C}}

\newcommand{\cE}{{\mathcal E}}
\newcommand{\cF}{{\mathcal F}}

\newcommand{\cH}{{\mathcal H}}
\newcommand{\cI}{{\mathcal I}}

\newcommand{\cK}{{\mathcal K}}

\newcommand{\cM}{{\mathcal M}}
\newcommand{\cN}{{\mathcal N}}

\newcommand{\cQ}{{\mathcal Q}}

\newcommand{\cS}{{\mathcal S}}

\newcommand{\cU}{{\mathcal U}}

\newcommand{\cZ}{{\mathcal Z}}

\newcommand{\actson}{{\curvearrowright}}

\newcommand\Ad{\operatorname{Ad}}

\newcommand\Aut{\operatorname{Aut}}

\newcommand\dom{\operatorname{dom}}

\newcommand\Hom{\operatorname{Hom}}

\newcommand\id{\operatorname{id}}

\newcommand\Span{\operatorname{span}}

\newcommand\Tr{\operatorname{Tr}}

\newcommand{\R}{\mathbb{R}}
\newcommand{\vphi}{\varphi}
\newcommand{\si}{\sigma}
\newcommand{\AmEx}{\operatorname{AE}}
\newcommand{\core}{\operatorname{c}}

\makeatletter
\let\@wraptoccontribs\wraptoccontribs
\makeatother

\newcommand{\ip}[1]{\langle #1 \rangle}

\renewcommand{\>}{\rangle}

\newcommand{\cross}{\rotatebox[origin=c]{180}{\textnormal{\tiny\dag}}}

\begin{document}

\title{General solidity phenomena and anticoarse spaces \\for type $\mathrm{III}_1$ factors}

\author{Ben Hayes}
\address{Department of Mathematics, University of Virginia\hfill \url{brh5c@virginia.edu}}

\author{David Jekel}
\address{Department of Mathematical Sciences, University of Copenhagen, Denmark\hfill \url{daj@math.ku.dk}}

\author{Srivatsav Kunnawalkam Elayavalli}
\address{Department of Mathematics, University of California, San Diego\hfill \url{skunnawalkamelayaval@ucsd.edu}}

\author{Brent Nelson}
\address{Department of Mathematics, Michigan State University\hfill \url{brent@math.msu.edu}}

\contrib[with an appendix by]{Stefaan Vaes}
\address{Department of Mathematics, KU~Leuven,  Leuven (Belgium)\hfill \url{stefaan.vaes@kuleuven.be}}

\begin{abstract}
By developing a theory of anticoarse spaces in the purely infinite setting and using  1-bounded entropy techniques along with recent strong convergence results in random matrix theory,  we show that free Araki--Woods factors offer the first examples of type $\mathrm{III}$ factors satisfying vastly general degrees of indecomposability phenomena. Notably this includes strong solidity with respect to any weakening of the normalizer currently in the literature and the Peterson--Thom property. 
\end{abstract}

\maketitle

\section*{Introduction}

A central theme of modern von Neumann algebra research is investigating whether von Neumann algebras can be ``built up'' from amenable pieces via various constructions such as crossed products, tensor products, inductive limits, etc. 
In the finite setting,  free group factors have driven a substantial amount of research in this direction with several major developments showing that they \emph{cannot} be built up from amenable pieces \cite{OzawaSolidActa, OzPopaCartan, Jung2007, Voiculescu1996}. 
Based on their work on $L^{2}$-Betti numbers, with additional motivation from work of Jung \cite{Jung2007}, Ozawa \cite{OzawaSolidActa}, Ozawa--Popa \cite{OzPopaCartan}, and Peterson \cite{PetersonDeriva}, Peterson--Thom conjectured \cite{PetersonThom} that if $Q$ is a diffuse, amenable von Neumann subalgebra of a  free group factor, then $Q$ has a unique maximal amenable extension. This conjecture became known as the \emph{Peterson--Thom conjecture} and it inspired much work on various special cases, including showing specific maximal abelian subalgebras of free group factors are the unique maximal amenable extension of any diffuse subalgebra \cite{WenAOP, 2AuthorsOneCup, ParShiWen}.
It was listed on several problem lists \cite{JesseProblemlist, PopaProblemlist, PopaProblemlist2} and as a question in numerous papers \cite{2AuthorsOneCup, Cyril<3PT, Ozawa2015, PopaWeakInter, ParShiWen, WenAOP}.

A reduction of the Peterson--Thom conjecture to a natural question about strong convergence  in random matrices was given by the first author in \cite{hayespt}. Strong convergence arose from the seminal work \cite{HTExt} (see also \cite{Male, MaleCollins}), and has since bloomed into a rich subject with connections to graph theory \cite{BVH024universalitysharpmatrixconcentration, CollinsBordenave}, representation theory \cite{BCCompactGrp, mageethomasstrongly}, and geometry \cite{MWHyper}
(see also \cite{CollinsICM,MaidaBourbaki} for detailed surveys). Recently, the strong convergence conjecture given in \cite{hayespt} was
affirmatively settled in 
\cite{belinschi2022strong, bordenave2023norm, MdLSstrongasymptoticfreenesshaar, Parraud2024strong} (see also \cite{CGVTVHStrong,  BBvHConc, CGPStrongTen}), thereby resolving  the Peterson--Thom conjecture. Importantly, a strengthening of the Peterson--Thom conjecture was conjectured in \cite[Conjecture 1.12]{Hayes2018}, in terms of anticoarse spaces (there called singular subspaces). Relevant to our discussion, such a strengthening not only implies the Peterson--Thom conjecture but also recovers and generalizes all previous solidity results about free group factors, including to various well-established weakenings of the normalizer such as the wq-normalizer of \cite{GalatanPopa, IPP, PopaCohomologyOE}, the one sided quasi-normalizer of \cite{IzumiLongoPopa, PopaOrthoPairs, PopaMC}, and the weak intertwining space of \cite{PopaMC, PopaWeakInter}. This conjecture was later recast as the \emph{coarseness conjecture} in \cite{PopaWeakInter}. The strong convergence results in \cite{belinschi2022strong, bordenave2023norm} also imply the coarseness conjecture, as discussed in \cite{UselessResolutionOfPT}.

In the finite setting, this anticoarse space may be viewed as the largest $N$-$N$ subbimodule of $L^{2}(M)$ which is disjoint from the coarse bimodule. It also has the advantage of being more directly of an analytic flavor than the a priori algebraic notions of normalizer, quasi-normalizer etc., particularly since it can be studied via noncommutative harmonic analysis techniques. A precursor to this space already appeared in the work of Voiculescu \cite{Voiculescu1996}, who showed that (in different language) the anticoarse space of a diffuse abelian algebra $A$ of a free group factor cannot be equal to all of $L^{2}(M)$. In the finite setting, similar applications to crossed products by free Bogoliubov actions were given by Houdayer-Shlyakhtenko \cite{HoudShlStrongSolid}. 
A systematic development of the anticoarse space and its relation to weak normalizers was given in \cite{Hayes2018}, which showed that free group factors could not be decomposed as iterated von Neumann algebras generated by anticoarse spaces starting from hyperfinite algebras, vastly generalizing Voiculescu's work.

With the primary motivation of generalizing the above degrees of indecomposability to the purely infinite case, in this article we develop a non-tracial ``dynamical anticoarse space'' $L^{2,\textnormal{dyn}}_{\cross}(N\leq M,\psi)$ of $N\leq M$ where $N$ is a subalgebra of $M$ with expectation, and $\psi$ is a state on $M$ whose modular automorphism group leaves $N$ globally invariant. We show that this new anticoarse space contains the one-sided quasi-normalizer (and hence the stable normalizer as discussed in \cite{BHVSS}; see Theorem~\ref{thm:quasi anticoarse stuff}). An appropriate adaptation of the coarseness conjecture in the type $\mathrm{III}$ setting is then that if $P\leq M$ is maximal among amenable subalgebras which are with expectation, then $L^{2,\textnormal{dyn}}_{\cross}(P\leq M,\psi)=L^{2}(P,\psi)$ for any state $\psi$ whose modular automorphism group leaves $P$ globally invariant. We establish this property for type $\mathrm{III}_{1}$ factors whose continuous core is isomorphic to $L(\bF_{\infty})\overline{\otimes} B(\ell^2)$. 
More generally, we show that this adaptation of the coarseness conjecture holds for algebras whose continuous core satisfies a strengthening of the Peterson--Thom property (see Theorem \ref{thm:dynamical_anticoarse_contained_in_Pinsker} and Proposition \ref{prop: apparently free entropy is useless}); namely, the property that all Pinsker algebras are amenable (see Definition \ref{defn: Pinsker algebra}). As we discuss in Section \ref{sec: mining the gold}, the resolution of the Peterson--Thom conjecture shows that all Pinsker algebras for  $L(\bF_{\infty})\overline{\otimes} B(\ell^2)$ are amenable. We also show that having amenable Pinsker algebras automatically passes to subalgebras with expectation. 

As with the consequences of the coarseness conjecture in the finite setting (see \cite{UselessResolutionOfPT}), we can use the dynamical anticoarse space to prove new solidity phenomena for free Araki--Woods factors such as $\Gamma(U)''$ where $U\colon \bR\actson \cH_{\bR}$ embeds into an infinite direct sum of the left regular representation. 
For example, we can give the first example of an \emph{ultrastrongly solid} von Neumann algebra.
Using $M^{\omega}$ for the Ocneau ultrapower (see \cite{OcneauActions, AndoHaagerup}, or Section~\ref{sec:ultrapowers} below), we say that a von Neumann algebra $M$ is \textbf{ultrastrongly solid} if for any cofinal ultrafilter $\omega$ on a directed set, and any $Q\leq M^{\omega}$ which is diffuse, amenable, and with expectation, we have that every $N\leq \cN_{M^{\omega}}(Q)''\cap M$ with expectation in $M$ is amenable. Note this is stronger than both strong solidity and ultrasolidity as defined in \cite{HRSAsy}, see Proposition \ref{prop: defns work}. Solidity of the free Araki--Woods factors was established in \cite{HRApprox}, followed by ultrasolidity \cite{HRSAsy}, and strong solidity in \cite{BHVSS} (see also \cite{HCores, HURig, DImaFullIII, VV07}). This parallels the history of the free group factors, for which solidity was first proven in \cite{OzawaSolidActa}, followed by ultrasolidity in \cite{OzawaComment}, and strong solidity in \cite{OzPopaCartan}.
Other properties and results in this direction include absence of Cartan subalgebras \cite{ChifanSinclair, OzPopaII, PetersonDeriva, Po01a, PopaL2Betti,  PopaVaesHyp, Voiculescu1996} and proper proximality \cite{BIP18, DKEP21, DKE21, dke24, ding2023biexactvonneumannalgebras}.

\begin{introthm}\label{introthm: ultrasolid III}
Let $M$ be $\sigma$-finite von Neumann algebra whose continuous core has amenable Pinsker algebras. Then  $M$ is ultrastrongly solid.
\end{introthm}



Despite the fact that Theorem \ref{introthm: ultrasolid III} (as well as Theorem \ref{thm: the game is over and you should quit} below) only rely on assumptions about the continuous core, there  are significant difficulties in transferring properties from the core back to the original von Neumann algebras. For instance, a key step in our analysis is to show that the dynamical anticoarse space contains the normalizer. Showing that the normalizer is contained in the anticoarse space is rather straightforward in the finite setting. However, the dynamical anticoarse space is ultimately defined via passing to the continuous core, and, as discussed above, the normalizer of a subalgebra need not be contained in the normalizer of its continuous core. In fact, this issue persists even when considering any of the aforementioned weakenings of the normalizer in the continuous core. We instead show the dynamical anticoarse space absorbs the normalizer and even the one-sided quasi-normalizer by proving that the anticoarse space of a semifinite inclusion $N\leq M$ absorbs an analogue of the quasi-normalizer of $pNp \leq pMp$ relative to a subalgebra $pAp$, where $A\leq M$ is diffuse and abelian and $p\in A$ is an arbitrary finite trace projection. Ultimately, this is a delicate limiting argument using Popa's intertwining theory \cite{PopaStrongRigidity}, as well as a further, and non-trivial, reduction to the finite setting (see the proof of Theorem~\ref{thm:quasi anticoarse stuff}). This is one of the main technical novelties of the article.

Because the dynamical anticoarse space contains various weakenings of the normalizer, we are able to go beyond strong solidity to prove new degrees of indecomposability phenomena in the purely infinite setting.  In fact, our results continue to hold in an ultraproduct framework. The particular weakenings of the normalizer covered by our result include the one-sided quasi-normalizer, 
    \[
        q^{1}\cN_{M}(Q)=\left\{x\in M: \textnormal{ there exists } x_{1},\cdots,x_{n}\in M \textnormal{ with } xQ\subseteq \sum_{j}Qx_{j}\right\},
    \]
and the wq-normalizer
    \[
        \cN_{M}^{wq}(Q)=\{u\in \cU(M): \textnormal{ there exists } N\leq uQu^{*}\cap Q \textnormal{ with $N$ diffuse and with expectation in $M$}\}.
    \]
These solidity results can be still further generalized by showing that they continue to hold even after iterating along any ordinal.

\begin{introthm}\label{thm: the game is over and you should quit}
Let $M$ be a $\sigma$-finite von Neumann algebra whose continuous core has amenable Pinsker algebras. Then $M$ enjoys the following generalized solidity properties.
\begin{enumerate}[(i)]
    \item Suppose that $Q\leq M^{\omega}$ is diffuse, amenable, and $\sigma^\psi$-invariant for some faithful normal state $\psi$ on $M^{\omega}$. Let
        \[
            X \subseteq (L^{2,\textnormal{dyn}}_{\cross}(Q\leq M^{\omega},\psi)\cap M^\omega)\cup \cN_{M^{\omega}}^{wq}(Q).
        \]
    If $W^{*}(X,Q)$ is with expectation in $M^{\omega}$, then for any $N\leq W^{*}(X,Q)\cap M$ with expectation in $M$ is amenable. In particular:
    \begin{itemize}
    \item if $Q\leq M$ and if $X\in \{\cN_{M^{\omega}}(Q),q^{1}\cN_{M^{\omega}}(Q),\cN_{M^{\omega}}^{wq}(Q)\}$, then $W^{*}(X,Q)\cap M$ is amenable;
        \item if $Q\leq M$ and if $X\in \{\cN_{M}(Q),q^{1}\cN_{M}(Q),\cN_{M}^{wq}(Q)\}$, then $W^{*}(X,Q)$ is amenable.
        \end{itemize}
    \item More generally, suppose that $Q\leq M^{\omega}$ is diffuse, amenable, and $\sigma^\psi$-invariant for some faithful normal state $\psi$ on $M^{\omega}$. Assume we have subalgebras $Q_{\alpha}\leq M^{\omega}$ for every ordinal $\alpha$ so that:
    \begin{itemize}
        \item $Q_{0}=Q$,
        \item for a successor ordinal $\alpha$ we have $Q_{\alpha}=W^{*}(X_{\alpha},Q_{\alpha-1})$ where
            \[
                X_{\alpha}\subseteq (L^{2,\textnormal{dyn}}_{\cross}(Q\leq M^{\omega},\psi)\cap M^\omega) \cup \cN_{M^{\omega}}^{wq}(Q)
            \]
       is $\sigma^{\psi}$-invariant,
        \item for a limit ordinal $\alpha$, we have $Q_{\alpha}=\overline{\bigcup_{\beta<\alpha}Q_{\beta}}^{SOT}$. 
    \end{itemize}
Then for every $\alpha$, any $N\leq Q_{\alpha}\cap M$ with expectation in $M$, we have that $N$ is amenable.
\end{enumerate}
\end{introthm}


We stress that the state $\psi$ in Theorem \ref{thm: the game is over and you should quit} is \emph{not} required to be an ultrapower state, and can be any state whose modular automorphism group leaves $Q$ globally invariant. One can also expand the choices for $X$ and $X_\alpha$ in the above theorem to the unique smallest von Neumann subalgebra of $M^\omega$ containing the dynamical anticoarse space (see Proposition~\ref{cor:existence_of_vNa_generated_by_dynamical}).
For the case of free group factors, the analogous results were discussed in \cite{UselessResolutionOfPT}.
We remark that as soon as one moves to the ultrapower setting---even in the case that $M$ is a free group factor and $N$ is freely complemented in $M$---establishing that the von Neumann algebra generated by an analogue of the normalizer inside $M^{\omega}$ has amenable intersection with $M$ is non-trivial with \emph{all} known proofs relying on a connection to random matrices. See \cite{FreePinsker} and \cite{JKEupgradedfreeindependencephenomena} for related absorption results in the case of freely complemented finite von Neumann algebras. Thus, in the type III setting, Theorem \ref{thm: the game is over and you should quit} above is new even for the case that $Q$ is freely complemented in $M$.

Our proof of Theorems~\ref{introthm: ultrasolid III} and \ref{thm: the game is over and you should quit} go through a strengthening of the Peterson--Thom property for type $\mathrm{III}$ factors (see Proposition \ref{prop: apparently free entropy is useless}).
In Appendix~\ref{Vaes}, authored by Stefaan Vaes, a direct proof is provided for the fact that a von Neumann algebra $M$ has the Peterson--Thom property as soon as its continuous core has the Peterson--Thom property. 
Note, however, that the generalized solidity phenomena in Theorems~\ref{introthm: ultrasolid III} and \ref{thm: the game is over and you should quit} 
do not follow from $M$ (or its continuous core) having the Peterson--Thom property. 
As a further application of our methods, in Section~\ref{sec:FAWF} we present new non-embedding properties for free Araki--Woods factors.

\section*{Acknowledgements}

\noindent The authors thank Cyril Houdayer for his suggestion to investigate the Peterson--Thom property for type $\mathrm{III}$ factors and Stefaan Vaes for his comments on an early draft of the paper. BH acknowledges support from the NSF CAREER award DMS-214473. DJ was partially supported by the National Sciences and Engineering Research Council (Canada) and the Danish Independent Research Fund.  SKE was supported by NSF grant DMS-2350049. BN was supported by NSF grant DMS-2247047. SV was supported by FWO research project G090420N of the Research Foundation Flanders and by Methusalem grant METH/21/03 –- long term structural funding of the Flemish Government.

\tableofcontents

\section{Preliminaries}

Given a von Neumann algebra $M$, we will use lattice notation for its collection of (unital) von Neumann subalgebras. That is, we write $N\leq M$ to denote that $N$ is a von Neumann subalgebra of $M$, we write $N_1 \vee N_2$ for the von Neumann algebra generated by two von Neumann subalgebras $N_1,N_2\leq M$, etc. Similarly for a Hilbert space $\cH$ and its collection of closed subspaces.

\subsection{Exercises with weights} Let $M$ be a von Neumann algebra. Recall that a \emph{weight} on $M$ is a map $\varphi\colon M_+\to [0,+\infty]$ satisfying
    \[
        \varphi(\lambda x + y )= \lambda \varphi(x) + \varphi(y) \qquad \qquad \lambda\in \bR_+,\ x,y\in M_+.
    \]
We denote\footnote{These sets are traditionally denoted by $\mathfrak{n}_\varphi$ and $\mathfrak{m}_\varphi$, respectively, but we feel the above notation more clearly indicates their meaning.}
    \begin{align*}
        \sqrt{\dom}(\varphi)&:=\{x\in M\colon \varphi(x^*x) <\infty\}\\
        \dom(\varphi)&:=\Span\left\{x^*y \colon x,y\in \sqrt{\dom}(\varphi) \right\},
    \end{align*}
which are a left ideal and $*$-subalgebra of $M$, respectively. Then $\varphi$ naturally extends to a linear functional $\varphi\colon \dom(\varphi)\to \bC$. We say $\varphi$ is:
    \begin{itemize}
        \item \emph{faithful} if $\varphi(x^*x)=0$ implies $x=0$;
        \item \emph{normal} if $\varphi(\sup_i x_i) = \sup_i \varphi(x_i)$ for any bounded increasing net $(x_i)_{i\in I}\subseteq M_+$;
        \item \emph{semifinite} if $\dom(\varphi)''=M$;
        \item \emph{tracial} if $\varphi(x^*x)=\varphi(xx^*)$ for all $x\in M$.
    \end{itemize}
Note that for a tracial weight $\dom(\varphi)$ is a two-sided ideal. In this case, semifiniteness is equivalent to every non-trivial projection in $M$ having a non-trivial subprojection that is finite under $\varphi$ (see, for example, the proof of \cite[Theorem V.2.15]{TakesakiI}).

For a faithful normal semifinite weight, we write $L^2(M,\varphi)$ for the completion of $\sqrt{\dom}(\varphi)$ with respect to the norm induced by the inner product
    \[
        \<x,y\>_\varphi:= \varphi(y^*x) \qquad \qquad x,y\in \sqrt{\dom}(\varphi).
    \]
Left multiplication by $M$ extends to a faithful normal $*$-representation, and so we will frequently identify $M$ as a von Neumann subalgebra of $B(L^2(M,\varphi))$. (See \cite[Section VII.1]{TakesakiII} for additional details.)

The map $\sqrt{\dom}(\varphi)^*\cap \sqrt{\dom}(\varphi) \ni x \mapsto x^*$ gives a densely defined, closable, conjugate linear operator on $L^2(M,\varphi)$, whose closure we denote by $S_\varphi$ and call the \emph{Tomita operator} for $\varphi$. If $S_\varphi =J_\varphi \Delta_\varphi^{1/2}$ is its polar decomposition, then $J_\varphi$ is called the \emph{modular conjugation} for $\varphi$ and $\Delta_\varphi$ is called the \emph{modular operator} for $\varphi$. The modular conjugation $J_\varphi$ is a conjugate linear unitary operator that satisfies
    \[
        J_\varphi M J_\varphi = M'\cap B(L^2(M,\varphi)).
    \]
Consequently, $L^2(M,\varphi)$ has a natural $M$-bimodule structure defined by
    \[
        x\cdot \xi \cdot y := x (J_\varphi y^* J_\varphi) \xi \qquad \qquad x,y\in M,\ \xi\in L^2(M,\varphi).
    \]
The modular operator $\Delta_\varphi$ is a positive, nonsingular, selfadjoint operator that satisfies
    \[
        \Delta_\varphi^{it} M \Delta_\varphi^{-it} =M \qquad\qquad t\in \bR.
    \]
Consequently, $\sigma_t^\varphi(x):= \Delta_\varphi^{it} x \Delta_\varphi^{-it}$ defines an automorphism of $M$, and the associated action $\bR \overset{\sigma^\varphi}{\curvearrowright} M$ is known as the \emph{modular automorphism group} of $\varphi$. The fixed point subalgebra of this action is called the \emph{centralizer} of $\varphi$ and is denoted by $M^\varphi$. An alternate characterization for $x\in M^\varphi$ is that $x\cdot \dom(\varphi), \dom(\varphi)\cdot x \subseteq \dom(\varphi)$ and
    \[
        \varphi(xy) = \varphi(yx)
    \]
for all $y\in \dom(\varphi)$ (see \cite[Theorem VIII.2.6]{TakesakiII}).

If $M$ is equipped with two faithful normal semifinite weights $\varphi$ and $\psi$, then there are multiple ways to relate the above structures for $\varphi$ and $\psi$. First, the Connes cocycle derivative theorem \cite[Theorem 1.2.1]{ConnesCocycle} (see also \cite[Theorem VIII.3.3]{TakesakiII}) implies there exists a family of unitaries $\{u_t\in M\colon t\in \bR\}$ satisfying:
    \begin{itemize}
        \item $u_{s+t}=u_s \sigma_s^{\varphi}(u_t)$ for all $s,t\in \bR$;
        \item $\sigma_t^{\psi}(x)= u_t \sigma_t^\varphi(x) u_t^*$ for all $t\in \bR$ and $x\in M$.
    \end{itemize}
This family $\{u_t\in M\colon t\in \bR\}$ is known as the cocycle derivative of $\psi$ with respect to $\varphi$ and one denotes $(D\psi\colon D\varphi)_t:=u_t$ for all $t\in \bR$. (See the discussion preceding Lemma~\ref{lem:canonical_intertwiner_for_commuting_weights} for an explicit example.) 

The second way to relate $\varphi$ and $\psi$ is by the uniqueness of the standard form, which was established independently by Connes \cite[Theorem 2.1]{ConnesStdForm} and Haagerup \cite[Theorem 2.3]{HaagerupStdForm} (see also \cite[Theorem IX.1.14]{TakesakiII}). This implies, in particular, that there exists a unitary $U_{\varphi,\psi}\colon L^2(M,\psi)\to L^2(M,\varphi)$ between these Hilbert spaces that is uniquely determined by intertwining the left actions of $M$ and preserving the positive cones $U_{\varphi,\psi} L^2(M,\psi)_+ = L^2(M,\varphi)_+$. Here $L^2(M,\varphi)_+$ is the self-dual positive cone in $L^2(M,\varphi)$ given by the closure of $\Delta_\varphi^{1/4}\sqrt{\dom}(\varphi)_+$, and similarly for $L^2(M,\psi)_+$ (see \cite[Theorem IX.1.2 and Lemma IX.1.3]{TakesakiII}). This unitary also intertwines the modular conjugations $U_{\varphi,\psi} J_\psi = J_\varphi U_{\varphi,\psi}$, and consequently also intertwines the right actions of $M$:
    \[
        U_{\varphi,\psi}(x\cdot \xi\cdot y ) = x\cdot U_{\varphi,\psi}(\xi)\cdot y \qquad\qquad  x,y\in M,\ \xi\in L^2(M,\psi).
    \]
We shall refer to $U_{\varphi,\psi}$ as the \emph{canonical intertwiner} from $\psi$ to $\varphi$.

Suppose $h$ is a nonsingular, positive, selfadjoint operator affiliated with $M^\varphi$. For each $\varepsilon>0$ denote
    \[
        h_\varepsilon := \frac{h}{1+\varepsilon h} \in M^\varphi.
    \]
By \cite[Lemma VIII.2.8]{TakesakiII}
    \[
        \varphi_h(x):= \varphi(h^{\frac12} x h^{\frac12}) = \lim_{\varepsilon\to 0} \varphi(h_\varepsilon^{\frac12} x h_{\varepsilon}^{\frac12}) = \sup_{\varepsilon>0} \varphi(h_\varepsilon^{\frac12} x h_{\varepsilon}^{\frac12}) \qquad \qquad x\in M_+
    \]
defines a faithful normal semifinite weight on $M$, and by \cite[Lemma VIII.2.10]{TakesakiII} it satisfies
    \begin{align}\label{eqn:cocycle_derivative_formula}
        \sigma_t^{\varphi_h}(x) = h^{it} \sigma_t^{\varphi}(x) h^{-it} \qquad\qquad t\in \bR,\ x\in M.
    \end{align}
In fact, one has $h^{it}=(D \varphi_h \colon D \varphi)_t$ for all $t\in \bR$.

The following lemma provides a formula for the canonical intertwiner between the commuting weights $\varphi$ and $\varphi_h$. It is likely well-known to experts, but as we were unable to locate a reference we provide a sketch of the proof.

\begin{lem}\label{lem:canonical_intertwiner_for_commuting_weights}
Let $M$ be a von Neumann algebra equipped with a faithful normal semifinite weight $\varphi$, let $h$ be a nonsingular, positive, selfadjoint operator affiliated with $M^\varphi$, and let $h_\varepsilon$ be as above for all $\varepsilon>0$. Then $\sqrt{\dom}(\varphi_h) h_{\varepsilon}^{\frac12} \subseteq \sqrt{\dom}(\varphi)$ for all $\varepsilon>0$ and
    \[
        \lim_{\varepsilon\to 0} \| U_{\varphi, \varphi_h}(x) - x h_\varepsilon^{\frac12} \|_\varphi =0
    \]
for all $x\in \sqrt{\dom}(\varphi_h)$, where $U_{\varphi,\varphi_h}$ is the canonical intertwiner from $\varphi_h$ to $\varphi$.
\end{lem}
\begin{proof}
Using \cite[Lemma VIII.2.7]{TakesakiII} one has for $x\in \sqrt{\dom}(\varphi_h)$ and $\varepsilon\geq \delta >0$ that
    \[
        \| x h_\varepsilon^{\frac12}\|_\varphi \leq \sup_{\varepsilon>0} \| x h_\varepsilon^{\frac12}\|_\varphi = \| x \|_{\varphi_h}
    \]
and
    \[
        \| x h_{\varepsilon}^{\frac12} - x h_{\delta}^{\frac12}\|_\varphi^2 \leq \|x h_\varepsilon^{\frac12} \|_\varphi^2 - \| x h_\delta^{\frac12}\|_\varphi^2.
    \]
It follows that the map $x\mapsto xh_{\varepsilon}^{1/2}$ extends to a contraction $W_\varepsilon\colon L^2(M,\varphi_h)\to L^2(M,\varphi)$ for all $\varepsilon>0$, and that $(W_\varepsilon)_{\varepsilon>0}$ converges in the strong operator topology to an isometry $W\colon L^2(M,\varphi_h) \to L^2(M,\varphi)$. In fact, $W$ is a unitary whose inverse is given by the map obtained from the above argument after reversing the roles of $\varphi$ and $\varphi_h$ (note that $(\varphi_h)_{h^{-1}}=\varphi$). Thus it suffices to show $W=U_{\varphi,\varphi_h}$, and to do so we will invoke the uniqueness of the canonical intertwiner.

That $W$ intertwines the left actions of $M$ is immediate. Toward showing that $W L^2(M,\varphi_h)_+ = L^2(M,\varphi)_+$, denote $e_n:= 1_{[\frac1n,n]}(h) $ for each $n\in \bN$, and for $x\in M$ and $r>0$ define
    \[
        x_{n,r} := \sqrt{\frac{r}{\pi}} \int_{\bR} e^{-rt^2} \sigma_t^{\varphi_h}(e_n x e_n ) \ dt.
    \]
Then $x_{n,r}$ is $\sigma^{\varphi_h}$-analytic and, by (\ref{eqn:cocycle_derivative_formula}), $\sigma^\varphi$-analytic with
    \[
        \sigma_{z}^{\varphi_h}(x_{n,r}) =   \sigma_{z}^\varphi( h^{-iz}e_n x_{n,r} e_n h^{iz}) .
    \]
Consequently, it follows from \cite[Lemma VI.2.4]{TakesakiII} (see also the proof of \cite[Theorem VI.2.2.(i)]{TakesakiII}) that when $x\in \sqrt{\dom}(\varphi_h)\cap M_+$ one has
    \begin{align*}
        W\Delta_{\varphi_h}^{\frac14}(x) &= \lim_{n\to\infty} \lim_{r\to\infty} W\sigma_{-\frac{i}{4}}^{\varphi_h}(x_{n,r})\\
            &= \lim_{n\to\infty} \lim_{r\to\infty} W \sigma_{-\frac{i}{4}}^\varphi(h^{\frac14}e_n x_{n,r} e_n h^{-\frac14}) \\
            & = \lim_{n\to\infty} \lim_{r\to\infty} \lim_{\varepsilon\to 0}  \Delta_\varphi^{\frac14}(h^{\frac14}e_n x_{n,r} e_n h^{-\frac14} h_{\varepsilon}^{\frac12}) \\
            &= \lim_{n\to\infty} \lim_{r\to\infty} \Delta_\varphi^{\frac14}(h^{\frac14}e_n 
 x_{n,r} e_n h^{\frac14}) = \lim_{n\to\infty} \Delta_{\varphi}^{\frac14}( h^{\frac14} e_n x e_n h^{\frac14}).
    \end{align*}
Since $\Delta_{\varphi}^{\frac14}( h^{\frac14} e_n x e_n h^{\frac14}) \in \Delta_{\varphi}^{\frac14} M_+ \subseteq L^2(M,\varphi)_+$, the above implies $W L^2(M,\varphi_h)_+ \subseteq L^2(M,\varphi)_+$. Reversing the roles of $\varphi$ and $\varphi_h$ gives equality. Thus \cite[Lemma IX.1.5]{TakesakiII} implies $W=U_{\varphi,\varphi_h}$, as claimed.
\end{proof}

\subsection{Engaging the core}

Let $M$ be a von Neumann algebra equipped with a faithful normal semifinite weight $\varphi$. Define a unitary $U_\varphi$ on $L^2(\bR)\otimes L^2(M,\varphi)$ by
    \[
        [U_\varphi \xi](t) = \Delta_\varphi^{it} \xi(t),
    \]
where we have identified $L^2(\bR)\otimes L^2(M,\varphi)$ with (separably-valued) square integrable, weakly measurable functions $\xi\colon \bR \to L^2(M,\varphi)$ (see \cite[Section IV.7]{TakesakiII}). The \emph{continuous core} of $M$ is the crossed product von Neumann algebra
    \[
        N\rtimes_{\sigma^\varphi} \bR := (L(\bR)\otimes 1) \vee U_\varphi^*(1\otimes M) U_\varphi \subseteq B(L^2(\bR)\otimes L^2(M,\varphi)).
    \]
We will identify $M\cong U_\varphi^*(1\otimes M) U_\varphi $ and denote $L_\varphi(\bR):= L(\bR)\otimes 1$ and $\lambda_\varphi(t):= \lambda(t)\otimes 1$ for $t\in \bR$. In this case, $x\in M$ and $\lambda_\varphi(t)\in L_\varphi(\bR)$ are represented as follows:
    \begin{align*}
        [x\xi](s) &= \sigma_{-t}^\varphi(x) \xi(s)\\
        [\lambda_\varphi(t) \xi](s) &= \xi(s-t)
    \end{align*}
for all $\xi\in L^2(\bR)\otimes L^2(M,\varphi)$ and $s\in \bR$. Observe that
    \[
        \lambda_\varphi(t) x \lambda_\varphi(t)^* = \sigma_t^{\varphi}(x) \qquad \qquad t\in \bR,\ x\in M.
    \]
Importantly, the continuous core is independent of the choice of weight $\varphi$. If $\psi$ is another faithful normal semifinite weight on $M$, then there exists a $*$-isomorphism $\pi_{\varphi,\psi}\colon N\rtimes_{\sigma^\psi} \bR \to N\rtimes_{\sigma^\varphi}\bR$ satisfying
    \begin{align}\label{eqn:formula_isomorphism_continuous_cores}
         \begin{cases} \pi_{\varphi,\psi}(x) = x  &\qquad x\in M,\\
        \pi_{\varphi,\psi}(\lambda_\psi(t)) = (D\psi\colon D\varphi)_t \lambda_\varphi(t) &\qquad t\in \bR,
        \end{cases} 
    \end{align}
(see (14) in the proof of \cite[Theorem X.1.7.(ii)]{TakesakiII}). Here $\{(D\psi\colon D\varphi)_t\colon t\in \bR\}$ is the cocycle derivative of $\psi$ with respect to $\varphi$. Another feature of the continuous core that will be important for our purposes is that it is always \emph{diffuse} in the sense that it contains no minimal projections:

\begin{lem}\label{lem:continuous_core_is_diffuse}
Let $M$ be a von Neumann algebra equipped with a faithful normal weight $\varphi$. Then $M\rtimes_{\sigma^\varphi} \bR$ is diffuse.
\end{lem}
\begin{proof}
Let $z\in M'\cap M$ be a central projection such that $Mz$ is semifinite and $M(1-z)$ is purely infinite, which exists by the type decomposition of von Neumann algebras. Since $z\in M^\varphi$, it follows that $z$ is still central in $M\rtimes_{\sigma^\varphi} \bR$ and 
    \[
        M\rtimes_{\sigma^\varphi} \bR = \left[(M\rtimes_{\sigma^\varphi} \bR)z\right] \oplus \left[(M\rtimes_{\sigma^\varphi} \bR)(1-z)\right] = \left[(Mz) \rtimes_{\sigma^\varphi}\bR \right] \oplus \left[(M(1-z))\rtimes_{\sigma^\varphi} \bR\right],
    \]
where the last equality uses that $Mz$ and $M(1-z)$ are each $\sigma^\varphi$-invariant. We also note that $\varphi$ is semifinite on both $Mz$ and $M(1-z)$. So it suffices to show that each summand is diffuse. The latter is type $\mathrm{II}_\infty$, and hence diffuse, because $M(1-z)$ is type $\mathrm{III}$ (see \cite[Theorem XII.1.1]{TakesakiII}). For the former, if $\tau$ is a faithful normal semifinite tracial weight on $Mz$ then we have
    \[
        \pi_{\tau,\varphi}((Mz) \rtimes_{\sigma^\varphi}\bR) = (Mz) \rtimes_{\sigma^\tau}\bR  = (Mz)\widebar{\otimes} L(\bR). 
    \]
The above tensor product has diffuse center and is therefore diffuse.
\end{proof}

The continuous core is equipped with a \emph{dual action} $\bR \overset{\theta}{\curvearrowright} M\rtimes_{\sigma^\varphi} \bR$ determined by
    \begin{align}\label{eqn:dual_action_formula}
        \begin{cases}
        \theta_s(x) =x & \qquad s\in \bR,\ x\in M\\
       \theta_s(\lambda_\varphi(t) ) = e^{-2\pi ist} \lambda_\varphi(t) &\qquad t\in \bR
        \end{cases},
    \end{align}
and, in fact, $(M\rtimes_{\sigma^\varphi} \bR)^\theta = M$. By \cite[Theorem 1.1]{HaagerupDualWeightsII} there exists a faithful normal semifinite operator valued weight $T$ from $M\rtimes_{\sigma^\varphi} \bR$ to $M$ given by
    \[
        T(x):= \int_{\bR} \theta_s(x)\ ds \qquad x\in (M\rtimes_{\sigma^\varphi} \bR)_+.
    \]
The \emph{dual weight} to $\varphi$ is the faithful normal semifinite weight $\widetilde{\varphi}:= \varphi\circ T$ on $M\rtimes_{\sigma^\varphi} \bR$. Observe that $\widetilde{\varphi}\circ \theta_s = \widetilde{\varphi}$ for all $s\in \bR$ by definition. Let $C_c(\bR,M)$ denote the set of compactly supported continuous functions $f\colon \bR \to M$, where $M$ is equipped with the  $\sigma$-strong* topology, and define
    \[
        \lambda_\varphi(f):= \int_\bR \lambda_\varphi(t) f(t)\ dt \qquad \quad f\in C_c(\bR,M).
    \]
Then $\lambda_\varphi(C_c(\bR,M))''= M\rtimes_{\sigma^\varphi} \bR$ (see \cite[Lemma X.1.8]{TakesakiII}), and the dual weight satisfies
    \[
        \widetilde{\varphi}(\lambda_\varphi(f)^* \lambda_\varphi(f)) = \int_\bR \varphi(f(s)^* f(s))\ ds.
    \]
Consequently, $\lambda_\varphi(f) x \mapsto f \otimes x$ for $f\in C_c(\bR)$ and $x\in \sqrt{\dom}(\varphi)$ extends to an isomorphism
    \[
        L^2(M\rtimes_{\sigma^\varphi}\bR, \widetilde{\varphi}) \cong L^2(\bR)\otimes L^2(M,\varphi).
    \]
In particular, the modular conjugation $J_{\widetilde{\varphi}}$ is naturally represented on the latter Hilbert space as
    \[
        [J_{\widetilde{\varphi}} \xi](t) = \Delta_{\varphi}^{-it} J_\varphi ( \xi(-t)) \qquad\qquad t\in \bR,
    \]
(see \cite[Lemma X.1.13]{TakesakiII}). It follows that $L^2(\bR)\otimes L^2(M,\varphi)$ has the following $(M\rtimes_{\sigma^\varphi}\bR)$-bimodule structure:
    \begin{align}\label{eqn:continuous_core_bimodue_structure}
    \begin{split}
        [x\cdot \xi \cdot y](t)&= \sigma_{-t}^\varphi(x)\cdot \xi(t) \cdot y  \\
        [\lambda_\varphi(s)\cdot \xi \cdot \lambda_\varphi(r)](t) &= \Delta_\varphi^{-ir} \xi(t-s-r)
    \end{split}
    \end{align}
for $r,s,t\in \bR$ and $x,y\in M$.

The modular automorphism group of $\widetilde{\varphi}$ is determined by
    \[
        \begin{cases}
            \sigma_s^{\widetilde{\varphi}}(x) = \sigma_t^\varphi(x) & \qquad s\in \bR,\ x\in M,\\
            \sigma_s^{\widetilde{\varphi}}(\lambda_\varphi(t)) = \lambda_\varphi(t) & \qquad t\in \bR.
        \end{cases}
    \]
Note that $L_\varphi(\bR)\vee M^\varphi \leq (M\rtimes_{\sigma^\varphi} \bR)^{\widetilde{\varphi}}$ and the modular automorphism group of $\widetilde{\varphi}$ commutes with the dual action. The above formulas also imply $\sigma_t^{\widetilde{\varphi}} = \Ad(\lambda_\varphi(t))$ for all $t\in \bR$. Consequently, if $D_\varphi$ is the nonsingular, positive, selfadjoint operator affiliated with $L_\varphi(\bR)$ satisfying $D_\varphi^{it} =\lambda_\varphi(t)$, then
    \[
        \tau_\varphi:= \widetilde{\varphi}_{D_\varphi^{-1}}
    \]
defines a faithful normal semifinite tracial weight on $M\rtimes_{\sigma^\varphi} \bR$ (see the proof of \cite[Theorem XII.1.1.(ii)]{TakesakiII}). We will refer to $\tau_\varphi$ as the tracial weight \emph{induced by $\varphi$}. An important feature of this tracial weight is that it is scaled by the dual action: $\tau_\varphi\circ \theta_s = e^{2\pi s} \tau_\varphi$ for all $s\in \bR$.

\subsection{Managing expectations}

We say an inclusion $N\leq M$ of von Neumann algebras is \emph{with expectation} if there exists a faithful normal conditional expectation $\cE_N\colon M\to N$. Such inclusions will be at the heart of our analysis, and so in this section we record several results that will be useful to us throughout the article.

Let $M$ be a von Neumann algebra equipped with a faithful normal semifinite weight $\varphi$. For a subalgebra $N\leq M$, the following are equivalent by \cite[Theorem IX.4.2]{TakesakiII}:
    \begin{enumerate}[(i)]
        \item there exists a $\varphi$-preserving faithful normal conditional expectation $\cE_N\colon M\to N$;
        \item $\varphi|_N$ is semifinite and $N$ is $\sigma^\varphi$-invariant.
    \end{enumerate}
We will say such an inclusion is \emph{with $\varphi$-expectation}. Note that any inclusion $N\leq M$ with a faithful normal conditional expectation $\cE_N\colon M\to N$ will be with $(\varphi\circ \cE)$-expectation for any faithful normal semifinite weight $\varphi$ on $N$.

The next two lemmas will be used frequently (and often implicitly) to pass properties up from a subalgebra $N$ to $M$. The first implies that if $\varphi$ is a faithful normal weight on a von Neumann algebra $M$ with $\varphi|_N$ semifinite, then $\varphi$ is semifinite on $M$. The second implies that if $N\leq M$ is with expectation and $N$ is diffuse, then $M$ is diffuse.

\begin{lem}\label{lem:semifinite_iff_approx_proj_unit}
A faithful normal weight $\varphi$ on a von Neumann algebra $M$ is semifinite if and only if there exists a net of projections $(p_i)_{i\in I} \subseteq \dom(\varphi)$ converging to $1$ in the strong operator topology. If $\varphi$ is tracial then we can choose $(p_{i})_{i\in I}$ to be increasing.
\end{lem}
\begin{proof}
$(\Rightarrow):$ Fix a faithful normal representation $M\subseteq B(\cH)$. It suffices to show that every strong operator topology neighborhood of $1$ contains some projection $p\in \dom(\varphi)$, so further fix $\xi_1,\ldots, \xi_n\in \cH$ and $\varepsilon>0$. Since $\dom(\varphi)$ is a weak$^{*}$-dense subalgebra, Kaplansky's density theorem gives a net $(x_{j})_{j\in J}\subseteq \dom(\varphi)$ with $0\leq x_{j}\leq 1$ and with $x_{j}\to 1$ strongly. Thus we can find $j\in J$ with $\|(x_{j}-1)\xi_{k}\|<\varepsilon$ for each $k=1,\cdots,n$. Then 
    \[
        \|(1_{(1/2,\infty)}(x_{j})-1)\xi_{k}\|=\|1_{[0,1/2]}(x_{j})\xi_{k}\|\leq 2\|(1-x_{j})\xi_{k}\|<2\varepsilon
    \]
for each $k=1,\cdots,n$. Additionally, $1_{(1/2,\infty)}(x_{j})\leq 2x_{j}$ so that $1_{(1/2,\infty)}(x_{j})\in \dom(\varphi)$.

Now suppose that $\varphi$ is tracial, let $(p_{i})_{i\in I}$ be a net of projections tending to $1$ strongly. Let $J$ be the set of finite subsets of $I$, ordered by inclusion. For $F\in J$, let $p_{F}=\bigvee_{i\in F}p_{i}$. By the proof of \cite[Proposition V.2.10]{TakesakiI}, we have that $p_{F}\in \dom(\varphi)$ and since $(1-p_{F})\leq (1-p_{i})$ for all $i\in F$, we have that $p_{F}\to 1$ strongly. \\

\noindent $(\Leftarrow):$ Note that $p_i\in \sqrt{\dom}(\varphi)^*\cap \sqrt{\dom}(\varphi)$ and consequently
    \[
        p_i M p_i \subseteq \dom(\varphi)
    \]
for all $i\in I$. It follows that $\dom(\varphi)''= M$.
\end{proof}

\begin{lem}\label{lem:diffuseness_passed_upwards}
For an inclusion $N\leq M$ with expectation, $N$ being diffuse implies $M$ is diffuse
\end{lem}
\begin{proof}
This follows from restricting the conditional expectation to the maximal purely atomic direct summand of $M$ and applying \cite[Theorem IV.2.2.3]{BlackadarOA}.
\end{proof}

We will use the following lemma primarily in the semifinite case, where demanding a common conditional expectation is less stringent (see the discussion following the proof).

\begin{lem}\label{lem:common_expectation_canonical_intertwiner}
Let $M$ be a von Neumann algebra equipped with faithful normal semifinite weights $\varphi$ and $\psi$ and let $U_{\varphi,\psi}$ be the canonical intertwiner from $\psi$ to $\varphi$. If $N\leq M$ admits a faithful normal conditional expectation that preserves both $\varphi$ and $\psi$, then
    \[
        U_{\varphi,\psi} L^2(N,\psi) = L^2(N,\varphi).
    \]
\end{lem}
\begin{proof}
Let $\cE\colon M\to N$ be faithful normal conditional expectation that preserves both $\varphi$ and $\psi$. Then $\cE\otimes I_2\colon M\otimes M_2(\bC)\to N\otimes M_2(\bC)$ preserves the faithful normal semifinite weight $\rho=\varphi\oplus \psi$ given by
    \[
        \rho \left(\begin{array}{cc}
             x_{11} & x_{12}  \\
             x_{21} & x_{22} 
        \end{array} \right):= \varphi(x_{11}) + \psi(x_{22}).
    \]
This implies that $L^2(N\otimes M_2(\cC),\rho)$ is invariant for the modular conjugation $J_\rho$ on $L^2(M\otimes M_2(\bC), \rho)$, which has the form
        \[
         J_\rho = \left(\begin{array}{cccc} J_\varphi & 0 & 0 & 0 \\
        0 & 0 & J_{\varphi,\psi} & 0 \\
        0 & J_{\psi, \varphi} & 0 & 0\\
        0 & 0 & 0 & J_{\psi}\end{array}\right)
    \]
under the identification
    \begin{align*}
        L^2(M\otimes M_2(\bC),\rho) &\overset{\cong}{\to} L^2(M,\varphi)\oplus L^2(M,\psi)\oplus L^2(M,\varphi) \oplus L^2(M,\varphi)\\
            \left(\begin{array}{cc}
             x_{11} & x_{12}  \\
             x_{21} & x_{22} 
        \end{array} \right) &\to (x_{11}, x_{12}, x_{21}, x_{22}).
    \end{align*}
(see \cite[Section VIII.3]{TakesakiII}). Therefore
    \[
        J_\varphi J_{\varphi,\psi} L^2(N,\psi) = J_\varphi L^2(N,\varphi) = L^2(N,\varphi),
    \]
where the second equality follows from $N\leq M$ being with $\varphi$-expectation. But the canonical intertwiner is precisely given by $U_{\varphi,\psi}= J_\varphi J_{\varphi,\psi}$ (see \cite[Lemma IX.1.5]{TakesakiII}).
\end{proof}

Suppose $N\leq M$ is with $\varphi$-expectation $\cE_N\colon M\to N$.
Then the $\sigma^\varphi$-invariance of $N$ implies we can identify $N\rtimes_{\sigma^\varphi} \bR$ as a subalgebra of the continuous core $M\rtimes_{\sigma^\varphi} \bR$, and this inclusion admits a faithful normal conditional expectation determined by
    \[
        \cE_{N\rtimes \bR}(\lambda_\varphi(f)) = \lambda_\varphi(\cE_N\circ f) \qquad\qquad  f\in C_c(\bR,M).
    \]
This conditional expectation preserves the dual weight $\widetilde{\varphi}$ by definition, and it preserves the tracial weight $\tau_\varphi$ induced by $\varphi$ since $D_\varphi$ is affiliated with $L_\varphi(\bR)\leq N\rtimes_{\sigma^\varphi}\bR$. Hence Lemma~\ref{lem:common_expectation_canonical_intertwiner} can be applied 
 to $N\rtimes_{\sigma^\varphi} \bR\leq M\rtimes_{\sigma^\varphi}\bR$ and the pair of weights $\widetilde{\varphi}$ and $\tau_\varphi$. That this inclusion is with $\tau_\varphi$-expectation also implies that $\tau_\varphi|_{N\rtimes \bR}$ is semifinite by \cite[Theorem IX.4.2]{TakesakiII}. In particular, when $\varphi$ is a state it follows by considering $N=\bC$ that $\tau_\varphi|_{L_\varphi(\bR)}$ is semifinite and
    \[
        \cE_{L_\varphi(\bR)}(\lambda_\varphi(s) x \lambda_\varphi(t)) = \varphi(x) \lambda_\varphi(s+t) \qquad \qquad s,t\in \bR,\ x\in M
    \]
determines the $\tau_\varphi$-preserving faithful normal conditional expectation onto $L_\varphi(\bR)$. Furthermore, if $\psi$ is another faithful normal state on $M$, then $\tau_\varphi \circ \pi_{\varphi,\psi} = \tau_\psi$ (see \cite[Section 2.2]{HRApprox}).

We conclude this section by recording a recognition lemma for inclusions of the form $N\rtimes_{\sigma^\varphi} \bR \leq M\rtimes_{\sigma^\varphi} \bR$, which follows from the proof of \cite[Proposition X.2.6]{TakesakiII}. Note that the fixed points under the dual action $\theta$ are $\sigma^\varphi$-invariant since the $\theta$ commutes with the modular automorphism group of $\sigma^{\widetilde{\varphi}}$, which restricts to $\sigma^\varphi$ on $M$.

\begin{lem}\label{lem:continuous_core_recognition_lemma}
Let $M$ be a von Neumann algebra equipped with a faithful normal semifinite weight $\varphi$. Then a von Neumann subalgebra $\widetilde{P}\leq M\rtimes_{\sigma^\varphi} \bR$ is of the form $P\rtimes_{\sigma^\varphi} \bR$ for some $\sigma^\varphi$-invariant $P\leq M$ if and only if $\widetilde{P}$ contains $L_\varphi(\bR)$ and is $\theta$-invariant. In this case, one has $P=(\widetilde{P})^\theta$.
\end{lem}

\subsection{Peterson--Thom property}\label{sec: PT property}

Though we will prove much stronger indecomposability properties for type $\mathrm{III}$ factors, we nevertheless feel that an important aspect of our work is clarifying what the appropriate definition of the Peterson--Thom property should be for general von Neumann algebras. We record this observation here. For terminology, recall that a \emph{largest element} in a poset is an element that is greater than or equal to every element of the poset (this is of course stronger than simply being maximal).

\begin{prop}\label{prop:PT foundation}
 Let $M$ be a von Neumann algebra. Consider the following properties:
 \begin{enumerate}[(i)]
     \item For any diffuse, amenable $Q\leq M$ the set 
        \[
            \{P_{0}\colon P_0\leq M \text{ is with expectation, $P_0\supseteq Q$, and $P_0$ is amenable}\}
        \]
     has a largest element. \label{item: PT via maximality}
     \item If $\{Q_{j}\colon j\in J\}$ are diffuse, amenable, and with expectation in $M$ and are such that $\bigcap_{j\in J}Q_{j}$ contains a diffuse subalgebra with expectation in $M$, then for any $B\leq \bigvee_{j\in J}Q_{j}$ with expectation in $M$, we have that $B$ is amenable. \label{item: PT via joins}
 \end{enumerate}
 Then (\ref{item: PT via maximality}) implies (\ref{item: PT via joins}). 
\end{prop}

\begin{proof}
Fix a family $\{Q_{j}\colon j\in J\}$ as in (\ref{item: PT via joins}) and let $N\leq \bigcap_{j\in J} Q_j$ be a diffuse subalgebra with expectation in $M$. Note that $N$ is with expectation in $Q_{j}$ for all $j\in J$ since we can restrict to $Q_j$ the expectation from $M$ to $N$. In particular, $N$ is amenable. Thus we can let $P$ the unique largest element of the set
    \[
        \cN:=\{P_{0}\colon P_0\leq M \text{ is with expectation, $P_0\supseteq N$, and $P_0$ is amenable}\}.
    \]
For each $j\in J$, let $P_{j}$ be the largest element of the set
    \[
        \cQ_j:=\{P_{0}\colon P_0\leq M \text{ is with expectation, $P_0\supseteq Q_j$, and $P_0$ is amenable}\}.
    \]
Since $\cQ_j \subset \cN$, we have $P_{j}\subseteq P$
for all $j\in J$. In particular, $\bigvee_{j\in J}Q_{j}\subseteq P$.
So any $B\leq \bigvee_{j\in J}Q_{j}$ with expectation in $M$ is contained in $P$ and with expectation. Thus the amenability of $P$ implies that of $B$.  
\end{proof}

We say that a von Neumann algebra $M$ has the \textbf{Peterson--Thom property} if (\ref{item: PT via maximality}) of the above theorem holds. Note that if $M$ is finite, then both items in Proposition \ref{prop:PT foundation} are equivalent. It is unclear if this is true in general, because, for example, $\bigvee_{j\in J}Q_{j}$ may fail to be with expectation. 

The following result should be compared to properties like strong solidity, which pass to diffuse subalgebras with expectation.

\begin{prop} \label{prop: PT subalgebra}
Suppose $M$ is a $\sigma$-finite von Neumann algebra with the Peterson--Thom property. Then any diffuse subalgebra of $M$ which is with expectation has the Peterson--Thom property.
\end{prop}

\begin{proof}
Suppose that $N\leq M$ is diffuse and with a faithful normal conditional expectation $\cE_{N}\colon M\to N$. Let $Q\leq N$ be diffuse, amenable and with expectation. Let $P(Q,M)$ be the largest element of
    \[
        \{P_{0}: P_{0}\leq M \textnormal{ is diffuse, amenable, with expectation and $Q\subseteq P_{0}$}\}.
    \]
Extend a faithful normal state $\psi$ on $Q$ to $M$ by precomposing with expectation onto $N$ followed by the expectation onto $Q$, and denote the resulting state on $M$ by $\psi$. Then $Q$ and $N$ are both $\sigma^\psi$-invariant. The former, along with the uniqueness of $P(Q,M)$, implies $P(Q,M)$ is $\sigma^\psi$-invariant, which combined with the $\sigma^\psi$-invariance of $N$ implies $P:=P(Q,M)\cap N$ is also $\sigma^\psi$-invariant. Hence $P$ is with expectation in $M$ and therefore with expectation in $P(Q,M)$, implying that $P$ is amenable.

Now, suppose that $P_{0}\leq N$ is an amenable algebra with expectation and contains $Q$. Then $P_{0}\subseteq P(Q,M)$, by definition of $P(Q,M)$, and hence $P_{0}\subseteq P(Q,M)\cap N=P$. 
\end{proof}

\section{Pinsker algebras}

Key tools for our analysis are versions of the Pinsker algebra in the semifinite and purely infinite settings. For a finite von Neumann algebra $M$ and $Q\leq M$ with non-positive $1$-bounded entropy in the presence, the \emph{Pinsker algebra} is the largest intermediate subalgebra with entropy zero (see \cite{FreePinsker, UselessResolutionOfPT}). By \cite[Theorem 3.8]{Hayes2018}, such algebras automatically contain the anticoarse space and hence several weakenings of the normalizer of interest to us. We will establish similar results for faithful normal states (see Theorem~\ref{thm:quasi anticoarse stuff}), which will be crucial for our generalized solidity results in Theorem~\ref{thm: the game is over and you should quit}.

\subsection{Pinsker algebras for tracial weights}\label{sec:Pinsker for traces}

In the finite case, recall from \cite[Lemma 3.3]{UselessResolutionOfPT} that the non-positivity of $1$-bounded entropy is preserved under taking corners. Since we are primarily interested in this non-positivity, it is natural to extend the definition of the $1$-bounded entropy to the semifinite setting in the following way. For a von Neumann algebra $M$ equipped with a faithful normal semifinite tracial weight $\tau$, a non-zero projection $q\in \dom(\tau)$, and $Q\leq qMq$, we let $h_{\tau}(Q:qMq)$ denote the $1$-bounded entropy of $Q$ in the presence of $qMq$ with respect to the faithful normal tracial state $\tau^{(q)}:=\frac{1}{\tau(q)}\tau|_{qMq}$ (see \cite{Hayes2018} for the precise definition of $1$-bounded entropy, as well as \cite{FreePinsker, UselessResolutionOfPT, JKEupgradedfreeindependencephenomena} for other developments). 
For example, if $Q\leq M$ is with $\tau$-expectation then we can consider $h_\tau(q Q q : qMq)$ for any non-zero projection $q\in \dom(\tau|_Q)$, and we will largely be interested in inclusions where this quantity is non-negative for all such projections. Our main result in this subsection, Corollary~\ref{cor:independent of the proj}, states that this condition does not depend on the choice of tracial weight $\tau$. We begin by establishing this for tracial \emph{states}.

\begin{lem}\label{lem:independ of the trace finite case}
Let $M$ be a finite von Neumann algebra equipped with two faithful normal tracial states $\tau_{1},\tau_{2}$. Given $Q\leq M$ we have that $h_{\tau_{1}}(Q:M)\leq 0$ if and only if $h_{\tau_{2}}(Q:M)\leq 0$.
\end{lem}

\begin{proof}
Decomposing the center into diffuse and atomic pieces gives a family of projections $\{z_0\}\cup\{z_i\colon i\in I\}\subseteq M'\cap M$ (potentially with $I=\emptyset$) so that:
    \begin{itemize}
    \item $z_0+ \sum_{i\in I}z_{i}=1$;
    \item $Mz_0$ is either has diffuse center or is $\{0\}$;
    \item $Mz_i$ is a factor for every $i\in I$.
    \end{itemize}
Suppose $h_{\tau_{1}}(Q:M)\leq 0$. While the $z_{j}$'s need not lie in $Q$, we still have that $Qz_{j}$ is a subalgebra of $Mz_{j}$ for every $j\in \{0\}\cup I$. Since $Mz_{0}$ is either zero or has diffuse center, \cite[Example 3]{FreePinsker} implies 
    \[
        h_{\tau_{2}}(Qz_0:Mz_0)\leq h_{\tau_{2}}(Mz_0)\leq 0.
    \]
Define $N=\overline{\sum_{i\in I}Q z_{i}}^{SOT}$. By centrality of each $z_{i}$, we have that $2z_{i}-1\in \cN_{M}(Q)$ for each $i$, and thus 
$N\leq W^{*}(\cN_{M}(Q)).$ So  \cite[Theorem 3.8]{Hayes2018} implies $h_{\tau_{i}}(N:M)=h_{\tau_{i}}(Q:M)$, for $i=1,2$. So it suffices to show that $h_{\tau_{2}}(N:M)\leq 0$. By the proof of \cite[Proposition A.13]{Hayes2018}, we have that 
\[h_{\tau_{2}}(N:M)\leq \sum_{i}\tau_{2}(z_{i})^{2}h_{\tau_{2}}(Nz_{i}:M z_{i}).\] 
So it suffices to show that $h_{\tau_{2}}(Nz_{i}:Mz_{i})\leq 0$ for all $i\in I$. 
Since $Nz_{i}=Qz_{i},$ we thus have to show that $h_{\tau_{2}}(Qz_{i}:Mz_{i})\leq 0$.
But, by factoriality,
    \[
        h_{\tau_{2}}(Qz_{i}:Mz_{i})=h_{\tau_{1}}(Qz_{i}:Mz_{i}),
    \]
for all $i\in I$, so it remains to show that $h_{\tau_{1}}(Qz_{i}:Mz_{i})\leq 0$ for every $i\in I$.  
By \cite[Lemma 3.3]{UselessResolutionOfPT}, it follows that 
    \[
        h_{\tau_{1}}(Qz_{i}:M z_{i})=h_{\tau_{1}}(Nz_{i}:M z_{i})\leq 0.\qedhere
    \]
\end{proof}

\begin{lem}\label{lem: independ of the proj}
Let $M$ be a von Neumann algebra equipped with a faithful normal semifinite tracial weight $\tau$, and let $Q\leq M$ with $\tau|_Q$ semifinite. Suppose there exists a net of projections $(q_i)_{i\in I} \subseteq \dom(\tau|_Q)$ converging to $1$ in the strong operator topology so that $h_{\tau}(q_{i}Qq_{i}:q_{i}Mq_{i})\leq 0$. Then for every $q\in \dom(\tau|_Q)$ we have that $h_{\tau}(qQq:qMq)\leq 0$.
\end{lem}
\begin{proof}
Fix a $q\in \dom(\tau|_Q)$. Decomposing the center of $Q$ into diffuse and atomic pieces gives a family of projections $\{f_0\}\cup \{f_j\colon j\in J\}\subseteq Q'\cap Q$ (potentially with $J=\emptyset$) so that:
    \begin{itemize}
    \item $f_0+\sum_{j\in J}f_{j}=1$;
    \item $Qf_0$ either has diffuse center or is $\{0\}$;
    \item $Qf_j$ is a factor for every $j\in J$.    
\end{itemize}
Setting $N:=\bigvee_{k\in \{0\}\cup J} f_k M f_k$, we have
    \[
        h_{\tau}(qQq:qMq)\leq h_{\tau}(qQq:qNq),
    \]
so that it suffices to show $h_{\tau}(qQq:qNq)\leq 0$.

By \cite[Example 3]{FreePinsker} and \cite[Lemma 3.3]{UselessResolutionOfPT}, we have that $h_\tau(qQqf_0:f_0qMqf_0)\leq 0$. By the proof of \cite[Proposition A.13]{Hayes2018} we have that 
    \[
        h_\tau(qQq:qNq)\leq \sum_{j\in J}\frac{\tau(qf_{j})^{2}}{\tau(q)^{2}}h_\tau(qQqf_{j}:f_{j}qMqf_{j})\leq 0,
    \]
so it suffices to show that $h(qQqf_{j}:f_{j}qMqf_{j})\leq 0$ for all $j\in J$. This is true if $Qf_{j}$ is atomic or if $qf_{j}=0$, so we only focus on the  case that $Qf_{j}$ is diffuse and $qf_{j}\ne 0$. Since $q_{i}\to 1$, we find an $i$ so that $q_{i}f_{j}\ne 0$. If $(q\vee q_{i})f_{j}=q_{i}f_{j}$, then $qf_{j}\leq q_{i}f_{j}$ and the result follows from \cite[Lemma 3.3]{UselessResolutionOfPT}. So we further assume that $(q\vee q_{i})f_{j}\ne q_{i}f_{j}$. Since $q,q_{i}$ have finite trace and $Qf_{j}$ is a factor, we may find finitely many partial isometries $v_{1},\cdots,v_{k}\in Qf_{j}$ with  $v_{1}=q_{i}$, $v_{\ell}^{*}v_{\ell}\leq q_{i}f_{j}$, and 
    \[
        \sum_{\ell=1}^k v_{\ell}v_{\ell}^{*}=(q\vee q_{i})f_{j}.
    \]
Set
    \[
        Q_{0}=q_{i}Qq_{i}f_{j}+\sum_{\ell=1}^k v_{\ell}v_{\ell}^{*}Qv_{\ell}v_{\ell}^{*}.
    \]
Since the inclusion $v_{\ell}v_{\ell}^{*}Qv_{\ell}v_{\ell}^{*}\subseteq v_{\ell}v_{\ell}^{*}Mv_{\ell}v_{\ell}^{*}$ is isomorphic with $v_{\ell}^{*}v_{\ell}Qv_{\ell}^{*}v_{\ell}\subseteq v_{\ell}^{*}v_{\ell}Mv_{\ell}^{*}v_{\ell}$, we have
    \[
        h_\tau(v_{\ell}v_{\ell}^{*}Qv_{\ell}v_{\ell}^{*}:v_{\ell}v_{\ell}^{*}Mv_{\ell}v_{\ell}^{*})=h_{\tau}(v_{\ell}^{*}v_{\ell}Qv_{\ell}^{*}v_{\ell}:v_{\ell}^{*}v_{\ell}Mv_{\ell}^{*}v_{\ell})\leq 0,
    \]
where in the last step we use \cite[Lemma 3.3]{UselessResolutionOfPT} and that 
    \[
        h_\tau(q_{i}Qq_{i}:q_{i}Mq_{i})\leq 0.
    \]
Thus by the  proof of \cite[Proposition A.13]{Hayes2018}
    \begin{align*}
       h_\tau(Q_{0}: f_j(q\vee q_{i}) M (q\vee q_{i})f_{j})&\leq h\left(Q_{0}:\left(\sum_{\ell=1}^k v_{\ell}v_{\ell}^{*}\right)M\left(\sum_{\ell=1}^k v_{\ell}v_{\ell}^{*}\right)\right)\\
       &\leq \sum_{\ell=1}^{k}\frac{\tau(v_{\ell}^{*}v_{\ell})^{2}}{\tau((q\vee q_{i})f_{j})^{2}}h_\tau(v_{\ell}v_{\ell}^{*}Qv_{\ell}v_{\ell}^{*}:v_{\ell}v_{\ell}^{*}Mv_{\ell}v_{\ell}^{*})\\
       &\leq 0.
    \  
    \end{align*}
Since $Qf_{j}$ is a factor, we have \[qQqf_{j}\subseteq qW^{*}(\cN_{f_j(q\vee q_{i})M(q\vee q_{i})f_{j}}(Q_{0}))q,\] and thus by \cite[Lemma 3.3]{UselessResolutionOfPT}, and \cite[Theorem 3.8]{Hayes2018} it follows that $h_\tau(qQqf_{j}:f_{j}qMqf_{j})\leq 0$.
\end{proof}

\begin{cor}\label{cor:independent of the proj}
Let $M$ be a semifinite von Neumann algebra equipped with two faithful normal semifinite tracial weights $\tau_1,\tau_2$, and let $Q\leq M$ with $\tau_i|_Q$ semifinite for each $i=1,2$. Then $h_{\tau_1}(q_{1}Qq_{1}:q_{1}Mq_{1})\leq 0$ for every projection $q_1\in \dom(\tau_1|_Q)$ if and only if $h_{\tau_2}(q_{2}Qq_{2}:q_{2}Mq_{2})\leq 0$  for every projection $q_{2}\in \dom(\tau_2|_Q)$. 
\end{cor}
\begin{proof}
We first note that $(\tau_1+\tau_2)|_Q$ is semifinite. Indeed, for any $x\in \dom(\tau_1|_Q)$ and $y\in \dom(\tau_2|_Q)$ we have
    \[
        xQy \in \dom(\tau_1|_Q)\cap \dom(\tau_2|_Q)
    \]
by traciality. The latter set is of course contained in $\dom(\tau_1+\tau_2|_Q)$, and hence the semifiniteness of $\tau_1|_Q$ and $\tau_2|_Q$ implies that of $(\tau_1+\tau_2)|_Q$. Consequently, Lemma~\ref{lem:semifinite_iff_approx_proj_unit} yields an increasing net of projections $(q_i)_{i\in I}\subseteq \dom(\tau_1+\tau_2|_Q)$ such that $q_{i}\to 1$ in strong operator topology. Lemma~\ref{lem:independ of the trace finite case} implies $h_{\tau_1}(q_i Q q_i: q_i Mq_i)\leq 0$ if and only if $h_{\tau_2}(q_i Q q_i : q_i M q_i) \leq 0$ for all $i\in I$, and so the conclusion then follows from Lemma~\ref{lem: independ of the proj}.
\end{proof}

Based on the previous corollary, we establish the following terminology.

\begin{defn}\label{defn:semifintie zero entropy}
Let $M$ be a semifinite von Neumann algebra. 
We say that $Q\leq  M$ is an \textbf{entropy free inclusion} if for some (equivalently, every, by Corollary \ref{cor:independent of the proj}) faithful normal semifinite tracial weight $\tau$ on $M$ with $\tau|_Q$ semifinite we have $h_\tau(pQp:pMp)\leq 0$ for every $p\in \dom(\tau|_Q)$.
\end{defn}

Note that by Lemma~\ref{lem: independ of the proj}, we can instead require that there exists some net of projections $(q_i)_{i\in I} \subseteq \dom(\tau|_Q)$ converging to $1$ in the strong operator topology so that $h_\tau(q_{i}Qq_{i}:q_{i}Mq_{i})\leq 0$ for all $i\in I$.

\begin{defn}\label{defn: Pinsker algebra}
 Let $M$ be a semifinite von Neumann algebra. We say that $P\leq M$ is a \textbf{Pinsker algebra in $M$} if $P\leq M$ is a maximal entropy free inclusion. That is, $P\leq M$ is an entropy free inclusion, and whenever $Q\leq M$ is an entropy free inclusion with $Q\supseteq P$ then $P=Q$.
\end{defn}

\begin{prop}\label{prop:PT_II_infty}
Let $M$ be a semifinite von Neumann algebra and let $Q\leq M$ be an entropy free inclusion. 
If $Q$ is diffuse, then there exists a unique Pinsker algebra $P\leq M$ satisfying $P\supseteq Q$. 
\end{prop}

\begin{proof}
Fix a faithful normal semifinite tracial weight $\tau$ on $M$ such that $\tau|_Q$ is semifinite, and use Lemma~\ref{lem:semifinite_iff_approx_proj_unit} to find an increasing net of projections $(q_i)_{i\in I}\subseteq \dom(\tau|_Q)$ that converges to $1$ in the strong operator topology. For each $i\in I$, let $P_{i}$ be the Pinsker algebra for $q_{i} Q q_{i}$ inside of $q_i M q_i$, which exists by \cite[Lemma 3.3]{UselessResolutionOfPT}. For $i\leq j$ we have $q_i P_j q_i = P_i$ by behavior of Pinsker algebras under corners (see \cite[Propositon 3.4]{UselessResolutionOfPT}). Thus if we set
    \[
        P=\overline{\bigcup_{i\in I} P_{i}}^{SOT},
    \]
then $P\leq M$ is an entropy free inclusion by Lemma~\ref{lem: independ of the proj} and $Q\leq P$. If $N\leq M$ is another entropy free inclusion with $Q\leq N$, then for any $i\in I$ we have $q_i Q q_i \leq q_i N q_i$. Moreover, $q_i N q_i$ is diffuse and  $h_\tau(q_i N q_i:q_i M q_i)\leq 0$. 
So by uniqueness of the Pinsker algebra in the finite case (see the discussion following \cite[Definition 2.5]{UselessResolutionOfPT}), we must have $q_i N q_i \subseteq P_i$. It follows that $N \leq P$, and hence $P$ is both a Pinsker algebra and unique.
\end{proof}

\subsection{Pinsker algebras for general weights}

We will now adapt the results from the previous subsection to arbitrary von Neumann algebras $M$ equipped with a choice of faithful normal semifinite weight $\varphi$ by considering the representation of the continuous core given by $M\rtimes_{\sigma^\varphi}\bR$. Recall from the discussion following Lemma~\ref{lem:common_expectation_canonical_intertwiner} that $Q\rtimes_{\sigma^\varphi}\bR \leq M\rtimes_{\sigma^\varphi}\bR$ is with $\tau_\varphi$-expectation whenever $Q\leq M$ is with $\varphi$-expectation, where $\tau_\varphi$ is the tracial weight induced by $\varphi$.

\begin{defn}\label{defn:varphi_entropy_free_inclusion}
Let $M$ be a von Neumann algebra equipped with a faithful normal semifinite weight $\varphi$. We say that $Q\leq M$  is a \textbf{$\varphi$-entropy free inclusion} if it is with $\varphi$-expectation and $Q\rtimes_{\sigma^{\varphi}}\bR \leq M\rtimes_{\sigma^{\varphi}}\bR$ is an entropy free inclusion in the sense of Definition~\ref{defn:semifintie zero entropy}.
\end{defn}

\begin{remark}\label{rem: entropy just isn't what it used to be}
If $\tau$ is a faithful normal semifinite tracial weight on $M$, we warn the reader that $Q\leq M$ being an entropy free inclusion is \emph{not} equivalent to being a $\tau$-entropy free inclusion. Indeed, the traciality of $\tau$ implies $M\rtimes_{\sigma^\tau} \bR \cong M\bar\otimes L(\bR)$, and similarly for $Q\rtimes_{\sigma^{\tau}} \bR$. Using the trace $\tau\otimes \tau_{\bR}$ on $M\bar\otimes L(\bR)$, one has
    \[
        h_{\tau\otimes \tau_{\bR}}( Q\bar\otimes L(\bR): M\bar\otimes L(\bR)) \leq 0
    \]
always by \cite[Example 2]{FreePinsker}. Thus $Q \leq M$ is always $\tau$-entropy free, but need not be entropy free in the sense of Definition~\ref{defn:semifintie zero entropy}.
\end{remark}

We have the following monotonicity property of entropy free inclusions.

\begin{prop}\label{prop: nonpositive monotone}
Let $M_1$ be a von Neumann algebra equipped with a faithful normal semifinite weight $\varphi$. Suppose $Q_{1}\leq Q_{2}\leq  M_2\leq M_1$ are inclusions with $\varphi$-expectation. If $Q_{2}\leq M_2$ is a $\varphi$-entropy free inclusion, then so is $Q_{1}\leq M_1$.
\end{prop}
\begin{proof}
Let $\tau_\varphi$ be the faithful normal semifinite tracial weight on $M_1\rtimes_{\sigma^\varphi} \bR$ induced by $\varphi$, and let $(q_i)_{i\in I} \subseteq \dom(\tau_\varphi|_{Q_1\rtimes \bR})$ be a net of projections converging to $1$ in the strong operator topology. Then for each $i\in I$ we have
    \[
        h_{\tau_\varphi}( q_i (Q_1\rtimes_{\sigma^\varphi}\bR) q_i: q_i (M_1\rtimes_{\sigma^\varphi}\bR) q_i)\leq h_{\tau_\varphi}( q_i (Q_2\rtimes_{\sigma^\varphi}\bR) q_i: q_i (M_2 \rtimes_{\sigma^\varphi}\bR) q_i) \leq 0,
    \]
where the first inequality follows \cite[Property P2]{FreePinsker}. Thus Lemma~\ref{lem: independ of the proj} implies $Q_1\leq M_1$ is $\varphi$-entropy free.
\end{proof}

We will obtain many examples of entropy free inclusions from general permanence properties of entropy free inclusions, such as closure under taking (generalizations of) the normalizer or under joins with diffuse intersection (see Section~\ref{sec: Pinsker for weights}). For now let us list one class of examples.

\begin{prop}\label{prop:amenable}
Let $M$ be a von Neumann algebra equipped with a faithful normal semifinite weight $\varphi$. If $Q \leq M$ is with $\varphi$-expectation and $Q$ is amenable, then $Q\leq M$ is a $\varphi$-entropy free inclusion.
\end{prop}
\begin{proof}
It suffices to show that $Q\rtimes_{\sigma^\varphi}\bR$ is amenable, since every corner will therefore be amenable and the result will then follow from \cite[Property P3]{FreePinsker}. Toward this end, denote the dual weight to $\varphi$ by $\widetilde{\varphi}$. 
By \cite[Proposition 2.6]{ADAmenableCorr} there is a (non-normal) conditional expectation from $(J_{\widetilde{\varphi}} Q J_{\widetilde{\varphi}})'\cap B(L^2(Q\rtimes_{\sigma^\varphi} \bR, \widetilde{\varphi}))\to Q\rtimes_{\sigma^\varphi} \bR$. Since $Q$ is injective, so is this commutant, and thus $Q\rtimes_{\sigma^\varphi} \bR$ is amenable.
\end{proof}

\begin{defn}
Let $M$ be a von Neumann algebra equipped with a faithful normal semifinite weight $\varphi$. We say that $P\leq M$ is a \textbf{$\varphi$-Pinsker algebra} if it is a maximal $\varphi$-entropy free inclusion. That is, $P\leq M$ is a $\varphi$-entropy free inclusion, and whenever $Q\leq M$ is a $\varphi$-entropy free inclusion with $Q\geq P$ then $P=Q$.
\end{defn}

\begin{thm}\label{thm:PT_III}
Let $M$ be a von Neumann algebra equipped with a faithful normal semifinite weight $\varphi$. If $Q\leq M$ is a $\varphi$-entropy free inclusion, then there exists a unique $\varphi$-Pinsker algebra $P\leq M$ satisfying $P\geq Q$.
\end{thm}
\begin{proof}
Denote $\widetilde{Q}:=Q\rtimes_{\sigma^\varphi} \bR$ and $\widetilde{M}:=M\rtimes_{\sigma^\varphi} \bR$. Let $\tau_\varphi$ be the faithful normal semifinite tracial weight induced by $\varphi$, and recall from the discussion preceding Lemma~\ref{lem:continuous_core_recognition_lemma} that $\tau_\varphi|_{\widetilde{Q}}$ is semifinite. Since $\widetilde{Q}\leq \widetilde{M}$ is entropy free by assumption and $\widetilde{Q}$ is diffuse by Lemma~\ref{lem:continuous_core_is_diffuse}, Proposition~\ref{prop:PT_II_infty} yields a unique Pinsker algebra $\widetilde{Q} \leq \widetilde{P} \leq \widetilde{M}$.  If $\bR\overset{\theta}{\curvearrowright} \widetilde{M}$ is the action dual to $\sigma^\varphi$, then the uniqueness of $\widetilde{P}$ implies it is $\theta$-invariant. Since we also have $L_\varphi(\bR)\leq \widetilde{Q} \leq \widetilde{P}$, it follows from Lemma~\ref{lem:continuous_core_recognition_lemma} that $\widetilde{P} = P\rtimes_{\sigma^\varphi} \bR$ for $P:=(\widetilde{P})^\theta$. The same lemma gives that $P$ is $\sigma^\varphi$-invariant, and $Q = (\widetilde{Q})^\theta \leq (\widetilde{P})^\theta = P$ implies $\varphi|_P$ is semifinite by Lemma~\ref{lem:semifinite_iff_approx_proj_unit}. Thus $P\leq M$ is with $\varphi$-expectation. We also have that $P\leq M$ is $\varphi$-entropy free since $\widetilde{P}\leq \widetilde{M}$ is entropy free.

We next show that if $N\leq M$ is a $\varphi$-entropy free inclusion with $N\geq Q$, then $N\leq P$. Note that this will prove that $P$ is both a $\varphi$-Pinsker algebra and unique. Set $\widetilde{N}:=N\rtimes_{\sigma^\varphi} \bR$ so that $\widetilde{Q}\leq \widetilde{N}$. Note that $\widetilde{N}$ is diffuse by Lemma~\ref{lem:continuous_core_is_diffuse} and $\widetilde{N} \leq \widetilde{M}$ is entropy free by assumption. Thus Proposition~\ref{prop:PT_II_infty} implies there is a unique Pinsker algebra containing $\widetilde{N}$. This algebra will necessarily also contain $\widetilde{Q}$, and so the uniqueness of $\widetilde{P}$ implies $\widetilde{N} \leq \widetilde{P}$. By considering the fixed points under $\theta$, we obtain $N\leq P$.
\end{proof}

We remark that, in contrast to Proposition~\ref{prop:PT_II_infty}, the previous theorem does not demand that $Q$ be diffuse. Consider, for example, an orthogonal family of projections $\{p_i \in \dom(\varphi|_{M^\varphi})\colon i\in I\}$ that sum to $1$. Then
    \[
        Q:=\{p_i\colon i\in I\}'' \cong \ell^\infty(I)
    \]
is amenable, $\varphi$-invariant, and $\varphi|_Q$ is semifinite. Thus $Q\leq M$ is a $\varphi$-entropy free inclusion by Proposition~\ref{prop:amenable}, which in turn implies there is a unique $\varphi$-Pinsker algebra $P\geq Q$. Consequently, if $N\leq M$ is any other $\varphi$-entropy free inclusion with $p_i\in N$ for all $i\in I$, then we must have $N\leq P$. A special case of this occurs when $\varphi$ is a state and we take $\{p_i\colon i\in I\} = \{1\}$, which yields the following corollary.

\begin{cor}\label{cor:unique_varphi-Pinsker_algebra}
Let $M$ be a von Neumann algebra equipped with a faithful normal state $\varphi$. Then there is exactly one $\varphi$-Pinsker algebra $P\leq M$. Moreover, if $Q\leq M$ is any $\varphi$-entropy free inclusion, then $Q\subseteq P$. 
\end{cor}

\begin{notation}
  In light of the previous results, for a $\varphi$-entropy free inclusion $Q\leq M$ we denote by $\Pi(Q\leq M, \varphi)$ the unique $\varphi$-Pinsker algebra containing $Q$. In the case that $\varphi$ is a state, we write $\Pi(M,\varphi):= \Pi( \bC \leq M,\varphi)$ for the unique $\varphi$-Pinsker algebra in $M$.  
\end{notation}

The utility of the uniqueness of the $\varphi$-Pinsker algebra will be illustrated by establishing many examples of entropy free inclusions. This will be accomplished in the following sections by showing that entropy free inclusions are closed under various operations such as taking normalizers, one-sided quasi-normalizers and so on.

\section{Anticoarse spaces}

In the setting of finite von Neumann algebras, the anticoarse space of an inclusion captures all notions of normalizers (see \cite[Proposition 3.2]{Hayes2018} and \cite[Propositions 2.2 and 2.3]{UselessResolutionOfPT}), and in the case of an entropy free inclusion it is contained in the $L^2$-space of the associated Pinsker algebra (see \cite[Theorem 3.8]{Hayes2018}). Consequently, the Pinsker algebra by way of the anticoarse space offers a noncommutative harmonic analysis approach to establishing generalized solidity results (see, for example, \cite{UselessResolutionOfPT}). 

In this section, we will generalize these results to the semifinite and $\sigma$-finite settings. Our approach will demand two notions of the anticoarse space: one that is \emph{intrinsic} to the inclusion (see Definition~\ref{defn:varphi_anticoarse_space}) and another that is \emph{extrinsic} in the sense that it is defined in terms of the continuous core (see Definition~\ref{defn:dynamical_varphi_anticoarse_space}). We suspect that the two spaces should coincide in many cases, but this remains an open question in general and in the pursuit of our main results the two spaces will be utilized in quite different ways. For example, the extrinsic anticoarse space for an entropy free inclusion is more easily related to the associated Pinsker algebra because both concepts are defined via the continuous core.

\subsection{An intrinsic anticoarse space}

\begin{defn}\label{defn:varphi_anticoarse_space}
Let $M$ be a von Neumann algebra equipped with a faithful normal semifinite weight $\varphi$.  Given an inclusion $N\leq M$ with $\varphi$-expectation, we define its \textbf{$\varphi$-anticoarse space} to be the $N$-subbimodule of $L^2(M,\varphi)$ given by
    \[
        L^{2}_{\cross}(N\leq M,\varphi):= \bigcap \left\{\ker(T)\colon T\in \Hom_{N-N}\left(L^2(M,\varphi), L^2(N,\varphi)\otimes L^2(N,\varphi)\right)\right\}.
    \]
That is, the $\varphi$-anticoarse space consists of those $\xi\in L^{2}(M,\varphi)$ such that $\overline{\Span(N\cdot\xi\cdot N)}^{\|\cdot\|_\varphi}$ is disjoint from $L^{2}(N,\varphi)\otimes L^{2}(N,\varphi)$ as $N$-bimodules.
\end{defn}

We remark that $L^2_{\cross}(N\leq M,\varphi)$ is an invariant subspace of the representation $\bR\ni t\mapsto \Delta_\varphi^{it}$ since
    \[
        T\mapsto (\Delta_\varphi^{it}\otimes \Delta_\varphi^{it})T \Delta_\varphi^{-it}
    \]
defines an invertible, linear isometry on $\Hom_{N-N}\left(L^2(M,\varphi), L^2(N,\varphi)\otimes L^2(N,\varphi)\right)$ for all $t\in \bR$. More generally, we have the following.

\begin{prop}\label{prop: exercise with bimod maps}
For each $j=1,2$, let $M_j$ be a von Neumann algebra equipped with a faithful normal semifinite weight $\varphi_j$ and let $N_j\leq M_j$ be an inclusion with $\varphi_j$-expectation. Suppose that $\lambda,\rho\colon N_1\to N_2$ are a $*$-isomorphisms and that $T\colon L^{2}(M_1,\varphi_{1})\to L^{2}(M_2,\varphi_{2})$ is a bounded linear map satisfying
    \[
        T(x\cdot \xi \cdot y) = \lambda(x)\cdot T(\xi)\cdot \rho(y)
    \]
for all $x,y\in N_1$ and $\xi\in L^2(M_1,\varphi_1)$. Then 
    \[
        T\left(L^{2}_{\cross}(N_1\leq M_1,\varphi_{1})\right)\subseteq L^{2}_{\cross}(N_2\leq M_2,\varphi_{2}),
    \]
with equality if $T$ is a co-isometry.
\end{prop}
\begin{proof}
Fix $S\in \Hom_{N_2-N_2}(L^{2}(M_2,\varphi_{2}),L^{2}(N_2,\varphi_{2})\otimes L^{2}(N_2,\varphi_{2}))$. By \cite[Theorem IX.1.14]{TakesakiII} there exists unitaries $U_\lambda, U_\rho\colon L^2(N_2,\varphi_2)\to L^2(N_1,\varphi_1)$ satisfying
    \begin{align*}
        U_\lambda(x\cdot \xi \cdot y ) &= \lambda^{-1}(x) \cdot U_\lambda(\xi)\cdot \lambda^{-1}(y)\\
        U_\rho(x\cdot \xi \cdot y ) &= \rho^{-1}(x) \cdot U_\rho(\xi)\cdot \rho^{-1}(y)
    \end{align*}
for all $x,y\in N_2$ and $\xi\in L^2(N_2,\varphi_2)$. Then $(U_\lambda \otimes U_\rho)ST \in \Hom_{N_1-N_1}\left( L^{2}(M_1,\varphi_{1}),L^{2}(N_1,\varphi_{1})\otimes L^{2}(N_1,\varphi_{1})\right)$. So if $\xi\in L^{2}_{\cross}(N_1\leq M_1,\varphi_{1})$, then
    \[
        S(T(\xi)) = (U_\lambda\otimes U_\rho)^* (U_\lambda\otimes U_\rho)ST(\xi) = 0.
    \]
Since $S$ was arbitrary, it follows that  $T(L^{2}_{\cross}(N_1\leq M_1,\varphi_{1}))\subseteq L^{2}_{\cross}(N_2\leq M_2,\varphi_{2})$. 

If $T$ is a co-isometry, the applying the above to $T^*$ (and $\lambda^{-1}$ and $\rho^{-1}$) gives
    \[
        T^* \left(L^2_{\cross}(N_2\leq M_2, \varphi_2) \right)  \subseteq L^2_{\cross}(N_1\leq M_1, \varphi_1).
    \]
Combining this with the previously established inclusion we obtain
    \[
        L^2_{\cross}(N_2\leq M_2, \varphi_2) = T\left( T^* \left( L^2_{\cross}(N_2\leq M_2, \varphi_2)\right) \right) \subseteq T\left(L^2_{\cross}(N_1\leq M_1, \varphi_1) \right) \subseteq L^2_{\cross}(N_2\leq M_2, \varphi_2),
    \]
and hence equality holds throughout. 
\end{proof}

\noindent We highlight four important applications of Proposition~\ref{prop: exercise with bimod maps}:
    \begin{itemize}
    \item Suppose $N\leq M$ is with $\varphi_j$-expectation for $j=1,2$, and let $U_{\varphi_2,\varphi_1}\colon L^2(M,\varphi_1)\to L^2(M,\varphi_2)$ be the canonical intertwiner from $\varphi_1$ to $\varphi_2$. Then one has
        \[
            U_{\varphi_2,\varphi_1} \left( L^2_{\cross}(N\leq M,\varphi_1) \right) = L^2_{\cross}(N\leq M,\varphi_2),
        \]
    by considering $T=U_{\varphi_2,\varphi_1}$ and $\lambda=\rho=\text{id}$.
    
    \item Let $N\leq M$ be with $\varphi$-expectation, let $\alpha\in \Aut(M)$ be an automorphism satisfying $\alpha(N)=N$, and let $U_\alpha\in \cU(L^2(M,\varphi))$ be the \emph{unitary implementation} of $\alpha$ (see \cite[Theorem IX.1.15]{TakesakiII}). Then one has
        \[
            U_\alpha \left( L^2_{\cross}(N\leq M,\varphi)\right) = L^2_{\cross}(N\leq M,\varphi).
        \]
    by considering $T=U_\alpha$ and $\lambda=\rho=\text{id}$.

    \item Suppose $N\leq M_2\leq M_1$ are inclusions with $\varphi$-expectation, and let $e_{M_2}\in B(L^2(M_1,\varphi)$ be the projection onto the subspace $L^2(M_2,\varphi)$. Then one has
        \[
            e_{M_2}\left(L^2_{\cross}(N\leq M_1,\varphi)\right) = L^2_{\cross}(N\leq M_2,\varphi),
        \]
    by considering $T=e_{M_2}$ as a co-isometry onto $L^2(M_2,\varphi)$ and letting $\lambda=\rho=\text{id}$.

    \item Let $N\leq M$ be with $\varphi$-expectation.  Then one has
        \[
            \cN_M(N)''\cdot L^2_{\cross}(N\leq M,\varphi) \cdot \cN_M(N)'' = L^2_{\cross}(N\leq M,\varphi),
        \]
    by considering $T(\xi):=u\cdot\xi\cdot v$, $\lambda(x):=uxu^*$, and $\rho(y):=v^*xv$ for $u,v\in \cN_M(N)$.
    \end{itemize}
The $\varphi$-anticoarse space is also well behaved under compressions by projections in $N^\varphi$ as the following proposition demonstrates.

\begin{prop}\label{prop:compressions_varphi_anticoarse_space}
Let $M$ be a von Neumann algebra equipped with a faithful normal semifinite weight $\varphi$ and let $N\leq M$ be an inclusion with $\varphi$-expectation. For any projection $p\in N^\varphi$, one has
    \[
        p\cdot L^2_{\cross}(N\leq M,\varphi)\cdot p = L^2_{\cross}(pNp \leq pMp, \varphi|_{pMp}).
    \]
\end{prop}
\begin{proof}
Fix a projection $p\in N^\varphi$. Observe that $p$ being in the centralizer of $\varphi$ implies it is a multiplier for both $\dom(\varphi|_N)$ and $\dom(\varphi)$. Consequently, $\varphi$ is semifinite on both $pNp$ and $pMp$ and we may view $L^2(pNp,\varphi)$ and $L^2(pMp,\varphi)$ as subspaces of $L^2(N,\varphi)$ and $L^2(M,\varphi)$, respectively.

Now, let $z\in N'\cap N$ be the central support of $p$ in $N$. Then there exists a family of partial isometries $\{v_i\in N\colon i\in I\}$ satisfying $v_i v_i^* \leq p$ and $\sum_i v_i^* v_i = z$, and we may assume $p=v_i$ for some $i\in I$. Fix $T\in \Hom_{pNp-pNp}(L^2(pMp,\varphi), L^2(pNp,\varphi)\otimes L^2(pNp,\varphi))$. Using our initial observation, 
    \[
       S(\xi):= \sum_{i,j\in I} v_i^*\cdot T(v_i\cdot \xi \cdot v_j^*) \cdot v_j
    \]
defines a bounded linear map $S\colon L^2(M,\varphi)\to L^2(N,\varphi)\otimes L^2(N,\varphi)$ satisfying $S(p\cdot \xi \cdot p) = T(p\cdot\xi\cdot p)$. A routine computation shows that $S$ is $N$-bilinear, and thus if $\xi \in L^2_{\cross}(N\leq M,\varphi)$ then
    \[
        T(p\cdot\xi\cdot p) = S(p\cdot \xi \cdot p) = p\cdot S(\xi)\cdot p = 0.
    \]
This establishes the inclusion $p\cdot L^2_{\cross}(N\leq M,\varphi)\cdot p \subseteq L^2_{\cross}(pNp\leq pMp,\varphi|_{pMp})$.

Conversely, fix $\xi = p\cdot \xi\cdot p\in L^2_{\cross}(pNp\leq pMp,\varphi|_{pMp})$. Note that $L^2(N,\varphi)\otimes L^2(N,\varphi)$ is a normal $pNp \bar{\otimes} (pNp)^{\text{op}}$-module and therefore embeds into an infinite direct sum of $L^2(pNp,\varphi)\otimes L^2(pNp,\varphi)$. This along with our initial observation implies that $S(\xi)=0$ for all $S\in \Hom_{N-N}(L^2(M,\varphi), L^2(N,\varphi)\otimes L^2(N,\varphi))$. Hence $\xi = p\cdot \xi \cdot p\in p\cdot L^2_{\cross}(N\leq M,\varphi)\cdot p$.
\end{proof}

The following proposition gives one a sense of the kind of elements contained in a $\varphi$-anticoarse space. In particular, it provides sufficient conditions for a $\varphi$-anticoarse space over $N$ to contain $L^2(N,\varphi)$.

\begin{prop}\label{prop:diffuse_algebras_contained_in_their_anticoarse_space}
Let $M$ be a von Neumann algebra equipped with a faithful normal semifinite weight $\varphi$.  Let $N$ be a diffuse $\sigma$-finite subalgebra of $M$ with $\varphi$-expectation. Then $\cN_M(N)'' \leq M$ is with $\varphi$-expectation and
    \[
        L^2( \cN_M(N)'', \varphi) \subseteq L^{2}_{\cross}(N\leq M,\varphi).
    \]

\end{prop}
\begin{proof}
Using Lemma~\ref{lem:semifinite_iff_approx_proj_unit}, let $(p_i)_{i\in I}\subseteq \dom(\varphi|_N)$ be a net of projections converging to $1$ in the strong operator topology. Then the same lemma implies $\varphi$ is semifinite on $\cN_M(N)''$ since it contains $N$, and moreover the $\sigma^\varphi$-invariance of $N$ implies that of $\cN_M(N)''$. Thus $\cN_M(N)''\leq M$ is with $\varphi$-expectation. Next, observe that
    \[
        L^2(\cN_M(N)'',\varphi) = \overline{\cN_M(N)''\cdot L^2(N,\varphi)},
    \]
since for $a\in\sqrt{\dom}(\varphi|_{\cN_M(N)''})$ we have
    \[
        a = \lim_{i\to\infty}  J_\varphi p_i J_\varphi a = \lim_{i\to\infty} a\cdot (J_\varphi p_i),
    \]
where the second equality follows from \cite[Equation VI.1.(42)]{TakesakiII} applied with $\eta=J_\varphi p_i$. Recalling from the discussion following Proposition~\ref{prop: exercise with bimod maps} that the $\varphi$-anticoarse space is an $\cN_M(N)''$-bimodule, we see that it now suffices to show $L^2(N,\varphi)\subseteq L^2_{\cross}(N\leq M,\varphi)$.

Using our assumptions on $N$, \cite[Theorem 11.1]{HSFlow} (see also \cite[Lemma 2.1]{HUAsymptot} and \cite{CSStateSpace}) yields a faithful normal state $\psi$ on $N$ so that $N^\psi$ is diffuse. Since $N^\psi$ is a finite von Neumann algebra, this implies in particular that we can find a sequence of unitaries $(u_n)_{n\in \bN} \subseteq N^\psi$ converging to 0 in the weak operator topology. Thus if $T\in \Hom_{N-N}(L^2(N,\psi), L^2(N,\psi) \otimes L^2(N,\psi))$, then
    \[
        T(1) = T(u_n\cdot 1 \cdot u_n^*) = u_n\cdot T(1) \cdot u_n^*,
    \]
and the above converges to 0 by coarseness. Hence $T(L^2(N,\psi))= \overline{N\cdot T(1)}=0$, which shows $L^2(N,\psi) \subseteq L^2_{\cross}(N\leq N,\psi)$. From the discussion following Proposition~\ref{prop: exercise with bimod maps} we then have
    \begin{align*}
        L^2(N,\varphi) &= U_{\varphi,\psi}\left(L^2(N,\psi) \right) \subseteq U_{\varphi,\psi}\left(L^2_{\cross}(N\leq N,\psi) \right)\\
        &= L^2_{\cross}(N\leq N,\varphi) = e_N\left(L^2_{\cross}(N\leq M,\varphi) \right) \subseteq L^2_{\cross}(N\leq M,\varphi),
    \end{align*}
where $U_{\varphi,\psi}$ is the canonical intertwiner from $\psi$ to $\varphi|_N$ and $e_N$ is the projection onto the subspace $L^2(N,\varphi) \leq L^2(M,\varphi)$.
\end{proof}

We next consider the case of semifinite von Neumann algebras. Given a faithful normal semifinite tracial weight $\tau$ on $M$, recall from \cite[Theorem VIII.3.14]{TakesakiII} that every faithful normal semifinite weight $\varphi$ on $M$ is given uniquely by
    \begin{align}\label{eqn:semifinite_RD_formula}
        \varphi(x) = \tau(x h ) = \lim_{\epsilon\searrow 0} \tau\left( x\frac{h}{1+\epsilon h} \right) \qquad \qquad x\in M_+,
    \end{align}
for some positive, self-adjoint, non-singular operator $h$ affiliated with $M$. We will denote $\frac{d \varphi}{d\tau} :=h$ and $\frac{d\tau}{d\varphi} := h^{-1}$. Further recall from the proof of \cite[Theorem VIII.3.14]{TakesakiII} that
    \[
        \sigma_t^\varphi(x) = \left(\frac{d \varphi}{d\tau}\right)^{it} x \left( \frac{d \varphi}{d\tau}\right)^{-it}
    \]
for all $x\in M$ and $t\in \bR$. That is, $(\frac{d\varphi}{d\tau})^{it}$ equals the Connes cocycle derivative $(D\varphi\colon D\tau)_t$ (see \cite[Definition VIII.3.4]{TakesakiII}). Moreover, $\frac{d \varphi}{d\tau}$ is affiliated with $M^\varphi$.

\begin{lem}\label{lem:semifinite_expectation_swapping}
Let $M$ be a semifinite von Neumann algebra equipped with a faithful normal semifinite tracial weight $\tau$ and a faithful normal semifinite weight $\varphi$. If $\frac{d\varphi}{d\tau}$ is affiliated with a subalgebra $N \leq M$, then this inclusion is with $\tau$-expectation if and only if it is with  $\varphi$-expectation. In this case, there exists a faithful normal conditional expectation $\cE_N\colon M\to N$ that preserves both $\tau$ and $\varphi$.
\end{lem}
\begin{proof}
The affiliation of $\frac{d\varphi}{d \tau}$ with $N$ and (\ref{eqn:semifinite_RD_formula}) implies a $\tau$-preserving expectation will also be $\varphi$-preserving, and vice versa.
\end{proof}

\begin{prop}\label{prop:semifinite_anticoarse_space_containment_switch}
Let $M$ be a semifinite von Neumann algebra equipped with a faithful normal semifinite tracial weight $\tau$ and a faithful normal semifinite weight $\varphi$. Suppose $N\leq M$ is an inclusion with both $\tau$-expectation and $\varphi$-expectation and that $N$ is diffuse and $\sigma$-finite. Further suppose $\frac{d \varphi}{d\tau}$ is affiliated with $\cN_M(N)''$. 
Then $P\leq M$ is with $\tau$-expectation and satisfies
    \[
        L^2_{\cross}(N\leq M, \tau) \subseteq L^2(P,\tau)
    \]
if and only if the inclusion is with $\varphi$-expectation and satisfies
    \[
        L^2_{\cross}(N\leq M, \varphi) \subseteq L^2(P,\varphi).
    \]
\end{prop}
\begin{proof}
Assume $P\leq M$ is with $\tau$-expectation and that $L^2(P,\tau)$ contains the $\tau$-anticoarse space of $N\leq M$ (the proof of the converse is identical). Proposition~\ref{prop:diffuse_algebras_contained_in_their_anticoarse_space} implies
    \[
        L^2(\cN_M(N)'', \tau)\subseteq L^2_{\cross}(N\leq M, \tau) \subseteq L^2(P,\tau),
    \]
so that $\frac{d\varphi}{d \tau}$ is affiliated with $P$. Lemma~\ref{lem:semifinite_expectation_swapping} then yields a faithful normal conditional expectation $\cE_P\colon M\to P$ that preserves both $\tau$ and $\varphi$. Applying Proposition~\ref{prop: exercise with bimod maps} and Lemma~\ref{lem:common_expectation_canonical_intertwiner} gives
    \[
        L^2_{\cross}(N\leq M,\varphi) = U_{\varphi,\tau}\left(L^2_{\cross}(N\leq M,\tau)\right) \subseteq U_{\varphi,\tau}\left( L^2(P,\tau) \right) = L^2(P,\varphi),
    \] 
where $U_{\varphi,\tau}\colon L^2(M,\tau)\to L^2(M,\varphi)$ is the canonical intertwiner from $\tau$ to $\varphi$.
\end{proof}

If $S$ is a set of operators affiliated with $M$ then we define the von Neumann algebra generated by $S$ to be the subalgebra of $M$ generated by $v$ and all spectral projections of $|x|$, where $x\in S$ with polar decomposition $x=v|x|$. Note that in the context of the following corollary, $L^2(M,\tau_j)$ for $j=1,2$ can be viewed as affiliated operators (see \cite[Theorems IX.2.5 and IX.2.13]{TakesakiII}).

\begin{cor}\label{cor:anticoarse_space_same_over_traces}
Let $M$ be a semifinite von Neumann algebra equipped with faithful normal semifinite tracial weights $\tau_1$ and $\tau_2$. Suppose $N\leq M$ is with both $\tau_1$-expectation and $\tau_2$-expectation and $N$ is diffuse and $\sigma$-finite. Then the $\tau_j$-anticoarse spaces for $j=1,2$ generate the same von Neumann subalgebra of $M$.
\end{cor}
\begin{proof}
Note that $\frac{d \tau_2}{d \tau_1}$ is affiliated with the center of $M$, which is contained in $\cN_M(N)''$. Furthermore, for each $j=1,2$, $\tau_j$ is semifinite on the von Neumann algebra generated by the $\tau_j$-anticoarse space because the algebra contains $N$ by Proposition~\ref{prop:diffuse_algebras_contained_in_their_anticoarse_space}. The corollary then follows from Proposition~\ref{prop:semifinite_anticoarse_space_containment_switch}.
\end{proof}

\subsection{An extrinsic anticoarse space}

Recall that for the weight $\widetilde{\varphi}$ on $M\rtimes_{\sigma^\varphi} \bR$ dual to a faithful normal semifinite weight $\varphi$ on $M$, one has $\lambda_\varphi(f) x\mapsto f\otimes x$ for $f\in C_c(\bR)$ and $x\in \sqrt{\dom}(\varphi)$ extends to an $M\rtimes_{\sigma^{\varphi}}\bR$-$M\rtimes_{\sigma^{\varphi}}\bR$ bimodular unitary
    \[
        L^2(M\rtimes_{\sigma^\varphi} \bR, \widetilde{\varphi}) \cong L^2(\bR)\otimes L^2(M,\varphi).
    \]
We will identify these two spaces in this section.

\begin{defn}\label{defn:dynamical_varphi_anticoarse_space}
Let $M$ be a von Neumann algebra equipped with a faithful normal semifinite weight $\varphi$. Given an inclusion $N\leq M$ with $\varphi$-expectation, we define its \textbf{dynamical $\varphi$-anticoarse space}
to be the set
    \[
        L^{2,\textnormal{dyn}}_{\cross}(N\leq M, \varphi):=\{\xi\in L^2(M,\varphi)\colon f\otimes \xi\in L^2_{\cross}(N\rtimes_{\sigma^\varphi} \bR \leq M\rtimes_{\sigma^\varphi} M, \widetilde{\varphi})\ \textnormal{ for every  } f\in L^2(\bR)\},
    \]
where $\widetilde{\varphi}$ is the dual weight of $\varphi$.
\end{defn}

\noindent The following proposition shows the dynamical $\varphi$-anticoarse spaces arise quite naturally from the anticoarse spaces considered in the previous subsection.

\begin{prop}\label{prop: tensor splitting}
 Let $M$ be a von Neumann algebra equipped with a faithful normal semifinite weight $\varphi$, and let $N\leq M$ be with $\varphi$-expectation. Then
 \[L^{2}_{\cross}(N\rtimes_{\sigma^{\varphi}}\bR\leq M\rtimes_{\sigma^{\varphi}}\bR,\widetilde{\varphi})=L^{2}(\bR)\otimes L^{2,\textnormal{dyn}}_{\cross}(N\leq M,\varphi).\]
\end{prop}
\begin{proof}
Let $p\in B(L^{2}(\bR)\otimes L^{2}(M,\varphi))$ be the projection onto $L^{2}_{\cross}(N\rtimes_{\sigma^{\varphi}}\bR\leq M\rtimes_{\sigma^{\varphi}}\bR,\widetilde{\varphi})$. Since $L^{2}_{\cross}(N\rtimes_{\sigma^{\varphi}}\bR\leq M\rtimes_{\sigma^{\varphi}}\bR,\widetilde{\varphi})$ is an $L_{\varphi}(\bR)$-bimodule, we have that $p$ commutes with $\lambda(t)\otimes 1$ for all $t\in \bR$.
Recall the dual action $\bR\overset{\theta}{\curvearrowright}M\rtimes_{\sigma^{\varphi}}\bR$ given by (\ref{eqn:dual_action_formula}) and note that $\theta_{s}$ has unitary implementation $u(s)\otimes 1$ where $[u(s)f](t)=e^{-2\pi i st}f(t)$ for all $t\in \bR$ and $f\in L^{2}(\bR)$. The $\theta$-invariance of $N\rtimes_{\sigma^\varphi}\bR$ and the discussion following Proposition~\ref{prop: exercise with bimod maps} implies $L^{2}_{\cross}(N\rtimes_{\sigma^{\varphi}}\bR\leq M\rtimes_{\sigma^{\varphi}}\bR,\widetilde{\varphi})$ is invariant under $\{u(s)\otimes 1:s\in \bR\}$. Hence, by Tomita's commutation theorem, we have that 
    \[
        p\in (\{\lambda(t):t\in \bR\}\cup \{u(s)\otimes 1:s\in \bR\})'\overline{\otimes}B(L^{2}(M,\varphi)).
    \]
It is well known that the covariant representations $\lambda$ and $u$ together generate $B(L^2(\bR))$ (see, for example, \cite[Proposition X.2.2]{TakesakiII}), and therefore $p=1\otimes q$ for some projection $q\in B(L^{2}(M,\varphi))$. Thus we may write
\[L^{2}_{\cross}(N\rtimes_{\sigma^{\varphi}}\bR\leq M\rtimes_{\sigma^{\varphi}}\bR,\widetilde{\varphi})=L^{2}(\bR)\otimes \cH\]
for some $\cH\subseteq L^{2}(M,\varphi)$. Having established the above tensor splitting, it follows from the definition of the dynamical $\varphi$-anticoarse space that $\cH=L^{2,\textnormal{dyn}}_{\cross}(N\leq M,\varphi)$. 
\end{proof}

\begin{cor}\label{cor: extrinsic anticoarse space as a bimodule}
Let $M$ be a von Neumann algebra equipped with a faithful normal semifinite weight $\varphi$, and let $N\leq M$ be with $\varphi$-expectation. Then $L^{2,\textnormal{dyn}}_{\cross}(N\leq M,\varphi)$ is a closed $N$-subbimodule of $L^{2}(M,\varphi)$ which is invariant for the representation $\bR\ni t\mapsto \Delta_{\varphi}^{it}$.
\end{cor}

\begin{proof}
Let $\xi\in L^{2,\textnormal{dyn}}_{\cross}(N\leq M,\varphi)$, and fix an $f\in L^{2}(\bR)$ which is nowhere vanishing. For any $t\in \bR$, we have that 
    \[
        f\otimes \Delta_{\varphi}^{it}(\xi)=\lambda_{\varphi}(t)\cdot(f\otimes \xi)\cdot\lambda_{\varphi}(-t)\in L^{2}_{\cross}(N\rtimes_{\sigma^{\varphi}}\bR\leq M\rtimes_{\sigma^{\varphi}}\bR,\widetilde{\varphi})=L^{2}(\bR)\otimes L^{2,\textnormal{dyn}}_{\cross}(N\leq M,\varphi),
    \]
where in the second-to-last equality we use that $L^{2}_{\cross}(N\rtimes_{\sigma^{\varphi}}\bR\leq M\rtimes_{\sigma^{\varphi}}\bR,\widetilde{\varphi})$ is an $L_{\varphi}(\bR)$-bimodule, and in the last equality we use Proposition \ref{prop: tensor splitting}.
This shows that $L^{2,\textnormal{dyn}}_{\cross}(N\leq M,\varphi)$ is invariant for the representation $t\mapsto \Delta_\varphi^{it}$. The fact that $L^{2,\textnormal{dyn}}_{\cross}(N\leq M,\varphi)$ is right $N$-invariant is proved similarly. Toward seeing that  $L^{2,\textnormal{dyn}}_{\cross}(N\leq M,\varphi)$ is left $N$-invariant, fix $x\in N$. Then by Proposition~\ref{prop: tensor splitting} and the fact that $ L^{2}_{\cross}(N\rtimes_{\sigma^{\varphi}}\bR\leq M\rtimes_{\sigma^{\varphi}}\bR,\widetilde{\varphi})$ is an $N$-bimodule,  we have that $x\cdot (f\otimes \xi)\in L^{2}_{\cross}(N\rtimes_{\sigma^{\varphi}}\bR\leq M\rtimes_{\sigma^{\varphi}}\bR,\widetilde{\varphi})=L^{2}(\bR)\otimes L^{2,\textnormal{dyn}}_{\cross}(N\leq M,\varphi).$ This this implies (even if $L^{2,\textnormal{dyn}}_{\cross}(N\leq M,\varphi)$ is not separable), 
there is a weakly measurable $\eta\colon\bR\to  L^{2,\textnormal{dyn}}_{\cross}(N\leq M, \varphi)$ so that $\eta=x(f\otimes \xi)$ almost everywhere \cite[Proposition IV.7.4]{TakesakiI}. Thus for almost every $r\in \bR$ we have that $[x(f\otimes \xi)](r)\in  L^{2}_{\cross}(N\rtimes_{\sigma^{\varphi}}\bR\leq M\rtimes_{\sigma^{\varphi}}\bR,\widetilde{\varphi})$. Hence, for almost every
$r\in \bR$:
\[f(r)\sigma^{\varphi}_{-r}(x)\cdot \xi=x(f\otimes \xi)(r)\in  L^{2}_{\cross}(N\rtimes_{\sigma^{\varphi}}\bR\leq M\rtimes_{\sigma^{\varphi}}\bR,\widetilde{\varphi}).\]
Since $f$ is nowhere vanishing, for almost every $r\in \bR$ we have that $\sigma^\varphi_{-r}(x)\cdot \xi \in L^{2,\textnormal{dyn}}_{\cross}(N\leq M,\varphi)$. By point strong continuity of $\sigma^{\varphi}$, we have that $\{r\in \bR:\sigma_{-r}^\varphi(x)\cdot \xi \in L^{2,\textnormal{dyn}}_{\cross}(N\leq M,\varphi)\}$ is a conull closed subset of $\bR$; that is, $\sigma^\varphi_{-r}(x)\cdot \xi \in L^{2,\textnormal{dyn}}_{\cross}(N\leq M,\varphi)$ for all $r\in \bR$. Setting $r=0$ yields the left $N$-invariance. 
\end{proof}

The following theorem establishes the relationship between dynamical $\varphi$-anticoarse spaces for $\varphi$-Pinsker algebras. This will allow us to detect elements of $\varphi$-Pinsker algebras by first witnessing them as elements of dynamical $\varphi$-anticoarse spaces. The more concrete definition of the latter makes this an easier task.

\begin{thm}\label{thm:dynamical_anticoarse_contained_in_Pinsker}
Let $M$ be a von Neumann algebra equipped with a faithful normal semifinite weight $\varphi$. If $Q\leq M$ is a $\varphi$-entropy free inclusion with $Q$ $\sigma$-finite, then
    \[
        L^{2,\textnormal{dyn}}_{\cross}(Q\leq M,\varphi) \subseteq L^2( \Pi(Q\leq M,\varphi), \varphi).
    \]
Consequently, $M\cap L^{2,\textnormal{dyn}}_{\cross}(Q\leq M,\varphi) \subseteq \Pi(Q\leq M,\varphi)$.
\end{thm}
\begin{proof}
Let $\tau_\varphi$ be the tracial weight induced by $\varphi$ and define
    \[
        P:=W^*(L^2_{\cross}(Q\rtimes_{\sigma^\varphi} \bR \leq M\rtimes_{\sigma^\varphi} \bR, \tau_\varphi)).
    \]
Note that \cite[Theorem 4.16]{Hayes2018} implies that implies that $P\leq M\rtimes_{\sigma^\varphi} \bR$ is an entropy free inclusion. We claim that $Q\rtimes_{\sigma^\varphi} \bR \subseteq P$ and that $P$ is invariant under the dual action $\bR \overset{\theta}{\curvearrowright} (M\rtimes_{\sigma^\varphi} \bR)$ given by (\ref{eqn:dual_action_formula}). The former can be seen by recalling that $Q\rtimes_{\sigma^\varphi} \bR$ is diffuse by Lemma~\ref{lem:continuous_core_is_diffuse}, and then applying Proposition~\ref{prop:diffuse_algebras_contained_in_their_anticoarse_space}. Note $Q$ being $\sigma$-finite implies $Q\rtimes_{\sigma^\varphi} \bR$ is $\sigma$-finite: identify $Q\rtimes_{\sigma^\varphi}\bR \cong Q\rtimes_{\sigma^\psi} \bR$ where $\psi$ is a faithful normal state on $Q$, and then use that $L_\psi(\bR)$ is $\sigma$-finite and with $\widetilde{\psi}$-expectation. The $\theta$-invariance of $P$ can be seen as follows. Recall that $\tau_\varphi\circ \theta_s = e^{-s}\tau_\varphi$ for all $s\in \bR$, and consequently $\sqrt{\dom}(\tau_\varphi)\ni x\mapsto e^{s/2}\theta_s(x)$ extends to the unitary implementation $U_s\in \cU(L^2(M\rtimes_{\sigma^\varphi} \bR,\tau_\varphi))$ of $\theta_s$. Proposition~\ref{prop: exercise with bimod maps} and the discussion following it imply
    \[
        U_s\left( L^2_{\cross}(Q\rtimes_{\sigma^\varphi}\bR \leq M\rtimes_{\sigma^\varphi}\bR, \tau_\varphi) \right)=L^2_{\cross}(Q\rtimes_{\sigma^\varphi}\bR \leq M\rtimes_{\sigma^\varphi}\bR, \tau_\varphi),
    \]
for all $s\in \bR$, from which the $\theta$-invariance of $P$ follows.

With the claim in hand, Lemma~\ref{lem:continuous_core_recognition_lemma} yields that $P = P^{\theta}\rtimes_{\sigma^\varphi}\bR$, $P^\theta \leq M$ is with $\varphi$-expectation, and $Q \subseteq P^\theta$. In fact, $P^\theta \leq M$ is a $\varphi$-entropy free inclusion because, as we saw above, $P\leq M\rtimes_{\sigma^\varphi} \bR$ is an entropy free inclusion, and so Theorem~\ref{thm:PT_III}  implies that
    \[
        \Pi(Q\leq M,\varphi) = \Pi(P^\theta \leq M. \varphi).
    \]
In particular, we have $L^2(P^\theta,\varphi) \subseteq L^2( \Pi(Q\leq M,\varphi), \varphi)$.

Now, $P=P^\theta\rtimes_{\sigma^\varphi}\bR \leq M\rtimes_{\sigma^\varphi}\bR$ admits a faithful normal conditional expectation that preserves both $\widetilde{\varphi}$ and $\tau_\varphi$. So, combining Proposition~\ref{prop: exercise with bimod maps}, the definition of $P$, and Lemma~\ref{lem:common_expectation_canonical_intertwiner} gives
    \[
        L^2_{\cross}(Q\rtimes_{\sigma^\varphi} \bR \leq M\rtimes_{\sigma^\varphi} \bR, \widetilde{\varphi}) = U_{\widetilde{\varphi},\tau_\varphi} \left( L^2_{\cross}(Q\rtimes_{\sigma^\varphi} \bR \leq M\rtimes_{\sigma^\varphi} \bR, \tau_\varphi)\right) 
        \subseteq U_{\widetilde{\varphi},\tau_\varphi}L^2(P, \tau_\varphi) = L^2(P, \widetilde{\varphi}).
    \]
Thus
    \begin{align*}
        L^2(\bR)\otimes L^{2,\textnormal{dyn}}_{\cross}(Q\leq M,\varphi) &\subseteq L^2_{\cross}(Q\rtimes_{\sigma^\varphi} \bR \leq M\rtimes_{\sigma^\varphi} \bR, \widetilde{\varphi})\\
        &\subseteq L^2(P, \widetilde{\varphi}) = L^2(\bR)\otimes L^2(P^\theta, \varphi) \subseteq L^2(\bR)\otimes L^2(\Pi(Q\leq M,\varphi),\varphi),
    \end{align*}
where the first inclusion follows from the definition of the dynamical $\varphi$-anticoarse space.
\end{proof}

\begin{cor}\label{cor:existence_of_vNa_generated_by_dynamical}
Let $M$ be a von Neumann algebra equipped with a faithful normal semifinite weight $\varphi$. If $N\leq M$ is $\sigma$-finite and with $\varphi$-expectation, then there exists a unique smallest $P\leq M$ with $\varphi$-expectation satisfying
    \[
        L^{2,\textnormal{dyn}}_{\cross}(N\leq M, \varphi) \subseteq L^2(P,\varphi).
    \]    
\end{cor}
\begin{proof}
It suffices to prove the existence of such a $P$ because the uniqueness will follow from being the smallest such algebra.

Arguing as in the proof of Theorem~\ref{thm:dynamical_anticoarse_contained_in_Pinsker}, we obtain that the von Neumann algebra generated by
    \[
        L^2_{\cross}(N\rtimes_{\sigma^\varphi} \bR \leq M \rtimes_{\sigma^\varphi} \bR, \tau_\varphi),
    \]
is of the form $P\rtimes_{\sigma^\varphi}\bR$, where $P\leq M$ is with $\varphi$-expectation. Note that $N\rtimes_{\sigma^\varphi}\bR$ is diffuse by Lemma~\ref{lem:continuous_core_is_diffuse}, $\sigma$-finite by the same argument as in Theorem~\ref{thm:PT_III}, and that $\frac{d\widetilde{\varphi}}{d\tau_\varphi}$ is affiliated with $L_\varphi(\bR)$. Thus we can apply Proposition~\ref{prop:semifinite_anticoarse_space_containment_switch} to see that
    \[
        L^2_{\cross}(N\rtimes_{\sigma^\varphi}\bR \leq M\rtimes_{\sigma^\varphi}\bR, \widetilde{\varphi}) \subseteq L^2(P\rtimes_{\sigma^\varphi} \bR, \widetilde{\varphi}) = L^2(\bR)\otimes L^2(P,\varphi).
    \]
This inclusion in conjunction with Proposition~\ref{prop: tensor splitting} implies  $L^2(P,\varphi)$ contains the dynamical $\varphi$-anticoarse space of $N\leq M$.

Now, suppose $Q\leq M$ is another inclusion with $\varphi$-expectation such that $L^2(Q,\varphi)$ contains the dynamical $\varphi$-anticoarse space of $N\leq M$. Then Proposition~\ref{prop: tensor splitting} implies that
    \[
         L^2_{\cross}(N\rtimes_{\sigma^\varphi}\bR \leq M\rtimes_{\sigma^\varphi}\bR, \widetilde{\varphi}) = L^2(\bR)\otimes L^{2,\textnormal{dyn}}_{\cross}(N\leq M,\varphi) \subseteq L^2(\bR)\otimes L^2(Q,\varphi) = L^2(Q\rtimes_{\sigma^\varphi} \bR,\widetilde{\varphi}).
    \]
Proposition~\ref{prop:semifinite_anticoarse_space_containment_switch} then yields
    \[
        L^2_{\cross}(N\rtimes_{\sigma^\varphi}\bR \leq M\rtimes_{\sigma^\varphi}\bR, \tau_\varphi) \subseteq L^2(Q\rtimes_{\sigma^\varphi}\bR, \tau_\varphi),
    \]
and hence $P\rtimes_{\sigma^\varphi} \bR \subseteq Q\rtimes_{\sigma^\varphi} \bR$. By considering fixed points of the action dual to $\sigma^\varphi$, we see that $P\subseteq Q$.
\end{proof}

In light of the previous corollary, we make the following definition.

\begin{defn}\label{defn:vNa_generated_by_dynamical_anticoarse_space}
Let $M$ be a von Neumann algebra equipped with a faithful normal semifinite weight $\varphi$, and let $N\leq M$ be $\sigma$-finite and with $\varphi$-expectation. We define the von Neumann algebra \textbf{generated by $L^{2,\textnormal{dyn}}_{\cross}(N\leq M,\varphi)$} to be the smallest von Neumann subalgebra $P\leq M$ with $\varphi$-expectation and satisfying $L^{2,\textnormal{dyn}}_{\cross}(N\leq M,\varphi) \subseteq L^2(P,\varphi)$, and we denote it by $W^*(L^{2,\textnormal{dyn}}_{\cross}(N\leq M,\varphi))$.
\end{defn}

\subsection{Canonicity of Pinsker algebras}\label{sec: Pinsker for weights}

In this subsection, we will show that the notion of being an entropy free inclusion is independent of the choice of faithful normal state on the ambient algebra which leaves the subalgebra invariant. This will ultimately show that Pinsker algebras associated to two different states agree if they have a common diffuse subalgebra which is invariant under each state. In particular, any entropy free inclusion has a canonical Pinsker algebra associated to it, which is state independent. These results will be crucial for our future applications, as it will allow us to deftly work with arbitrary subalgebras which are with expectation and not force the expectation to be compatible with whatever state we started with. For example, we will be able to quickly establish that Pinsker algebras absorb all wq-normalizers of subalgebras $Q$ by, for instance, working with an arbitrary state with leaves invariant a diffuse subalgebra in $uQu^{*}\cap Q$, where $u$ is in the wq-normalizer of $Q$. Similarly, our ultrapower applications will also heavily use the canonicity of the Pinsker algebra. 

Recall that the Connes cocycle derivative theorem says the modular automorphism groups associated to two different states differ by a path of inner automorphisms. If a subalgebra is invariant under the modular automorphism of two different states, then the unitaries implementing these inner automorphisms must be in the normalizer of this subalgebra. Because of this, our proof of state independence of the Pinsker algebra will reduce to proving that Pinsker algebras absorb the normalizer of any diffuse subalgebra which is with expectation. We will in fact prove something stronger with the normalizer replaced with the one-sided quasi-normalizer.

After passing to the continuous core, the one-sided quasi-normalizer becomes an even weaker generalization of the normalizer than the one-sided quasi-normalizer. In order to handle this  weaker normalizer, we will need the following general lemma. 
Recall that if $(M,\tau)$ is a tracial von Neumann algebra and $Q,N\leq M$, then we say that \emph{no corner of $Q$ intertwines into $N$ inside $M$} and write $Q\nprec_{M}N$ if there is a net $u_{n}\in \cU(Q)$ so that $\|\cE_{N}(au_{n}b)\|_{2}\to_{n}0$ for all $a,b\in M$ (here $\cE_{N}$ is the unique trace-preserving conditional expectation $M\to N$). We refer to \cite{PopaStrongRigidity} for a detailed discussion, including other equivalent definitions.

\begin{lem}\label{lem:not quite defn of intertwining?}
Let $(M,\tau)$ be a tracial von Neumann algebra, and let $Q,N\leq M$ be such that $Q\nprec_{M}N$.  Let $(B,\Tr)$ be a semifinite tracial von Neumann algebra. Suppose that $\cH_{0}$ is a normal $M\overline{\otimes}B$-module.
If $\cH\subseteq\cH_{0}$ is a $Q\overline{\otimes}B$-submodule contained in a  $N\overline{\otimes}B$-submodule of $\cH_{0}$ which is finite-dimensional over $N\overline{\otimes}B$, then $\cH=0$.
\end{lem}

\begin{proof}
Writing $\cH_{0}$ as a direct sum of normal, cyclic representations and applying \cite[Theorem V.3.15]{TakesakiI}, we may find some set $J$ so that   $\cH_{0}$ isometrically embeds into a $[L^{2}(M\overline{\otimes}B)]^{\oplus J}$ as an $M\overline{\otimes}B$-module. So we may, and will, assume that $\cH_{0}=L^{2}(M\overline{\otimes}B)^{\oplus J}$ for some set $J$. For $j\in J$, the projections $\pi_{j}\colon L^{2}(M\overline{\otimes}B)^{\oplus J}\to L^{2}(M\overline{\otimes}B)$ given by $\pi_{j}(\xi)=\xi(j)$ are $M\overline{\otimes}B$-modular, so it suffices to show that $\pi_{j}(\cH)=\{0\}$ for each $j\in J$. Since Murray--von Neumann dimension over $N\overline{\otimes}B$ is decreasing under bounded $N\overline{\otimes}B$-linear maps with dense image, replacing $\cH$ with $\overline{\pi_{j}(\cH)}$ we may (and will) assume that $\cH_{0}=L^{2}(M\overline{\otimes}B)$. 

Let $\cK$ be a $N\overline{\otimes}B$-submodule of $L^{2}(M\overline{\otimes}B)$ which is finite dimensional over $N\overline{\otimes}B$ and with $\cH\subseteq \cK$. Since $\cK$ has finite dimension over $N\overline{\otimes}B$, we may find vectors $(\xi_{j})_{j=1}^{\infty}\subseteq \cK$ and projections $(p_{j})_{j=1}^{\infty} \subseteq N\overline{\otimes}B$ such that 
\begin{itemize}
\item $\sum_{j}\tau\otimes \Tr(p_{j})<+\infty$,
\item $(N\overline{\otimes} B)\xi_{j}\perp (N\overline{\otimes} B)\xi_{k}$ if $j\ne k$,
\item $\ip{x\xi_{j},\xi_{j}}=\tau(xp_{j})$ for all $x\in N\overline{\otimes}B$,
\item $\cK=\bigvee_{j} \overline{(N\otimes B)\xi_{j}}.$
\end{itemize}
Applying \cite[Theorem IX.2.5 and Theorem IX.2.13]{TakesakiII}, for each $p\in [1,+\infty]$ we may view $L^{p}(M\overline{\otimes}B)$ as $\tau\otimes\Tr$-measurable operators affiliated with $M\overline{\otimes}B$. Note that the conditional expectation $\cE_{N\overline{\otimes}B}\colon M\overline{\otimes}B\to N\overline{\otimes}B$ extends to  contractions $L^{p}(M\overline{\otimes}B)\to L^{p}(N\overline{\otimes}B)$ for $p\in [1,+\infty]$ which we still denote by $\cE_{N\overline{\otimes}B}$. In particular, we can make sense of $\cE_{N\overline{\otimes}B}(\xi_{j}\xi_{j}^{*})$ as an element of $L^{1}(N\overline{\otimes}B)$. 
Note that by the third item we have:
    \[
        \tau\otimes \Tr(x\cE_{N\overline{\otimes}B}(\xi_{j}\xi_{j}^{*}))=\ip{x\xi_{j},\xi_{j}}=\tau(xp_{j})
    \]
for every $x\in N\overline{\otimes}B$. This implies that $\cE_{N\overline{\otimes}B}(\xi_{j}\xi_{j}^{*})=p_{j}$. 
We now prove a few preliminary claims.\\

\noindent\textbf{Claim 1:} \emph{For every $x\in M\overline{\otimes}B$ we have that $\cE_{N\overline{\otimes}B}(x \xi_{j}^{*})\in M\overline{\otimes}B$, and $|\cE_{N\overline{\otimes}B}(x\xi_{j}^{*})|^{2}\leq \|x\|_{\infty}^{2}p_{j}$.} Indeed, by the Kadison--Schwartz inequality
    \[
        |\cE_{N\overline{\otimes}B}(x\xi_{j}^{*})|^{2}\leq \cE_{N\overline{\otimes}B}(\xi_{j}x^{*}x\xi_{j}^{*})\leq \|x\|_{\infty}^{2}\cE_{N\overline{\otimes}B}(\xi_{j}\xi_{j}^{*})=\|x\|_{\infty}^{2}p_{j}.
    \]\\

\noindent\textbf{Claim 2:} \emph{If $E_{j}$ is the projection from $L^{2}(M\overline{\otimes} B)$ onto $L^{2}((N\overline{\otimes} B)\xi_{j})$, then $E_{j}(x)=\cE_{N\overline{\otimes}B}(x\xi_{j}^{*})\xi_{j}$.}
It suffices to show that $\ip{\cE_{N\overline{\otimes}B}(x\xi_{j}^{*})\xi_{j},b\xi_{j}}=\ip{x,b\xi_{j}}$ for all $b\in N\overline{\otimes}B$. Note that
\[\ip{\cE_{N\overline{\otimes}B}(x\xi_{j}^{*})\xi_{j},b\xi_{j}}=\tau(\xi_{j}^{*}b^{*}\cE_{N\overline{\otimes}B}(x\xi_{j}^{*})\xi_{j})=\tau(b^{*}\cE_{N\overline{\otimes}B}(x\xi_{j}^{*})\xi_{j}\xi_{j}^{*})=\tau(b^{*}\cE_{N\overline{\otimes}B}(x\xi_{j}^{*})p_{j}),\]
where in the last step we use that $\cE_{N\overline{\otimes}B}$ is trace-preserving and that $\cE_{N\overline{\otimes}B}(\xi_{j}\xi_{j}^{*})=p_{j}$.  Thus
    \[
        \ip{\cE_{N\overline{\otimes}B}(x\xi_{j}^{*})\xi_{j},b\xi_{j}}=\tau(x\xi_{j}^{*}p_{j}b^{*}) = \<x, bp_j\xi_j\>.
    \]
By choice of $\xi_{j},p_{j}$ we have 
    \[
        \|\xi_{j}\|_{2}^{2}=\tau\otimes \Tr(p_{j})= \tau\otimes \Tr(p_j p_j)=\ip{p_{j}\xi_{j},\xi_{j}}=\|p_{j}\xi_{j}\|_{2}^{2},
    \]
so that $p_{j}\xi_{j}=\xi_{j}$. Hence
\[\ip{\cE_{N\overline{\otimes}B}(x\xi_{j}^{*})\xi_{j},b\xi_{j}}=\<x, bp_j\xi_j\>=\ip{x,b\xi_{j}},\]
as required.\\

\noindent\textbf{Claim 3:} \emph{Suppose that $u_{n}\in \cU(Q)$ and $\|\cE_{N}(au_{n}b)\|_{2}\to 0$ as $n\to\infty$ for all $a,b\in M$. Then 
\[\lim_{n\to\infty}\sum_{j=1}^{\infty}\|\cE_{N\overline{\otimes}B}((u_{n}\otimes 1)x\xi_{j}^{*})\xi_{j}\|_{2}^{2}=0,\]
for every $x\in M\overline{\otimes}B$.} 
It follows from the proof of Claim~2 that $p_{j}\xi_{j}=\xi_{j}$ and that $\|\xi_{j}\|_{2}=(\tau\otimes \Tr)(p_{j})^{1/2}.$ Thus by Claim~1,
$\|\cE_{N\overline{\otimes}B}((u_{n}\otimes 1)x\xi_{j}^{*})\xi_{j}\|_{2}\leq \|x\|_{\infty}\|p_j\xi_{j}\|_{2}=\|x\|_{\infty}(\tau\otimes \Tr)(p_{j})^{1/2}$.
Since $\sum_{j} (\tau\otimes \Tr)(p_{j})<+\infty$, by dominated convergence, it suffices to show
\[\lim_{n\to\infty}\|\cE_{N\overline{\otimes}B}((u_{n}\otimes 1)x\xi_{j}^{*})\xi_{j}\|_{2}=0\]
for every $j\in \bN$. 

The fact that $\cE_{N\overline{\otimes}B}$ is $L^{2}$-$L^{2}$ contractive, the density of $M\otimes_{\textnormal{alg}}(B\cap L^{2}(B))$ inside $L^{2}(M\otimes B)$, the Kaplansky density theorem, and our assumptions imply that $\|\cE_{N\overline{\otimes}B}(a(u_{n}\otimes 1)\xi)\|_{2}\to 0$ as $n\to\infty$ for all $\xi\in L^{2}(M\otimes B)$ and $a\in M\overline{\otimes}B$. In particular, $\|\cE_{N\overline{\otimes}B}((u_{n}\otimes 1)x\xi_{j}^{*})\|_{2}\to 0$ as $n\to\infty$. Moreover, Claim~1 implies that $\|\cE_{N\overline{\otimes}B}((u_{n}\otimes 1)x\xi_{j}^{*})\|_{\infty}\leq \|x\|_{\infty}$, and so $\cE_{N\overline{\otimes}B}((u_{n}\otimes 1)x\xi_{j}^{*})\to 0$ as $n\to\infty$ in the strong operator topology. A fortiori, $\|\cE_{N\overline{\otimes}B}((u_{n}\otimes 1)x\xi_{j}^{*})\xi_{j}\|_{2}\to 0$ as $n\to\infty$, which proves Claim~3.\\

Having proven Claims 1-3, we complete the proof of the lemma. Since $Q\nprec N$, we may find unitaries $u_{n}\in \cU(Q)$ with $\|\cE_{N}(au_{n}b)\|_{2}\to 0$ as $n\to\infty$ for all $a,b\in M$. Let $x\in \cH$, and $E_{j}$ the orthogonal projection onto $L^{2}((N\otimes B)\xi_{j})$. By $Q\overline{\otimes}B$-invariance of $\cH$ we have that $(u_{n}\otimes 1)x\in \cH\subseteq \cK$ for all $n$. Thus by choice of $\xi_{j}$,
\[\|x\|_{2}^{2}=\|(u_{n}\otimes 1)x\|_{2}^{2}=\sum_{j=1}^{\infty}\|E_{j}((u_{n}\otimes 1)x)\|_{2}^{2}=\sum_{j=1}^{\infty}\|\cE_{N\overline{\otimes}B}((u_{n}\otimes 1)x\xi_{j}^{*})\xi_{j}\|_{2}^{2},\]
where in the last step we use Claim~2. By Claim~3, the right-hand side of this equality tends to zero as $n\to\infty$. Thus taking limits on both sides, we see that $x=0$. This proves the lemma.
\end{proof}

\begin{thm}\label{thm:quasi anticoarse stuff}
Let $M$ be a von Neumann algebra equipped with a faithful normal state $\varphi$. Suppose that $Q\leq M$ is with $\varphi$-expectation  and diffuse. 
Then $q^{1}\cN_{M}(Q)\subseteq L^{2,\textnormal{dyn}}_{\cross}(Q\leq M,\varphi)$.
\end{thm}

\begin{proof}
Set $\widetilde{Q}:=Q\rtimes_{\sigma_{t}}\bR$ and  $\widetilde{M}:=M\rtimes_{\sigma_{t}}\bR$. Let $\tau_{\varphi}$ be tracial weight on $\widetilde{M}$ induced by $\varphi$ and  let $(p_{n})_{n\in \bN} \subseteq \dom(\tau_\varphi|_{L_\varphi(\bR)})$ be an increasing sequence of projections converging to $1$ in the strong operator topology. Fix $a\in q^{1}\cN_{M}(Q)$ and $T\in \Hom_{\widetilde{Q}-\widetilde{Q}}(L^{2}(\widetilde{M},\tau_{\varphi}),L^{2}(\widetilde{Q},\tau_\varphi)\otimes L^2(\widetilde{Q},\tau_{\varphi}))$. We will first show that $T(p_{n}a)=0$ for every $n\in \bN$. Fix an $n\in \bN$.

Choose $a_{1},\cdots,a_{k}\in Q$ with
$Qa\subseteq \sum_{j=1}^{k}a_{j}Q.$
Then:
    \[
        p_{n}L_{\varphi}(\bR)Qa\widetilde{Q}\subseteq \sum_{j=1}^{k}p_{n}L_{\varphi}(\bR)a_{j}\widetilde{Q}=\sum_{j=1}^{k}p_{n}L_{\varphi}(\bR)p_{n}a_{j}\widetilde{Q}.
    \]
Given $x,y\in \widetilde{Q}$,
choose a bounded net $x_{\alpha}\in \Span(L_{\varphi}(\bR)Q)$ with $x_{\alpha}\to x$ in the strong operator topology. Then, using normality of the right action of $\widetilde{Q}$ we have:
    \[
        \|p_{n}xay-p_{n}x_{\alpha}ay\|_{\tau_\varphi}\leq \|ay\|\|p_{n}(x-x_{\alpha})\|_{\tau_\varphi}\to_{\alpha}0.
    \]
Viewing $p_n a, p_n a_j \in L^2(\widetilde{M}, \tau_\varphi)$, this implies that 
    \[
        (p_n \widetilde{Q} p_n)\cdot (p_n a)\cdot \widetilde{Q}\subseteq p_{n}\widetilde{Q}\cdot a\cdot \widetilde{Q}\subseteq \overline{\sum_{j=1}^{k}(p_{n}L_{\varphi}(\bR))\cdot p_{n}a_{j}\cdot \widetilde{Q} }.
    \]
Set $\xi=T(p_{n}a)$ and $\xi_{j}=T(p_{n}a_{j})$. 
By bimodularity of $T$, we have that $\xi=(p_{n}\otimes 1)\xi$. So applying $T$ to the above inclusion and taking linear combinations yields:
    \[
        \overline{\Span}((p_{n}\widetilde{Q}p_n\otimes 1)\cdot\xi\cdot(1\otimes \widetilde{Q}) \subseteq \overline{\sum_{j=1}^{k}\Span((p_{n}L_{\varphi}(\bR)\otimes 1)\cdot\xi_j\cdot (1\otimes\widetilde{Q}))},
    \]
where the closures are taken in $L^2(\widetilde{Q}\overline{\otimes}\widetilde{Q},\tau_\varphi\otimes \tau_\varphi)$. This further implies that for every integer $\ell\geq n$ we have
    \[
        \overline{\Span}((p_{n}\widetilde{Q}\otimes p_{\ell})\xi (1\otimes \widetilde{Q}))\subseteq \overline{\sum_{j=1}^{k}\Span((p_{n}L_{\varphi}(\bR)\otimes p_\ell)\cdot \xi_j \cdot(1\otimes \widetilde{Q}))}.
    \]
Since $p_{\ell}$ has finite trace, for each $\ell$ the right-hand side is a finite-dimensional module over $p_{n}L_{\varphi}(\bR)\overline{\otimes} \widetilde{Q}^{op}$, while the left-hand side is a submodule over $p_{n}\widetilde{Q}p_{n}\overline{\otimes} \widetilde{Q}^{op}$. By \cite[Lemma 2.5]{HUAsymptot} it follows that 
    \[
        p_{n}\widetilde{Q}p_{n} \nprec_{p_{n}\widetilde{M}p_{n}}p_{n}L_{\varphi}(\bR),
    \]
so Lemma~\ref{lem:not quite defn of intertwining?} implies that $(1\otimes p_{\ell})\xi=0$. Letting $\ell\to\infty$, we obtain $T(p_na)=\xi=0$, and thus $p_n a \in L^2_{\cross}(\widetilde{Q}\leq \widetilde{M},\tau_\varphi)$ for all $n\in \bN$.

Now, recall that $\widetilde{\varphi}= (\tau_\varphi)_{D_\varphi}$, where $D_\varphi$ is the operator affiliated with $L_\varphi(\bR)$ satisfying $D_\varphi^{it} = \lambda_\varphi(t)$ for all $t\in \bR$. For each $\varepsilon>0$, set
    \[
        h_\varepsilon:= \frac{D_\varphi}{1+\varepsilon D_\varphi} \in L_\varphi(\bR).
    \]
Using Lemma~\ref{lem:canonical_intertwiner_for_commuting_weights}, for any $f\in L^1(\bR)\cap L^2(\bR)$ we have 
    \[
        U_{\tau_{\varphi}, \widetilde{\varphi}}(\lambda_\varphi(f) a) =\lim_{\varepsilon\to 0}  \lambda_\varphi(f) a h_\varepsilon^{\frac12} = \lim_{\varepsilon\to 0} \lim_{n\to\infty}  p_n\lambda_\varphi(f) a h_\varepsilon^{\frac12} = \lim_{\varepsilon\to 0} \lim_{n\to\infty} \lambda_\varphi(f)\cdot p_n  a\cdot h_\varepsilon^{\frac12},
    \]
and this lies in $L^2_{\cross}(\widetilde{Q}\leq \widetilde{M},\tau_\varphi)$ since it is a norm closed $\widetilde{Q}$-bimodule. Thus $\lambda_\varphi(f) a \in L^2_{\cross}(\widetilde{Q}\leq \widetilde{M},\widetilde{\varphi})$ for all $f\in L^1(\bR)\cap L^2(\bR)$ by the discussion following Proposition~\ref{prop: exercise with bimod maps}, and therefore $a\in L^{2,\textnormal{dyn}}_{\cross}(Q\leq M, \varphi)$ by the density of such $f$ in $L^2(\bR)$. 
\end{proof}

As a corollary to the previous theorem, the notions of $\varphi$-entropy free inclusions and $\varphi$-Pinsker algebras can have their dependence on the state $\varphi$ loosened. Recall that for a faithful normal state $\varphi$ on $M$, $\Pi(M,\varphi)$ denotes the unique $\varphi$-Pinsker algebra in $M$.

\begin{cor}\label{cor:independence of Pinsker on the state}
Let $M$ be a von Neumann algebra equipped with faithful normal states $\varphi_j$ for $j=1,2$. 
    \begin{enumerate}[(a)]
    \item Let $Q$ be a diffuse subalgebra such that $Q\leq M$ is with $\varphi_j$-expectation for $j=1,2$. If $Q\leq M$ is a $\varphi_1$-entropy free inclusion, then it is also a $\varphi_{2}$-entropy free inclusion. \label{item: independence of state cor 1}
  
    \item If $\Pi(M,\varphi_{1})\cap \Pi(M,\varphi_{2})$ contains a diffuse subalgebra with $\varphi_j$-expectation in $M$ for $j=1,2$, then $\Pi(M,\varphi_{1})=\Pi(M,\varphi_{2})$. \label{item: independence of state cor 2}
    \end{enumerate}

\end{cor}

\begin{proof}
(\ref{item: independence of state cor 1}):
Set $P_{j}=\Pi(M,\varphi_{j})$, and let $u_t:=(D \varphi_1\colon D \varphi_2)_t$, $t\in \bR$, be the cocycle derivative of $\varphi_1$ with respect to $\varphi_2$ so that
    \[
        \sigma_t^{\varphi_1}(x) = u_t \sigma_t^{\varphi_2}(x) u_t^* \qquad t\in \bR,\ x\in M
    \]
(see \cite[Theorem VIII.3.3]{TakesakiII}). The invariance of $Q$ under $\sigma^{\varphi_j}$ for $j=1,2$ implies that $u_{t}\in \cN_{M}(Q)$ for all $t\in \bR$. By Theorem~\ref{thm:quasi anticoarse stuff} and Theorem~\ref{thm:dynamical_anticoarse_contained_in_Pinsker}, we thus have that $u_{t}\in \cU(P_{1})$ for all $t\in \bR$.

Now, let $\pi_{\varphi_2,\varphi_1}\colon M\rtimes_{\sigma^{\varphi_{1}}}\bR\to M\rtimes_{\sigma^{\varphi_{2}}}\bR$ be the $*$-isomorphism from (\ref{eqn:formula_isomorphism_continuous_cores}).
Then the above gives
    \[
        \pi_{\varphi_2,\varphi_1}(P_{1}\rtimes_{\sigma^{\varphi_{1}}}\bR)=P_{1}\rtimes_{\sigma^{\varphi_{2}}}\bR.
    \]
One consequence of Corollary~\ref{cor:independent of the proj} is that being entropy free is invariant under isomorphisms of inclusions of semifinite von Neumann algebras. 
Since $P_{1}\rtimes_{\sigma^{\varphi_{1}}}\bR\leq M\rtimes_{\sigma^{\varphi_{1}}}\bR$ is an entropy free inclusion, this implies that $P_{1}\rtimes_{\sigma^{\varphi_{2}}}\bR\leq M\rtimes_{\sigma^{\varphi_{2}}}\bR$ is an entropy free inclusion. Thus $P_{1}\leq M$ is a $\varphi_{2}$-entropy free inclusion. Since $Q\leq M$ and is with $\varphi_2$-expectation, it follows from Proposition~\ref{prop: nonpositive monotone} that $Q\leq M$ is a $\varphi_{2}$-entropy free inclusion.\\

\noindent(\ref{item: independence of state cor 2}): Let $Q$ be a diffuse subalgebra of $P_1\cap P_2$ with $\varphi_j$-expectation in $M$. Then $Q\leq M$ is an $\varphi_1$-entropy free inclusion by Proposition~\ref{prop: nonpositive monotone}. Arguing as in (\ref{item: independence of state cor 1}), we see that not only is $P_1\leq M$ a $\varphi_1$-entropy free inclusion, but also a $\varphi_2$-entropy free inclusion. Consequently, $P_1\subseteq P_2$ by Corollary~\ref{cor:unique_varphi-Pinsker_algebra}. Arguing symmetrically, we obtain $P_1=P_2$.
 \end{proof}

Motivated by the above corollary, we make the following definitions.

\begin{defn}\label{defn:sigma-entropy_free}
Let $M$ be a $\sigma$-finite von Neumann algebra. For $Q\leq M$ diffuse and with expectation, we say this inclusion is \textbf{$\sigma$-entropy free} if it is $\varphi$-entropy free for some faithful normal state $\varphi$ on $M$.
\end{defn}

Recall that $Q\leq M$ being $\varphi$-entropy free demands that the inclusion is with $\varphi$-expectation. Corollary~\ref{cor:independence of Pinsker on the state} implies the above definition is independent of the choice of state in the sense that a $\sigma$-entropy free inclusion $Q\leq M$ will be $\varphi$-entropy free for \emph{all} faithful normal states $\varphi$ satisfying that $Q\leq M$ is with $\varphi$-expectation.  Additionally, if $M$ is $\sigma$-finite and $Q\leq M$ is with any expectation, then it is with $\varphi$-expectation for at least one faithful normal state $\varphi$. Indeed, if $\psi$ is a faithful normal state on $M$ and $\cE_Q\colon M\to Q$ is a faithful normal conditional expectation, then this holds for $\varphi:=\psi|_Q\circ \cE_Q$.

\begin{defn}
Let $M$ be a $\sigma$-finite von Neumann algebra. For $P\leq M$ diffuse and with expectation, we say that $P$ is a \textbf{$\sigma$-Pinsker algebra} if $P=\Pi(M,\varphi)$ for some faithful normal state $\varphi$ on $M$.
\end{defn}

Corollary~\ref{cor:independence of Pinsker on the state} implies this definition is also independent of the choice of state in the sense that a $\sigma$-Pinsker algebra $P$ will equal $\Pi(M,\varphi)$ for \emph{all} faithful normal states $\varphi$ satisfying that $P\leq M$ is with $\varphi$-expectation. The assumption that $P$ is diffuse is needed in order to invoke the above corollary, but we note that $\varphi$-Pinsker algebras need not be diffuse in general. Indeed, Proposition~\ref{prop:amenable} implies $\Pi(B(\cH),\omega)=B(\cH)$ for any faithful normal state $\omega$ on $B(\cH)$.

In both of the above definitions, the prefix ``$\sigma$'' stems from the fact that this notion of being entropy free ultimately relies on using the modular automorphism group to pass to the continuous core. However, it should be stressed that while one uses the modular automorphism group to define this notion, it is ultimately independent of the choice of modular automorphism group. Additionally, as illustrated by Remark~\ref{rem: entropy just isn't what it used to be}, if $M$ is semifinite this differs from the notion of entropy free discussed in Section~\ref{sec:Pinsker for traces}. We list some general properties of  $\sigma$-Pinsker algebras and being $\sigma$-entropy free here.

\begin{cor}\label{cor: its over vince carter gif}
Let $M$ be a $\sigma$-finite von Neumann algebra.
\begin{enumerate}[(i)]
\item Suppose that $Q\leq M$ is $\sigma$-entropy free. Let $\Omega$ be the set of subalgebras $Q_{0}\leq M$ such that $Q_{0}\supseteq Q$ and $Q_{0}\leq M$ is $\sigma$-entropy free. Then $\Omega$ has a largest element $P$, which is a $\sigma$-Pinsker algebra. We call $P$ the \textbf{$\sigma$-Pinsker algebra} of $Q\leq M$, and we denote it by $\Pi^{\sigma}(Q\leq M)$. \label{item: extrinsic Pinsker}

\item \label{item: extrinsic diffuse absoprtion} Suppose $P\leq M$ is a $\sigma$-Pinsker algebra and that $Q\leq M$ is $\sigma$-entropy free inclusion. If $Q\cap P$ contains a diffuse subalgebra with expectation in $M$, then $Q\subseteq P$.

\item Suppose $Q_{1},Q_{2}\leq M$ are $\sigma$-entropy free and $Q_{1}\cap Q_{2}$ contains a diffuse subalgebra with expectation in $M$. Then for any $B\leq Q_{1}\vee Q_{2}$ such that $B$ is diffuse and with expectation in $M$, we have that $B\leq M$ is $\sigma$-entropy free. 
\label{item: extrinsic join lemma}
\end{enumerate} 
\end{cor}
\begin{proof}
\textbf{(\ref{item: extrinsic Pinsker}):} From the discussion following Definition~\ref{defn:sigma-entropy_free} we know that $Q\leq M$ is with $\varphi$-expectation for some faithful normal state $\varphi$ on $M$. We claim that $P:=\Pi(M,\varphi)$ is the largest element in $\Omega$. First, one has $P\in \Omega$ by definition. Next, let $\cE_Q\colon M\to Q$ be the $\varphi$-preserving faithful normal conditional expectation, and fix $Q_0\in \Omega$ with faithful normal conditional expectation $\cE_{Q_0}\colon M\to Q_0$. Then $\cE_Q|_{Q_0}\circ \cE_{Q_0}\colon M\to Q$ is another faithful normal conditional expectation, and $Q\leq M$ and $Q_0\leq M$ are with $\psi$-expectation where
    \[
        \psi:= \varphi|_Q\circ \cE_Q|_{Q_0}\circ \cE_{Q_0}.
    \]
Corollary~\ref{cor:independence of Pinsker on the state} implies $P=\Pi(M,\psi)$, and Corollary~\ref{cor:unique_varphi-Pinsker_algebra} implies $Q_0\subseteq P$. Hence $P$ is the largest element of $\Omega$.\\

\noindent \textbf{(\ref{item: extrinsic diffuse absoprtion}):} Let $N\subseteq Q\cap P$ be the diffuse subalgebra with expectation in $M$. Arguing as in (\ref{item: extrinsic Pinsker}), we can find faithful normal states $\psi_1,\psi_2$ on $M$ so that $N, \Pi^\sigma(Q\leq M) \leq M$ are with $\psi_1$-expectation and $N, P\leq M$ are with $\psi_2$-expectation. Then $\Pi^\sigma(Q\leq M)= \Pi(M,\psi_1)$ and $P=\Pi(M,\psi_2)$, and hence $\Pi^\sigma(Q\leq M) =P$ by Corollary~\ref{cor:independence of Pinsker on the state}. In particular, $Q\subseteq \Pi^\sigma(Q\leq M) = P$.\\

\noindent \textbf{(\ref{item: extrinsic join lemma}):} By (\ref{item: extrinsic diffuse absoprtion}) we have that $Q_{2}\leq \Pi^{\sigma}(Q_{1}\leq M)$, and so by (\ref{item: extrinsic Pinsker}) we have that $\Pi^{\sigma}(Q_{2}\subseteq M)\subseteq \Pi^{\sigma}(Q_{1}\subseteq M)$. Reversing the roles of $Q_1$ and $Q_2$ yields the equality of these $\sigma$-Pinsker algebras, so $Q_{1}\vee Q_{2}\leq \Pi^\sigma(Q_1\leq M)$. Since $B\subseteq \Pi^\sigma(Q_1\leq M)$, we can argue as in (\ref{item: extrinsic Pinsker}) to find a faithful normal state $\varphi$ on $M$ so that $B,\Pi^\sigma(Q_1\leq M)\leq M$ are with $\varphi$-expectation. Since then $\Pi^\sigma(Q_1\leq M)=\Pi(M,\varphi)$, It follows from Proposition~\ref{prop: nonpositive monotone} that $B\leq M$ is $\sigma$-entropy free.
\end{proof}

We close by showing that $\sigma$-entropy free inclusions are preserved by taking von Neumann algebras generated by another weakening of the normalizer.

\begin{defn}
Let $Q\leq M$ be an inclusion of von Neumann algebras with expectation. The \textbf{wq-normalizer} of $Q$ in $M$, denoted by $\cN^{wq}_M(Q)$, is the set of unitaries $u\in M$ such that $u^{*}Qu\cap Q$ contains a diffuse subalgebra with expectation in $M$.
\end{defn}

Note that in the above definition we do \emph{not} require that $u^*Qu\cap Q$ is itself with expectation in $M$. Also observe that $\cN_M(Q) \subseteq \cN^{wq}_M(Q)$ for any diffuse $Q\leq M$ with expectation because for any $u\in\cN_M(Q)$ one has that $u^* Q u \cap Q = Q$ is diffuse and with expectation in $M$.

\begin{cor}\label{cor: abosrbing wq normalizers}
Let $M$ be a $\sigma$-finite von Neumann algebra and let $Q\leq M$ be diffuse and with expectation. If $Q\leq M$ is $\sigma$-entropy free, then so is $\cN_{M}^{wq}(Q)''\leq M$.     
\end{cor}

\begin{proof}
Let $\varphi$ be a faithful normal state on $M$ such that $Q\leq M$ is with $\varphi$-expectation. It follows from the definition of the wq-normalizer that $\cN^{wq}_M(Q)$ is $\sigma^\varphi$-invariant, and hence $\cN_{M}^{wq}(Q)''\leq M$ is with expectation. Also observe that $Q\subseteq \cN_M(Q) \subseteq \cN_{M}^{wq}(Q)''$ by the discussion preceding the statement of the corollary, and consequently $\cN^{wq}_M(Q)''$ is diffuse by Lemma~\ref{lem:diffuseness_passed_upwards}.

Fix $u\in \cN^{wq}_{M}(Q)$, and choose $B\leq u^*Qu \cap Q$ which is diffuse and so that $B\leq M$ is with expectation. Then $B\leq Q$ is with expectation, and so $B\leq M$ is $\sigma$-entropy free by Corollary~\ref{cor: its over vince carter gif}.(\ref{item: extrinsic join lemma}) applied to $Q_{1}=Q_{2}=Q$.
Further, by Corollary~\ref{cor: its over vince carter gif}.(\ref{item: extrinsic diffuse absoprtion}) we have that $\Pi^{\sigma}(B\leq M)=\Pi^{\sigma}(Q\leq M)$. Similarly, $\Pi^{\sigma}(B\leq M)=\Pi^{\sigma}(uQu^{*}\leq M)$, and it is direct to see that  $\Pi^{\sigma}(uQu^{*}\leq M)=u\Pi^{\sigma}(Q\leq M)u^{*}$. Thus $u\in \cN_{M}(\Pi^{\sigma}(B\leq M))$, but Theorems~\ref{thm:dynamical_anticoarse_contained_in_Pinsker} and \ref{thm:quasi anticoarse stuff} imply $\Pi^{\sigma}(B\leq M))$ contains its own normalizer. Thus $\cN_{M}^{wq}(Q)\leq \Pi^{\sigma}(P\leq M)$. Since $\cN_{M}^{wq}(Q)''\leq M$ is with expectation by the first part of the proof, it follows from Corollary~\ref{cor: its over vince carter gif}.(\ref{item: extrinsic join lemma}), with $Q_{1}=Q_{2}=\Pi^{\sigma}(B\leq M)$ that $W^{*}(\cN_{M}^{wq}(Q))\leq M$ is a $\sigma$-entropy free inclusion. 
\end{proof}

\section{Pushing anticoarse spaces to the limit: ultrapowers}\label{sec:ultrapowers}

For a cofinal ultrafilter $\omega$ on a directed set $I$ and a von Neumann algebra $M$, denote
    \[
        \cI_{\omega}(M):=\{(x_{i})_{i\in I}\in \ell^{\infty}(I,M):\lim_{i\to\omega}x_{i}=0 \textnormal{ in the strong-$*$ topology}\},
    \]
    \[
        \cM^{\omega}(M):=\{(x_{i})_{i\in I}\in \ell^{\infty}(I,M):(x_{i})_{i\in I}\cI_{\omega}(M)+\cI_{\omega}(M)(x_{i})_{i\in I}\subseteq \cI_{\omega}(M)\}.
    \]
By \cite{OcneauActions, AndoHaagerup} the quotient $C^{*}$-algebra
    \[
        M^{\omega}:=\cM^{\omega}(M)/\cI_{\omega}(M)
    \]
is a von Neumann algebra, which we call the \emph{Ocneanu ultrapower of $M$}. For $(x_{i})_{i\in I}\in \cM^{\omega}(M)$ we use $(x_{i})_{i\to\omega}$ for its image in $M^{\omega}$. If $Q$ is a subalgebra of $M$ with a faithful normal conditional expectation $\cE_{Q}\colon M\to Q$, then $Q^{\omega}$ is a naturally a von Neumann subalgebra of $M^{\omega}$ and there is a natural conditional expectation given $\cE_{Q^{\omega}}$ given by
\[\cE_{Q^{\omega}}((x_{i})_{i\to\omega})=(\cE_{Q}(x_{i}))_{i\to\omega}\]
(see \cite[Section 2]{HIBicent} for details).
Applying this with $Q=\bC$, we see that if $\varphi$ is a faithful normal state on $M$, then the ultraproduct state $\varphi^{\omega}$ given by
\[\varphi^{\omega}((x_{i})_{i\to\omega})=\lim_{i\to\omega}\varphi(x_{i})\]
remains faithful. Moreover, the modular automorphism group of $\varphi^\omega$ is determined by that of $\varphi$:
\begin{align}\label{eqn:ultrapower_modular_automorphism_group}
        \sigma_t^{\varphi^\omega}( (x_i)_{i\to\omega}) = ( \sigma_t^{\varphi}(x_i))_{i\to \omega} \qquad \qquad t\in \bR.  
    \end{align}
(see \cite[Theorem 4.1]{AndoHaagerup}). Consequently, the diagonal embedding $M\leq M^\omega$ is with $\varphi^\omega$-expectation.

Let $M$ be a von Neumann algebra equipped with a faithful normal state $\varphi$. By \cite[Theorem 2.10]{MTUltra} (see also \cite[Theorem 4.1]{AndoHaagerup}), we have natural inclusions
    \[
        M\rtimes_{\sigma^\varphi} \bR \leq M^\omega \rtimes_{\sigma^{\varphi^\omega}} \bR \leq (M\rtimes_{\sigma^\varphi} \bR)^\omega.
    \]
Specifically, the first embedding restricted to $M$ is the diagonal embedding and restricted to $L_\varphi(\bR)$ maps $\lambda_\varphi(t)$ to $\lambda_{\varphi^\omega}(t)$ for all $t\in \bR$. Note that $M\leq M^\omega$ being with $\varphi^\omega$-expectation implies the first embedding above is with $\widetilde{\varphi^\omega}$-expectation and $\tau_{\varphi^\omega}$-expectation. In particular, $\tau_{\varphi^\omega}|_{M\rtimes_{\sigma^\varphi}\bR} = \tau_\varphi$. The second embedding arises from the natural embedding $\cI_{\omega}(M)\to \cI_{\omega}(M\rtimes_{\sigma^\varphi}\bR)$ along with the diagonal embedding $\lambda_{\varphi^\omega}(t) \to (\lambda_\varphi(t))_{i\to \omega}$.

\begin{prop}\label{prop: all too easy}
Let $M$ be a $\sigma$-finite von Neumann algebra and let $Q\leq M$ be diffuse and with expectation. Then $Q\leq M^{\omega}$ is   $\sigma$-entropy free inclusion  if and only if $Q\leq M$ is  an $\sigma$-entropy free inclusion. 
\end{prop}

\begin{proof}
Let $\psi$ be a normal state on $M$, and $\cE_{Q}\colon M\to Q$ a faithful, normal conditional expectation. Set $\varphi=\psi|_{Q}\circ \cE_{Q}$, so that $Q\leq M$ is with $\varphi$-expectation. Note that this along with (\ref{eqn:ultrapower_modular_automorphism_group}) implies $Q\leq M^\omega$ is with $\varphi^\omega$-expectation.  It suffices to show that $Q\leq M$ is $\varphi$-entropy free if and only if $Q\leq M^{\omega}$ is $\varphi^{\omega}$-entropy free. The ``if'' direction is immediate from Proposition~\ref{prop: nonpositive monotone} and the fact that $\varphi^\omega|_M=\varphi$. Conversely, suppose $Q\leq M$ is a $\varphi$-entropy free inclusion. Let $\tau_{\varphi^\omega}$ be the tracial weight induced by $\varphi^\omega$, which we noted above restricts to $\tau_\varphi$ on $M\rtimes_{\sigma^{\varphi}}\bR$. By Lemma \ref{lem: independ of the proj}, it suffices to show that $h_{\tau_{\varphi^\omega}}(p(Q\rtimes_{\sigma^\varphi} \bR)p:p(M^\omega\rtimes_{\sigma^{\varphi^\omega}}\bR) p)\leq 0$ for all $p\in \dom(\tau_\varphi|_{L_{\varphi}(\bR)})$. Fixing such a projection $p$, we may then identify $p(M\rtimes_{\sigma^\varphi}\bR)^\omega p$ as the tracial ultrapower of $p(M\rtimes_{\sigma^\varphi}\bR)p$. Consequently, \cite[Property 1 and Proposition 4.5]{Hayes2018} imply   
    \begin{align*}
       h_{\tau_{\varphi^\omega}}(p(Q\rtimes_{\sigma^\varphi} \bR)p:p(M^\omega\rtimes_{\sigma^{\varphi^\omega}}\bR)p)&\leq h_{\tau_{\varphi^\omega}}(p(Q\rtimes_{\sigma^\varphi} \bR)p:p(M\rtimes_{\sigma^\varphi}\bR)^\omega p)\\
       &= h_{\tau_{\varphi^\omega}}(p(Q\rtimes_{\sigma^\varphi} \bR)p:p(M\rtimes_{\sigma^\varphi} \bR)p)
    \end{align*}
and the last expression is nonpositive by assumption.
\end{proof}

As a first application, we show in the $\sigma$-finite setting that all ``relatively non-full'' subalgebras provide entropy zero inclusions.

\begin{prop}\label{prop: nonfull entropy free}
 Let $M$ be a $\sigma$-finite von Neumann algebra and let $Q\leq M$ be diffuse and with expectation. If $Q'\cap M^{\omega}$ is diffuse, then $Q\leq M$ is a $\sigma$-entropy free inclusion.
\end{prop}

\begin{proof}
 Since $Q\leq M$ is with expectation in $M$, we have that $Q'\cap M^{\omega} \leq M^\omega$ is with expectation. Indeed, $Q\leq M$ is with $\varphi$-expectation for some faithful normal state $\varphi$, and so (\ref{eqn:ultrapower_modular_automorphism_group}) implies $Q\leq M^\omega$ (and hence $Q'\cap M^\omega \leq M^\omega$) is with $\varphi^\omega$-expectation. Fix faithful normal conditional expectations $\cE_{Q}\colon M^{\omega}\to Q$ and $\cE_{Q'\cap M^{\omega}}\colon M^{\omega}\to Q'\cap M^{\omega}$. The $\sigma$-finiteness of $M$ passes to $M^\omega$, and in turn to $Q'\cap M^\omega$. Hence, by \cite[Theorem 11.1]{HSFlow} (see also \cite[Lemma 2.1]{HUAsymptot} and \cite{CSStateSpace}), we may choose a faithful normal state $\psi$ on $Q'\cap M^{\omega}$ so that $(Q'\cap M^{\omega})^{\psi}$ is diffuse. We will abuse notation to let $\psi$ also denote $\psi\circ \cE_{Q'\cap M^\omega}$, its extension to $M^\omega$. Fix a diffuse abelian subalgebra $A\leq (Q'\cap M^{\omega})^{\psi}$. Proposition~\ref{prop:amenable} implies that $A\leq M^\omega$ is a $\psi$-entropy free inclusion, and so Theorems~\ref{thm:dynamical_anticoarse_contained_in_Pinsker} and \ref{thm:quasi anticoarse stuff} imply
    \[
        \cN_{M^\omega}(A)'' \subseteq \Pi(A\leq M^\omega, \psi).
    \]  
Thus $\cN_{M^{\omega}}(A)''\leq M^{\omega}$ is $\sigma$-entropy free by Proposition~\ref{prop: nonpositive monotone}. By Corollary \ref{cor: its over vince carter gif}.(\ref{item: extrinsic join lemma}) with $B=Q$ and $Q_{1}=Q_{2}=\cN_{M^{\omega}}(A)''$ we have that $Q\leq M^{\omega}$ is $\sigma$-entropy free.   
Therefore $Q\leq M$ is $\sigma$-entropy free by Proposition~\ref{prop: all too easy}.
\end{proof}

We now list several strong ultrapower absorption properties for $\sigma$-Pinsker algebras.

\begin{cor}\label{cor: ultarmega solidity Pinsker}
Let $M$ be a $\sigma$-finite von Neumann algebra, and  let $P\leq M$ be a $\sigma$-Pinsker algebra.
\begin{enumerate}[(a)]
    \item Gamma absorption: if $Q\leq M$ is with expectation, $Q'\cap M^\omega$ is diffuse, and $Q\cap P$ is diffuse and with expectation in $M$, then $Q\subseteq P$. \label{item: gamma absorption}
    
    \item Generalized solidity phenomena: Suppose that $Q\leq M^{\omega}$ is diffuse and amenable  and that $\psi$ is a faithful, normal, state on $M^{\omega}$ which leaves $Q$ invariant. Suppose we are given subalgebras $Q_{\alpha}$ of $M^{\omega}$ for all ordinals $\alpha$ so that:
    \begin{itemize}
        \item $Q_{0}=Q$,
        \item for any successor ordinal, there is a faithful normal state $\varphi$ on $M^{\omega}$ which leaves $Q_{\alpha-1}$ invariant and so that $Q_{\alpha}=W^{*}(Q_{\alpha-1},X)$ where $X$ is globally $\sigma^{\psi}$-invariant and 
        \[X\subseteq L^{2,\textnormal{dyn}}(Q\leq M^{\omega},\varphi))\cup \cN^{wq}_{M^{\omega}}(Q_{\alpha-1}),\]
        (e.g. if $X\subseteq q^{1}\cN_{M^{\omega}}(Q_{\alpha-1})\cup \cN^{wq}_{M^{\omega}}(Q_{\alpha-1})$), 
        \item $Q_{\alpha}=\overline{\bigcup_{\beta<\alpha}Q_{\beta}}^{SOT}$, if $\alpha$ is a limit ordinal. 
    \end{itemize}
Then for any $\alpha$ if $N\leq Q_{\alpha}\cap M$ is with expectation in $M$, and if $N\cap P$ is diffuse, we must have that $N\subseteq P$.
\label{item: omega mega solidity}  
\end{enumerate}
\end{cor}

\begin{proof}
\textbf{(\ref{item: gamma absorption}):} By Proposition~\ref{prop: nonfull entropy free}, we have that $Q\leq M$ is a $\sigma$-entropy free inclusion (note that $Q$ is diffuse by Lemma~\ref{lem:diffuseness_passed_upwards}). By  Corollary~\ref{cor: its over vince carter gif}.(\ref{item: extrinsic diffuse absoprtion}), it follows that $Q\subseteq P$.\\

\noindent \textbf{(\ref{item: omega mega solidity}):} 
We first show that $Q_{\alpha}\leq M^{\omega}$ is an entropy free inclusion for every ordinal $\alpha$ by transfinite induction. For a successor ordinal, note that $Q_{\alpha-1}\leq Q_{\alpha}$ is with expectation and so diffuseness of $Q_{\alpha-1}$ implies that of $Q_{\alpha}$. Set 
\[X_{1}=X\cap L^{2,\textnormal{dyn}}_{\cross}(Q\leq M^{\omega},\varphi), \textnormal{ and } X_{2}=X\cap \cN_{M^{\omega}}^{wq}(Q_{\alpha-1}). \]
By definition, we have that $Q_{\alpha-1}\leq M^{\omega}$ is $\varphi$-entropy free, so that Theorem \ref{thm:dynamical_anticoarse_contained_in_Pinsker} and Proposition \ref{prop: nonpositive monotone} implies that $W^{*}(X,Q_{\alpha-1})\leq M^{\omega}$ is $\varphi$-entropy free, hence $\sigma$-entropy free. Note that $X_{2}\subseteq \cN_{M^{\omega}}^{wq}(Q_{\alpha-1})\subseteq \cN_{M^{\omega}}^{wq}(W^{*}(X_{1},Q_{\alpha-1})).$
Thus by Corollary \ref{cor: abosrbing wq normalizers} it follows that $Q_{\alpha}=W^{*}(X_{2},X_{1},Q_{\alpha-1})\leq M^{\omega}$ is $\sigma$-entropy free. 

To handle the case of a limit ordinal, first note that since each $Q_{\beta}$ is globally $\sigma^{\psi}$-invariant for any $\beta<\alpha$, we have that $Q_{\alpha}$
is globally $\sigma^{\psi}$-invariant. In particular, $Q_{\beta}\leq Q_{\alpha}$ is with expectation for any $\beta<\alpha$, so that $Q_{\alpha}$ is diffuse. Fix $\beta<\alpha$, and note that Corollary \ref{cor: its over vince carter gif} (\ref{item: extrinsic diffuse absoprtion}) implies that $\Pi^{\sigma}(Q_{\beta'}\leq M^{\omega})=\Pi^{\sigma}(Q_{\beta}\leq M^{\omega})$ for any $\beta\leq \beta'<\alpha$. Thus $Q_{\beta}\leq\Pi^{\sigma}(Q_{\beta}\leq M^{\omega})$ for any $\beta<\beta'<\alpha$, which implies that $Q_{\alpha}\leq \Pi^{\sigma}(Q_{\beta}\leq M^{\omega})$. Hence by Proposition \ref{prop: nonpositive monotone}, we have that $Q_{\alpha}\leq M^{\omega}$ is $\sigma$-entropy free.

So by transfinite induction, we have that $Q_{\alpha}\leq M^{\omega}$ is $\sigma$-entropy free for all $\alpha$. Since $N$ is with expectation in $M^{\omega}$, and $N\leq Q_{\alpha}\leq M^{\omega}$, it follows by Corollary \ref{cor: its over vince carter gif} (\ref{item: extrinsic join lemma}) with $Q_{1}=Q_{2}=Q_{\alpha}$ that $N\leq M^{\omega}$ is $\sigma$-entropy free. Hence, by Corollary \ref{cor: its over vince carter gif} (\ref{item: extrinsic diffuse absoprtion}) it follows that $N\leq M$ is $\sigma$-entropy free. Thus, Proposition \ref{prop: all too easy} implies that $N\subseteq P$.
\end{proof}

\section{Proofs of main results}\label{sec: mining the gold}

Theorems~\ref{introthm: ultrasolid III} and \ref{thm: the game is over and you should quit} follow from the results presented in this section. 
Throughout, we will use the $\sigma$-entropy free machinery developed in the previous sections, and so restrict to $\sigma$-finite von Neumann algebras whose continuous core has amenable Pinsker algebras.
We note that this is satisfied for a von Neumann algebra $N$ with continuous core isomorphic to $L(\bF_\infty)\overline{\otimes}B(\cH)$ for a separable Hilbert space $\cH$. First note that separability of $\cH$ is equivalent with $N$ being $\sigma$-finite. Indeed, if $N$ is $\sigma$-finite then so is $L(\bF_\infty)\overline{\otimes}B(\cH)$ (see the proof of Theorem~\ref{thm:dynamical_anticoarse_contained_in_Pinsker}). This in turn implies $B(\cH)$ is $\sigma$-finite and hence $\cH$ is separable. Conversely, if $\cH$ is separable, then $L(\bF_\infty)\overline{\otimes}B(\cH)$ is a separable von Neumann algebra. It follows that $N$ separable, and consequently $\sigma$-finite. In light of this, we will restrict our attention to von Neumann algebras with continuous cores of the form $L(\bF_\infty)\overline{\otimes}B(\ell^2)$.

Now, $L(\bF_{\infty})\overline{\otimes}B(\ell^{2})$ has amenable Pinsker algebras. Indeed, if $P\leq L(\bF_{\infty})\overline{\otimes}B(\ell^{2})$ is Pinsker, then for every finite trace projection $p\in P$ we have that $pPp\leq p[L(\bF_{\infty})\overline{\otimes}B(\ell^{2})]p$ is entropy zero. Hence, by the resolution of the Peterson--Thom conjecture by \cite{hayespt, belinschi2022strong, bordenave2023norm} we have that $pPp$ is amenable. As in the proof of Proposition \ref{prop: apparently free entropy is useless} below, this implies that $P$ is amenable. The following lemma allows us to more generally apply our analysis to any von Neumann algebra which embeds with expectation into a von Neumann algebra whose continuous core is isomorphic to $L(\bF_{\infty})\overline{\otimes}B(\cH)$.


\begin{lem}
Let $M$ be a $\sigma$-finite von Neumann algebra whose continuous core has amenable Pinsker algebras. If $N\leq M$ is with expectation, then the continuous core has amenable Pinsker algebras.   
\end{lem}
\begin{proof}
Fix a faithful, normal state $\psi$ on $M$ so that $N$ is with $\psi$-expectation. Then the continuous core of $M$ is isomorphic with $M\rtimes_{\sigma^{\psi}}\bR$, and $\tau_{\psi}$ is semifinite on $N\rtimes_{\sigma^{\psi}}\bR$ as $N$ is with $\psi$-expectation. Suppose that $Q\leq N\rtimes_{\sigma^{\psi}}\bR$ is entropy free, then by the proof of Proposition \ref{prop: nonpositive monotone} we have that $Q\leq M\rtimes_{\sigma^{\psi}}\bR$. Our assumptions on $M$ thus imply that $Q$ is amenable.
\end{proof}

\begin{prop}\label{prop: apparently free entropy is useless}
Let $M$ be a $\sigma$-finite von Neumann algebra von Neumann algebra  whose continuous core has amenable Pinsker algebras. If $Q\leq M$ is diffuse and with expectation, then $Q\leq M$ is $\sigma$-entropy free if and only if $Q$ is amenable. In particular, $M$ has the Peterson--Thom property.
\end{prop}
\begin{proof}
The ``if'' direction follows from Proposition~\ref{prop:amenable}.  Suppose that $Q\leq M$ is $\sigma$-entropy free, then by Proposition \ref{prop: nonpositive monotone} there is some faithful normal state $\varphi$ on $M$ so that $Q\leq M$ is $\varphi$-entropy free.
Let $\tau_\varphi$ be the tracial weight on $M\rtimes_{\sigma^\varphi}\bR$ induced by $\varphi$. Let $(p_{n})_{n\in \bN}\subseteq \dom(\tau_\varphi|_{L_\varphi(\bR)})$ be an increasing sequence projections converging to $1$ in the strong operator topology. Then, by definition, we have 
    \[
        h_{\tau_\varphi}(p_{n}(Q\rtimes_{\sigma^{\varphi}}\bR)p_{n}:p_{n}(M\rtimes_{\sigma^{\varphi}}\bR)p_{n})\leq 0.
    \]
Since $p_{n}(M\rtimes_{\sigma^{\varphi}}\bR)p_{n}
$ has amenable Pinsker algebras, the Pinsker algebra containing $p_{n}(Q\rtimes_{\sigma^{\varphi}}\bR)p_{n}$ is amenable, and hence $p_{n}(Q\rtimes_{\sigma^{\varphi}}\bR)p_{n}$ is amenable.
Letting $p_{n}\to 1$, we see that $Q\rtimes_{\sigma^{\varphi}}\bR$ is amenable, and hence $Q$ is amenable since there is a (non-normal) conditional expectation from $Q\rtimes_{\sigma^{\varphi}}\bR$ to $Q$ (see \cite[Proposition 2.6]{ADAmenableCorr}). The ``in particular" part follows from Corollary \ref{cor: its over vince carter gif}.
\end{proof}

Theorem \ref{thm: the game is over and you should quit} follows from the next result.
 
\begin{cor}\label{cor:ultramegasolidity}
Let $M$ be a $\sigma$-finite von Neumann algebra  whose continuous core has amenable Pinsker algebras.
Then:

\begin{enumerate}[(a)]
\item \label{item: ultramega solid for examlpes} $M$ is ultramega solid in the following sense. Suppose $Q\leq M^{\omega}$ is diffuse, amenable, and $\sigma^\psi$-invariant for a faithful normal state $\psi$ on $M^\omega$, and that we have subalgebras $Q_{\alpha}\leq M^{\omega}$ for every ordinal $\alpha$ so that:
    \begin{itemize}
        \item $Q_{0}=Q$,
        \item for a successor ordinal $\alpha$ we have $Q_{\alpha}=W^{*}(X,Q_{\alpha-1})$ where
        \[X\subseteq W^{*}(L^{2,\textnormal{dyn}}_{\cross}(Q\leq M^{\omega},\psi))\cup \cN_{M}^{wq}(Q_{\alpha-1})\]
       is globally $\sigma^{\psi}$-invariant,
        \item for a limit ordinal $\alpha$, we have $Q_{\alpha}=\overline{\bigcup_{\beta<\alpha}Q_{\beta}}^{SOT}$. 
    \end{itemize}
Then for every $\alpha$, any $B\leq Q_{\alpha}\cap M$ with expectation in $M$ is amenable.

\item In particular, suppose that $P$ is a maximal element of the set of subalgebras of $M$ which are amenable and with expectation in $M$.
Suppose $Q\leq M^{\omega}$ is diffuse, amenable, and $\sigma^\psi$-invariant for a faithful normal state $\psi$ on $M^\omega$, and that we have subalgebras $Q_{\alpha}\leq M^{\omega}$ for every ordinal $\alpha$ so that:
    \begin{itemize}
        \item $Q_{0}=Q$,
        \item for a successor ordinal $\alpha$ we have $Q_{\alpha}=W^{*}(X,Q_{\alpha-1})$ where
        \[X\subseteq W^{*}(L^{2,\textnormal{dyn}}(Q_{\alpha}\leq M^{\omega},\psi))\cup \cN_{M^{\omega}}^{wq}Q(_{\alpha-1})\]
        is globally $\sigma^{\psi}$-invariant,
        \item for a limit ordinal $\alpha$, we have $Q_{\alpha}=\overline{\bigcup_{\beta<\alpha}Q_{\beta}}^{SOT}$.  
            \end{itemize}
If there exists an $\alpha$ such that $Q_{\alpha}\cap P$ is diffuse, then $Q_{\beta}\cap M\subseteq P$ for every $\beta$. In particular, any $N\leq Q\cap M$ with expectation in $M$ is amenable. 
\label{item:ultramegasolid absorb for examples}
\end{enumerate}

\end{cor}

\begin{proof}
\textbf{(\ref{item: ultramega solid for examlpes}):}  As in the proof of Corollary \ref{cor: ultarmega solidity Pinsker}, we have that $Q_{\alpha}\leq M^{\omega}$ is $\sigma$-entropy free for every $\alpha$, and the same argument as in proof of Corollary \ref{cor: ultarmega solidity Pinsker} implies that $B\leq M$ is $\sigma$-entropy free. Hence we may apply Proposition \ref{prop: apparently free entropy is useless}.\\

\noindent\textbf{(\ref{item:ultramegasolid absorb for examples}):} This follows from Corollary \ref{cor: ultarmega solidity Pinsker} and Proposition \ref{prop: apparently free entropy is useless}.   
\end{proof}

Using $M^{\omega}$ for the Ocneau ultrapower (see \cite{OcneauActions, AndoHaagerup}), we say that a von Neumann algebra $M$ is \textbf{ultrastrongly solid} if for any cofinal ultrafilter $\omega$ on a directed set, and any $Q\leq M^{\omega}$ which is diffuse, amenable, and with expectation, we have that every $N\leq \cN_{M^{\omega}}(Q)''\cap M$ with expectation in $M$ is amenable. 

As a special case of the previous corollary, we have the following ultrastrong solidity result by taking $X=\cN_{M^{\omega}}(Q)$. Note that this implies Theorem~\ref{introthm: ultrasolid III}.

\begin{cor}\label{cor: weak ultramega solid}
Let $M$ be a $\sigma$-finite von Neumann algebra whose continuous core has amenable Pinsker algebras.
\begin{enumerate}[(a)]
\item If $Q\leq M^{\omega}$ is diffuse, amenable, and with expectation, and if $X\subseteq q^{1}\cN_{M^{\omega}}(Q)\cup \cN^{wq}_{M^{\omega}}(Q)$ is such that $W^{*}(X,Q)$ is with expectation in $M^{\omega}$,
then any $B\leq W^{*}(X,Q)\cap M$ with expectation in $M$ is amenable.

\item In particular, suppose that $P$ is a maximal element of the set of subalgebras of $M$ which are amenable and with expectation in $M$.  If $Q\leq M^{\omega}$ is diffuse, amenable, and with expectation, if $X\subseteq q^{1}\cN_{M^{\omega}}(Q)\cup \cN^{wq}_{M^{\omega}}(Q)$ is such that $W^{*}(X,Q)$ is with expectation in $M^{\omega}$, and if $W^{*}(X,Q)\cap M$ contains a diffuse subalgebra of $P$ which is with expectation in $M$, then any $B\leq W^{*}(X,Q)\cap M$ with expectation in $M$ is contained in $P$.
\end{enumerate}
\end{cor}
\begin{proof}
This follows from Corollary~\ref{cor:ultramegasolidity} with $\alpha=1$, after using Theorem~\ref{thm:quasi anticoarse stuff}.
\end{proof}

In \cite{HRSAsy} Houdayer--Raum defined a von Neumann algebra $M$ to be \emph{ultrasolid} if for any free ultrafilter $\omega$ on a cofinal directed set, and for every $Q\leq M$ with expectation such that $Q'\cap M^{\omega}$ is diffuse, one has that $Q$ is amenable. We prove here that this is indeed weaker than ultrastrongly solid.

\begin{prop}\label{prop: defns work}
 An ultrastrongly solid von Neumann algebra is ultrasolid.    
\end{prop}

\begin{proof}
Assume that $M$ is ultrastrongly solid.
Suppose that $Q\leq M$ is with expectation and $Q'\cap M^{\omega}$ is diffuse. Arguing as in Proposition \ref{prop: nonfull entropy free}, there is a state $\psi$ on $M^{\omega}$ so that $(Q'\cap M^{\omega})^{\psi}$ is diffuse. Fix a diffuse, abelian $A\leq(Q'\cap M^{\omega})^{\psi}$. Then $A\leq M^{\omega}$ is with expectation, and $Q\leq \cN_{M^{\omega}}(A)''\cap M$. Since $Q\leq M$ is with expectation, ultrastrong solidity implies that $Q$ is amenable. 
\end{proof}

We close this section by showing that for our main examples of interest, we have a strong version of Gamma absorption for subalgebras which are maximal among amenable subalgebras with expectation.

\begin{cor}
Let $M$ be a $\sigma$-finite von Neumann algebra  whose continuous core has amenable Pinsker algebras. Suppose that $P$ is a maximal element of the set of subalgebras of $M$ which are amenable and with expectation in $M$. Then $P$ satisfies Gamma absorption: if $Q\leq M$ is with expectation, and if both $Q\cap P$ and $Q'\cap M^{\omega}$ are diffuse, then $Q\leq P$.
\end{cor}

\begin{proof}
By Proposition \ref{prop: apparently free entropy is useless} we have that $P$ is $\sigma$-Pinsker. Since $M$ is ultrastrongly solid, we see that $Q$ is amenable. Hence, the result follows from Corollary~\ref{cor: its over vince carter gif}.(\ref{item: extrinsic diffuse absoprtion}).
\end{proof}

\section{Applications to free Araki--Woods factors}\label{sec:FAWF}

The free Araki--Woods factors were introduced by Shlyakhtenko in \cite{DIMAFAW} as von Neumann algebras $\Gamma(U)''$ associated to orthogonal representations $U\colon\bR \curvearrowright \cH_{\bR}$ of the real numbers on real Hilbert spaces. They are generated by self-adjoint operators $s(\xi)$, $\xi\in \bH_\bR$, and they come equipped with a natural state called the \emph{free quasi-free state} $\varphi_U$ whose modular automorphism group is determined by
    \begin{align}\label{eqn:FAWF_mod_aut}
        \sigma_t^{\varphi_U}(s(\xi)) = s(U_t \xi) \qquad \qquad t\in \bR,\ \xi\in \cH_{\bR}.
    \end{align}
For the trivial representation, $U=\text{id}_{\cH_{\bR}}$, this construction yields free group factors $\Gamma(\text{id}_{\cH_{\bR}})''\cong L(\bF_{\dim(\cH_{\bR})})$ (or $L(\bZ)$ if $\dim(\cH_{\bR})=1$), and otherwise $\Gamma(U)''$ is a type $\mathrm{III}$ factor (see \cite[Corollary 6.11]{DIMAFAW}). Consequently, $\Gamma(U)''$ for $U$ non-trivial are regarded as purely infinite analogs of the free group factors, especially from the perspective of free probability (see \cite{DimaAVal,DimaFreeFish,BrentTranspo}). We refer the reader to \cite[Section 2]{DIMAFAW} for further details, but for our purposes it will suffice to know the following facts:
    \begin{enumerate}[(I)]
        \item For a pair of orthogonal representations $U,V$ of $\bR$ one has
            \[
                \left(\Gamma( U\oplus V)'', \varphi_{U\oplus V}\right) \cong (\Gamma(U)'',\varphi_U)* (\Gamma(V)'',\varphi_V).
            \]\label{fact:FAWF_free_product_decomp}

        \item For an orthogonal representation $U\colon \bR\curvearrowright \cH_{\bR}$ with $\dim(\cH_{\bR})\geq 2$, $\Gamma(U)''$ is a non-amenable factor.\label{fact:FAWF_non-amenable_factor}
    \end{enumerate}
The first item follows from \cite[Theorem 2.11]{DIMAFAW}. For the second, the factoriality follows from a combination of \cite[Corollary 6.11]{DIMAFAW} and $\Gamma(\text{id}_{\cH_{\bR}})'' \cong L(\bF_{\dim(\cH_{\bR})})$. To see the non-amenability, first recall that every orthogonal representation decomposes as a direct sum $U=U_{ap}\oplus U_{wm}$ of an almost periodic (i.e. compact) representation and a weakly mixing representation. Thus by (\ref{fact:FAWF_free_product_decomp}) we have
    \[
        \left(\Gamma(U)'',\varphi_U \right) \cong (\Gamma(U_{ap})'', \varphi_{U_{ap}}) * (\Gamma(U_{wm})'', \varphi_{U_{wm}}),
    \]
and since each factor in the free product is with expectation it suffices to show at least one of them is non-amenable. If $U_{ap}$ is  over a space of dimension at least two, then \cite[Theorem 6.1]{DIMAFAW} implies $\Gamma(U_{ap})''$ is a full and diffuse factor, and hence non-amenable. Otherwise, $\dim(\cH_{\bR})\geq 2$ implies $U_{wm}$ must be over a non-zero (and hence infinite) dimensional space $\cK_{\bR}$. If $A$ is the infinitesimal generator of  the complexification of $U_{wm}$ acting on $\cK_{\bR}\otimes_{\bR} \bC$ (i.e. $\exp(it A) = U_t\otimes 1$ for all $t\in \bR$), then this representation being weakly mixing implies that the spectral measure of $A$ is diffuse. Consequently, we can find a non-trivial decomposition $U_{wm} = V_1\oplus V_2$ so that $(\Gamma(U_{wm})'', \varphi_{U_{wm}})$ decomposes as a free product of infinite dimensional algebras. Non-amenability then follows from \cite[Remark 4.2]{UedaTypeIIIfreeproduct}.

In the following results, we denote the Lebesgue measure on $\bR$ by $m$. We also denote by $A_U$ the infinitesimal generator of the complexification $U_{\bC}$ of $U$ acting on $\cH_{\bR}\otimes_{\bR} \bC$, and write $A_U\preceq \mu$ or $A_U \perp \mu$ when the spectral measure of $A_U$ is absolutely continuous or singular, respectively, with respect to a Borel measure $\mu$ on $\bR$. 
Recall from \cite{PopaWeakInter} that if $(M,\tau)$ is a tracial von Neumann algebra, then $P\leq M$ is \emph{coarsely embedded} if $L^{2}(M,\tau)\ominus L^{2}(P,\tau)$ as a $P$-$P$ bimodule embeds into an infinite direct sum of the coarse bimodule.

\begin{prop}\label{prop: lets go FAW ourselves}
Let $U\colon \bR\actson \cH_{\bR}$ be an orthogonal representation.
\begin{enumerate}[(a)]
    \item If $\cH_{\bR}$ is separable and $A_U \preceq m$, then there is a von Neumann algebra $N$ and a faithful normal state $\psi$ on $N$ so that $(\Gamma(U)'',\varphi_{U})*(N,\psi)$ has continuous core isomorphic to $L(\bF_{\infty})\overline{\otimes}B(\ell^2)$. \label{item: freely complemented inside nice continuous core 1}
    
    \item If $U$ is cyclic and $A_U\preceq m+\delta_0$, then there is a von Neumann algebra $N$ and a faithful normal state $\psi$ on $N$ so that $(\Gamma(U)'',\varphi_{U})*(N,\psi)$ has continuous core isomorphic to $L(\bF_{\infty})\overline{\otimes}B(\ell^2)$. \label{item: freely complemented inside nice continuous core 2}

    \item Suppose $A_U\perp m$ and $\dim(\cH_{\bR})\geq 2$. Let $N$ be any von Neumann algebra so that any maximal amenable subalgebra of a finite corner of its continuous core is coarsely embedded.
    Then $\Gamma(U)''$ does not embed with expectation into $N$. In particular, this applies when the continuous core of $N$ is isomorphic to $L(\bF_{\infty})\overline{\otimes}B(\ell^{2})$.
    \label{item: nice embeddings ain't gonna happen}
\end{enumerate}
\end{prop}

\begin{proof}
\textbf{(\ref{item: freely complemented inside nice continuous core 1}):} Our assumptions imply that there is a orthogonal representation $V$ of $\bR$ so that the complexification $V_{\bC}$ is such that $(U\oplus V)_{\bC}$ is isomorphic with $\lambda^{\oplus \bN}$, where $\lambda\colon \bR \actson L^2(\bR,m)$ is the (unitary) left regular representation. 
We have
    \[
         (\Gamma(U)'',\varphi_U)* (\Gamma(V)'',\varphi_V) \cong \left(\Gamma( U\oplus V)'', \varphi_{U\oplus V}\right),
    \]
by  (\ref{fact:FAWF_free_product_decomp}), and the continuous core of the right-hand side is isomorphic to $L(\bF_{\infty})\overline{\otimes}B(\ell^2)$ by \cite[Theorem 5.2]{DImaFullIII} (see also \cite[Theorem 4.8]{DimaFreeAmalgam}).\\

\noindent\textbf{(\ref{item: freely complemented inside nice continuous core 2}):} Our assumptions imply that there is a orthogonal rep $V$ of $\bR$ so that the complexification of $V_{\bC}$ is such that $(U\oplus V)_{\bC}$ is isomorphic with the direct sum of the trivial and left regular representations of $\bR$. 
We now argue as in  (\ref{item: freely complemented inside nice continuous core 1}), using \cite[Proposition 7.5]{HSV19}.\\

\noindent\textbf{(\ref{item: nice embeddings ain't gonna happen}):} 
Suppose, for contradiction, that $\Gamma(U)''$ embeds with expectation into $N$. Let $\psi$ be the extension of $\varphi_U$ to $N$ obtained by precomposing with a faithful normal conditional expectation.
First, let us show that $\overline{\Span\{\lambda_{\psi}(f)a:f\in L(\bR),a\in \cH_{\bC}\}}$ is in the  anticoarse space of $L_{\psi}(\bR)$ inside $N\rtimes_{\sigma^{\psi}}\bR$ with respect to $\tau_{\psi}$. 
This is the same proof as in \cite[Corollary 4.18]{Hayes2018} but we repeat it here. Throughout the proof, we use the identification
    \[
        \cH_{\bC} \cong \overline{\{s(\xi)+is(\eta)\colon \xi,\eta \in \cH_{\bR}\}}^{\|\cdot\|_{\psi}}. 
    \]
Hence we can view $L^{2}(\bR)\otimes \cH_{\bC}$ as an $L_{\psi}(\bR)$-subbimodule of $L^{2}(\Gamma(U)''\rtimes_{\sigma^{\psi}}\bR,\widetilde{\psi})$. We use this subbimodule structure to define a unitary representation $\pi$ of $\bR$ via
    \[
        \pi(t)\zeta=\lambda_{\psi}(t)\cdot \zeta\cdot \lambda_{\psi}(-t).
    \]
Using (\ref{eqn:FAWF_mod_aut}), this representation is isomorphic with $\id\otimes U_{\bC}$ when restricted to the invariant subspace $L^{2}(\bR)\otimes \cH_{\bC}$. By assumption, this implies that $\pi$ restricted to this invariant subspace is disjoint from the left regular representation. In particular, we see that $L^{2}(\bR)\otimes \cH_{\bC}\subseteq L^{2}_{\cross}(L_{\psi}(\bR)\leq N\rtimes_{\sigma^{\psi}}\bR,\widetilde{\psi})$, and arguing as in the last paragraph of the proof of Theorem~\ref{thm:quasi anticoarse stuff} this implies that 
    \[
        \cK:=\overline{\Span\{\lambda_{\psi}(f)s(\xi):f\in L^{1}(\bR)\cap L^{2}(\bR),\, \xi\in \cH_{\bR}\}}^{\|\cdot\|_{\psi}}\subseteq L^{2}_{\cross}(L_{\psi}(\bR)\leq N\rtimes_{\sigma^{\psi}}\bR,\tau_{\psi}).
    \]
Fix a projection $p\in \dom(\tau_\psi|_{L_{\psi}(\bR)})$ and let $Q_{p}$ be any maximal amenable subalgebra of $ p(N\rtimes_{\sigma^{\psi}}\bR)p$ containing $L_{\psi}(\bR)$. It follows from Proposition~\ref{prop:compressions_varphi_anticoarse_space} that 
    \[
        p\cdot\cK\cdot p\subseteq L^{2}_{\cross}(pL_{\psi}(\bR)\leq p(N\rtimes_{\sigma^{\psi}}\bR )p,\tau_{\psi}).
    \]
By assumption, $L^{2}(p(N\rtimes_{\sigma^{\psi}}\bR )p,\tau_{\psi})\ominus L^{2}(Q_{p},\tau_{\psi})$ embeds into an infinite direct sum of the coarse, and so we must have
    \[
        L^{2}_{\cross}(pL_{\psi}(\bR)\leq p(N\rtimes_{\sigma^{\psi}}\bR )p,\tau_{\psi})\subseteq L^{2}(Q_{p},\tau_{\psi}).
    \]
Thus it follows that $p\cdot \cK\cdot p\subseteq L^{2}(Q_{p},\tau_\psi)$. In particular, $W^{*}(\lambda_{\psi}(f)p s(\xi) p:f\in L^{1}(\bR)\cap L^{2}(\bR),\, \xi\in \cH_{\bR})$ is amenable. 

Now, let $(p_{n})_{n\in \bN} \subset \dom(\tau_\psi|_{L_\psi(\bR)})$ be an increasing sequence of projections converging to $1$ in the strong operator topology. Note that 
    \[
        \Gamma(U)''\rtimes_{\sigma^{\varphi_{U}}}\bR= \overline{\bigcup_{n}W^{*}(\lambda_{\psi}(f) p_n s(\xi) p_n :f\in L^{1}(\bR)\cap L^{2}(\bR),\, \xi\in \cH_{\bR})}.
    \]
Since the closure of an increasing union of injective von Neumann algebras is injective, the right-hand side of the above equation is injective. Hence \cite{Connes} implies that $\Gamma(U)''\rtimes_{\sigma^{\varphi_{U}}}\bR$ is amenable. As in the proof of Proposition~\ref{prop: apparently free entropy is useless}, this implies $\Gamma(U)''$ is amenable which, by (\ref{fact:FAWF_non-amenable_factor}),  contradicts our assumption that $\dim_{\bR}(\cH_{\bR}))\geq 2$.
\end{proof}


\begin{cor}\label{cor:FAWF_non-embedding}
Let $U\colon \bR\actson \cH_{\bR}$ be an orthogonal representation and let $U=U_a\oplus U_s$ be the decomposition over $\cH_{\bR}=\cH_{a}\oplus \cH_{s}$ satisfying $A_{U_a}\preceq m$ and $A_{U_s}\perp m$. Assume that $\dim_{\bR}(\cH_{s})\geq 2$.
\begin{enumerate}[(i)]
    \item If $V$ is an orthogonal representation of $\bR$ on a separable Hilbert space so that $A_{V}\preceq m$, then $\Gamma(U)''$ does not embed into $\Gamma(V)''$ with expectation. \label{item: no embed for you 1}
    \item If $V$ is a cyclic orthogonal representation of $\bR$ so that $A_{V}\preceq m+\delta_{0}$, then $\Gamma(U)''$ does not embed into $\Gamma(V)''$ with expectation. \label{item: no embed for you 2}
\end{enumerate}
\end{cor}
\begin{proof}
Note $\Gamma(U_{s})''$ embeds into $\Gamma(U)''$ with expectation. So we may, and will, assume that $U=U_{s}$. Then the first claim follows from Proposition~\ref{prop: lets go FAW ourselves}.(\ref{item: freely complemented inside nice continuous core 1}) and (\ref{item: nice embeddings ain't gonna happen}), and the second claim follows from Proposition~\ref{prop: lets go FAW ourselves} (\ref{item: freely complemented inside nice continuous core 2}) and (\ref{item: nice embeddings ain't gonna happen}).
\end{proof}

\begin{remark}
The non-isomorphism results implied by the non-embedding of $\Gamma(U)''$ into $\Gamma(V)''$ are also new. A theorem of Shlyakhtenko \cite[Theorem 4.4]{DImaFullIII} combined with \cite[Proposition 7.5]{HSV19},  proves non-isomorphism in the case that all convolution powers of the spectral measure of $U$ are singular with respect to Lebesgue measures. Shlyakhtenko's result was generalized in \cite[Corollary 4.18]{Hayes2018}, and when combined with \cite[Proposition 7.5]{HSV19} one obtains that in either item of the above theorem $\Gamma(U)''\not\cong \Gamma(V)''$ when $\cH_{a}=0$. The above theorem improves this by allowing $\cH_{a}\neq 0$ and achieving the stronger conclusion that embedding with expectation is also not possible. If the spectral measure of $U$ has an atomic part which is not $\delta_{0}$, then the fact that $\Gamma(U)''$ and $\Gamma(V)''$ are not isomorphic is handled by \cite[Theorem A]{HSV19}.
This does not imply our result, since we can allow $A_{U_s}$ to have an atomless spectral measure. From \cite[Corollary C]{HSV19}, it follows that if the atomic part of the spectral measure of $A_{U}$ is concentrated at $\{0\}$ and if $\dim(\ker(A_{U}))\leq 1$, then all of the centralizers of $\Gamma(U)''$ are amenable. This implies, for instance, that $\Gamma(U)''$ cannot embed into $\Gamma(W)''$ with expectation, if $W$ is an almost periodic orthogonal representation of $\bR$ (see e.g. \cite[Remark 5.2]{HSV19}). The reader is invited to compare the dichotomy  ``singular/absolutely continuous'' in Corollary \ref{cor:FAWF_non-embedding} with the dichotomy ``atomic/atomless'' appearing in the results from \cite{HSV19}.
\end{remark}

\begin{remark}
For $U$ and $V$ as in Corollary~\ref{cor:FAWF_non-embedding}, it also follows that the q-deformed Araki--Woods factor $\Gamma_q(U)''$ (see \cite{HiaiQAWF, KSM23}) does not embed with expectation into $\Gamma(V)''$ for any $-1<q<1$. Indeed, we first note the proof of Proposition~\ref{prop: lets go FAW ourselves}.(\ref{item: nice embeddings ain't gonna happen}) can be applied to $\Gamma_q(U_s)''$ since non-amenability holds by \cite[Theorem 2.5]{KSM23} and the analogue of (\ref{eqn:FAWF_mod_aut}) is \cite[(1.3)]{HiaiQAWF}. Then, even though $\Gamma_q(U_s)''$ need not be freely complemented in $\Gamma_q(U)''$, it is still $\sigma^{\varphi_U}$-invariant and therefore with expectation, and so the claim holds by arguing as in Corollary~\ref{cor:FAWF_non-embedding}.
\end{remark}

\appendix

\section{Peterson-Thom property for type III factors, by Stefaan Vaes}\label{Vaes}

By \cite{hayespt, belinschi2022strong, bordenave2023norm}, we know that the free group factors $P = L(\mathbb{F}_n)$, $2 \leq n \leq +\infty$, satisfy the Peterson-Thom conjecture: for all amenable von Neumann subalgebras $Q_1 \subset P$ and $Q_2 \subset P$ with diffuse intersection $Q_1 \cap Q_2$, the von Neumann algebra $Q_1 \vee Q_2$ generated by $Q_1$ and $Q_2$ remains amenable. By an application of Zorn's lemma, this property is equivalent to saying that any diffuse amenable von Neumann subalgebra $Q \subset P$ is contained in a unique maximal amenable von Neumann subalgebra $Q_1 \subset P$.

Given a von Neumann algebra $M$, we denote by $\AmEx(M)$ the collection of all its amenable von Neumann subalgebras with expectation. Recall from Section~\ref{sec: PT property} that $M$ has the Peterson--Thom property if for every diffuse $Q\in \AmEx(M)$, the von Neumann algebra 
    \[
        \bigvee \bigl\{Q_1 \bigm| Q \subset Q_1 \subset M \;\;\text{and}\;\; Q_1 \in \AmEx(M) \bigr\} 
    \]
still belongs to $\AmEx(M)$.


\begin{thm}\label{thm.main}
Let $M$ be a type III$_1$ factor with separable predual and continuous core $\core(M)$. If for some nonzero finite projection $p \in \core(M)$, the II$_1$ factor $p \core(M) p$ has the Peterson--Thom property, then $M$ itself has the Peterson--Thom property.

Also, if $\vphi$ is a faithful normal state on $M$, there is a unique amenable, globally $(\si^\vphi_t)_{t \in \R}$-invariant von Neumann subalgebra $Q \subset M$ that contains all amenable, globally $(\si^\vphi_t)_{t \in \R}$-invariant von Neumann subalgebras of $M$.
\end{thm}
\begin{proof}
Throughout the proof, we repeatedly use the following facts.
\begin{enumerate}
\item \cite[Theorem IX.4.2]{TakesakiII}. If $\vphi$ is a faithful normal state on $M$ and $Q \subset M$ is a von Neumann subalgebra, then $Q$ is globally invariant under $(\si_t^\vphi)_{t \in \R}$ if and only if there exists a faithful normal $\vphi$-preserving conditional expectation of $M$ onto $Q$.
\item \cite[Proposition IX.4.3]{TakesakiII}. If $Q \subset M$ is a von Neumann subalgebra such that $Q' \cap M = \cZ(Q)$, there is at most one faithful normal conditional expectation of $M$ onto $Q$.
\item Since $\core(M)$ is a II$_\infty$ factor that admits a trace scaling action of $\R$, all nonzero finite corners of $\core(M)$ are isomorphic and hence, all satisfy the Peterson-Thom property.
\end{enumerate}

\noindent{\bf Step 1.} Assume that $Q \in \AmEx(M)$ and that $A \subset Q$ is a diffuse abelian von Neumann subalgebra with expectation. Then $Q \vee (A' \cap M)$ belongs to $\AmEx(M)$.

Choose faithful normal conditional expectations $E_1 : Q \to A$ and $E_2 : M \to Q$. Define $E = E_1 \circ E_2$. Choose a faithful normal state $\vphi_0$ on $A$ and write $\vphi = \vphi_0 \circ E$. It follows that $(\si_t^\vphi)_{t \in \R}$ globally preserves $Q$ and $A$, and hence also $A' \cap M$ and $Q \vee (A' \cap M)$, which is therefore a von Neumann subalgebra of $M$ with expectation.

We realize $\core(M) = M \rtimes_{\si^\vphi} \R$ and choose an increasing sequence of finite projections $p_n \in \core(M)$ that belong to the group von Neumann algebra $L(\R)$ in this crossed product, such that $p_n \to 1$ strongly. Since $A$ is abelian, $\vphi_0$ is tracial on $A$ and thus $A$ commutes with $L(\R)$.

Fix $n$. Write $P = Q \rtimes_{\si^\vphi} \R$ and $B = (A' \cap M) \rtimes_{\si^\vphi} \R$. Since $p_n B p_n$ commutes with the diffuse abelian $A p_n$, it follows from Lemma \ref{lem.solid} below that $p_n B p_n$ is amenable. Since $A p_n \subset p_n P p_n \cap p_n B p_n$ and since $p_n \core(M) p_n$ has the Peterson-Thom property, we conclude that $T_n := p_n P p_n \vee p_n B p_n$ is amenable for every $n$. Fix $k$. Then $(p_k T_n p_k)_{n \geq k}$ is an increasing sequence of amenable von Neumann subalgebras of $p_k \core(M) p_k$, generating $p_k (P \vee B) p_k$. So, $p_k (P \vee B) p_k$ is amenable for every $k$. It follows that $P \vee B$ is amenable. By construction,
$$P \vee B = (Q \vee (A' \cap M)) \rtimes_{\si^\vphi} \R \; ,$$
so that $Q \vee (A' \cap M)$ is amenable. This concludes the proof of step~1.\\

\noindent{\bf Step 2.} Let $Q \in \AmEx(M)$ be diffuse. Then the set
\begin{equation}\label{eq.set-S}
\cS := \{Q_1 \mid Q \subset Q_1 \subset M \;\;\text{and}\;\; Q_1 \in \AmEx(M) \;\}
\end{equation}
ordered by inclusion admits a maximal element that contains $Q \vee (Q' \cap M)$.

Since $Q$ is diffuse, by \cite[Theorem 11.1]{HS90}, we can choose a diffuse abelian von Neumann subalgebra $A \subset Q$ with expectation. Define $Q_0 = Q \vee (A' \cap M)$. By Step 1, $Q_0 \subset M$ belongs to $\AmEx(M)$. By construction, $Q_0' \cap M = \cZ(Q_0)$ and $Q \vee (Q' \cap M) \subset Q_0$.

By Zorn's lemma, it thus suffices to prove that for every increasing net $(Q_i)_{i \in I}$ of von Neumann subalgebras $Q_i \subset M$ that belong to $\AmEx(M)$ and contain $Q_0$, also $\bigvee_{i \in I} Q_i$ belongs to $\AmEx(M)$. Choose faithful normal conditional expectations $E_i : M \to Q_i$ and $E_0 : M \to Q_0$. Since $Q_0' \cap M = \cZ(Q_0)$, $E_0$ is the unique faithful normal conditional expectation $M \to Q_0$. Therefore, $E_0 \circ E_i = E_0$ for all $i \in I$. Choose a faithful normal state $\vphi_0$ on $Q_0$ and define $\vphi = \vphi_0 \circ E_0$. Then $\vphi = \vphi \circ E_i$ for all $i \in I$. So, $(\si_t^\vphi)_{t \in \R}$ globally preserves $Q_i$ for every $i$ and thus also globally preserves $\bigvee_{i \in I} Q_i$, which is therefore with expectation. As a direct limit of amenable von Neumann algebras, $\bigvee_{i \in I} Q_i$ is amenable. This concludes the proof of step~2.\\

\noindent{\bf Proof of the Peterson-Thom property.} Define $\cS$ as in \eqref{eq.set-S}. By step~2, we can choose a maximal element $Q_0 \in \cS$ that contains $Q \vee (Q' \cap M)$. Fix any $Q_1 \in \cS$. It suffices to prove that $Q_1 \subset Q_0$.

Since $Q$ is diffuse, again by \cite[Theorem 11.1]{HS90}, we can choose a diffuse abelian von Neumann subalgebra $A \subset Q$ with expectation. Define $Q_2 = Q_1 \vee (A' \cap M)$. By step~1, $Q_2 \subset M$ belongs to $\AmEx(M)$. Write $S = Q \vee (Q' \cap M)$. Since $Q \subset M$ is with expectation, we can choose a faithful normal state $\psi$ on $M$ such that $(\si_t^\psi)_{t \in \R}$ globally preserves $Q$. Then $(\si_t^\psi)_{t \in \R}$ also globally preserves $Q' \cap M$ and thus $S$, which is therefore with expectation. Choose a faithful normal conditional expectation $E_S : M \to S$. Since $S' \cap M = \cZ(S)$, this $E_S$ is the unique faithful normal conditional expectation of $M$ onto $S$. Choose a faithful normal state $\vphi_S$ on $S$ and define $\vphi = \vphi_S \circ E_S$.

Choose faithful normal conditional expectations $E_0 : M \to Q_0$ and $E_2 : M \to Q_2$. By uniqueness of $E_S$, we get that $E_S \circ E_0 = E_S = E_S \circ E_2$, so that $\vphi \circ E_0 = \vphi = \vphi \circ E_2$. It follows that $(\si_t^\vphi)_{t \in \R}$ globally preserves $Q_0$ and $Q_2$. Then $(\si_t^\vphi)_{t \in \R}$ also globally preserves $Q_0 \vee Q_2$, so that $Q_0 \vee Q_2 \subset M$ is with expectation.

We realize $\core(M) = M \rtimes_{\si^\vphi} \R$ and choose an increasing sequence of finite projections $p_n \in \core(M)$ that belong to the group von Neumann algebra $L(\R)$ in this crossed product, such that $p_n \to 1$ strongly. Then $p_n (Q_0 \rtimes_{\si^\vphi} \R) p_n$ and $p_n (Q_2 \rtimes_{\si^\vphi} \R) p_n$ are amenable von Neumann subalgebras of $p_n \core(M) p_n$ whose intersection contains $p_n (S \rtimes_{\si^\vphi} \R) p_n$, which is diffuse. Since $p_n \core(M) p_n$ has the Peterson-Thom property, the von Neumann algebra generated by both remains amenable. Taking $n \to \infty$ and reasoning as in the proof of step~1, it follows that $(Q_0 \vee Q_2) \rtimes_{\si^\vphi} \R$ is amenable.

So, $Q_0 \vee Q_2$ is amenable. We have seen above that $Q_0 \vee Q_2 \subset M$ is with expectation. So, $Q_0 \vee Q_2 \in \cS$. Since $Q_0$ is a maximal element of $\cS$, it follows that $Q_2 \subset Q_0$. So also $Q_1 \subset Q_0$.\\

\noindent{\bf Proof of the final statement.} We realize $\core(M)$ as $M \rtimes_{\si^\vphi} \R$. Let $(Q_i)_{i \in I}$ be a family of amenable, globally $(\si^\vphi_t)_{t \in \R}$-invariant von Neumann subalgebras of $M$. Take projections $p_n \in L(\R)$ as above and note that $p_n (Q_i \rtimes_{\si^\vphi} \R) p_n$ is a family of amenable von Neumann subalgebras of $p_n \core(M) p_n$ with diffuse intersection $L(\R) p_n$. Since $p_n \core(M) p_n$ has the Peterson-Thom property, the same reasoning as above gives us that $\bigvee_{i \in I} Q_i$ is amenable.
\end{proof}

\begin{lemma}\label{lem.solid}
A II$_1$ factor $P$ satisfying the Peterson-Thom property is solid in the sense of \cite{OzawaSolidActa}: for every diffuse von Neumann subalgebra $A \subset P$, the relative commutant $A' \cap P$ is amenable.
\end{lemma}
\begin{proof}
Since we may replace $A$ by a diffuse abelian von Neumann subalgebra, we may assume that $A$ is abelian. Choose a family $A_i \subset A' \cap P$, $i \in I$, of abelian von Neumann subalgebras such that $\bigvee_{i \in I} A_i = A' \cap P$. For every $i \in I$, write $B_i = A \vee A_i$ and note that $B_i$ is abelian, hence amenable. Repeatedly applying the Peterson-Thom property, it follows that $\bigvee_{i \in \cF} B_i$ is amenable for every finite subset $\cF \subset I$. Then also $\bigvee_{i \in I} B_i$ is amenable. A fortiori, $A' \cap P$ is amenable.
\end{proof}


\end{document}